\documentclass[11pt, reqno]{amsart}
\usepackage[T1]{fontenc}
\usepackage[latin1]{inputenc}
\usepackage[english]{babel}
\usepackage{amsthm}
\usepackage{amssymb,amsmath,bbm,dsfont}
\usepackage{enumitem}
\usepackage[all]{xy}
\usepackage{mathrsfs}
\usepackage[left=0.0cm,right=0.0cm,top=1.7cm,bottom=1.7cm]{geometry}
\usepackage{hyperref}
\usepackage{verbatim}

\theoremstyle{plain}
\newtheorem{theorem}{Theorem}[section]

\newtheorem{lemma}[theorem]{Lemma}
\newtheorem{lem}[theorem]{Lemma}
\newtheorem{proposition}[theorem]{Proposition}
\newtheorem{prop}[theorem]{Proposition}
\newtheorem{corollary}[theorem]{Corollary}

\newtheorem{example}[theorem]{Example}

\theoremstyle{definition}
\newtheorem{definition}[theorem]{Definition}

\theoremstyle{remark}
\newtheorem{remark}[theorem]{Remark}

\theoremstyle{plain}

\def\Z{{\bf Z}}
\def\C{{\bf C}}

\def\O{{\mathcal{O}}}
\def\H{{H}}

\def\A{{\bf A}}

\def\epsilon{\varepsilon}

\def\G{\mathbf{G}}
\def\H{\mathbf{H}}

\def\unipu{{N}}
\def\GSp{{{\rm GSp}}}
\def\Sp{{\rm Sp}}
\def\SL{{\rm SL}}
\def\GU{{{\rm GU}}}

\def\GSpin{{{\rm GSpin}}}
\def\GO{{{\rm GO}}}

\def\GSO{{{\rm GSO}}}
\def\GL{{\rm GL}}
\def\Gm{\mathbb{G}_{\rm m}}

\def\matrix#1#2#3#4{{\big(\begin{smallmatrix}#1&#2\\ #3&#4\end{smallmatrix}\big)}}

\usepackage[usenames,dvipsnames]{color}

\title{Spherical Shalika models on $\mathrm{PGU}_{2,2}$ and the theta correspondence for $(\mathrm{PGSp}_4,\mathrm{PGU}_{2,2})$}

\author{Antonio Cauchi}
\address{Antonio Cauchi\newline University College Dublin\\ School of Mathematics and Statistics\\ Science Centre - South, Belfield, Dublin 4, Dublin, Ireland}
\email{antonio.cauchi@ucd.ie}

\author{Armando Gutierrez Terradillos}
\address{Armando Gutierrez Terradillos\newline 
Morningside Center of Mathematics, Chinese Academy of Sciences \\  
No. 55, Zhongguancun East Road, Haidian District, Beijing 100190, China}
\email{armando@amss.ac.cn}

\subjclass[2020]{11F27, 11F66, 11F70}
\keywords{Automorphic  $L$-functions, theta correspondence, Casselman-Shalika formula.}
\makeatletter
\def\@tocline#1#2#3#4#5#6#7{\relax
  \ifnum #1>\c@tocdepth 
  \else
    \par \addpenalty\@secpenalty\addvspace{#2}%
    \begingroup \hyphenpenalty\@M
    \@ifempty{#4}{%
      \@tempdima\csname r@tocindent\number#1\endcsname\relax
    }{%
      \@tempdima#4\relax
    }%
    \parindent\z@ \leftskip#3\relax \advance\leftskip\@tempdima\relax
    \rightskip\@pnumwidth plus4em \parfillskip-\@pnumwidth
    #5\leavevmode\hskip-\@tempdima
      \ifcase #1
       \or\or \hskip 1em \or \hskip 2em \else \hskip 3em \fi%
      #6\nobreak\relax
    \dotfill\hbox to\@pnumwidth{\@tocpagenum{#7}}\par
    \nobreak
    \endgroup
  \fi}
\makeatother
\usepackage{hyperref}

\begin{document}

\begin{abstract}
We study Shalika models for generic unramified representations of ${\rm PGU}_{2,2}$ over non-archimedean local fields  of characteristic zero. We show that they are unique up to constant by means of the theta correspondence for $(\mathrm{P}\GSp_4,{\rm PGU}_{2,2})$. We then prove a Casselman--Shalika formula which relates the values of spherical Shalika functionals on ${\rm PGU}_{2,2}$ to the values of finite dimensional complex representations of the dual group of $\mathrm{P}\GSp_4$. 
\end{abstract} 

\maketitle
\selectlanguage{english}

\setcounter{tocdepth}{1}
\tableofcontents

\section{Introduction}

Shalika models have been shown to shed profound insights into the conjectural framework that connects the theory of automorphic forms to the Langlands program. They first appeared in the study of the analytic properties of the partial exterior square $L$-function on $\GL_{2n}$. A celebrated theorem of Jacquet and Shalika (\textit{cf}. \cite{JacquetShalika}) asserts that the partial exterior square $L$-function of a cuspidal automorphic representation $\pi$ of ${\rm GL}_{2n}$ has a simple pole at $s=1$ if and only if the Shalika model does not vanish on the space of $\pi$. 
A similar statement for  the quasi-split similitude unitary group $\GU_{2,2}$, which is closely connected to $\GL_4$, was studied by Morimoto and Furusawa in \cite{FurusawaMorimoto}, where the existence of Shalika models is proved to be equivalent to being globally generic and having the partial exterior square $L$-function with a pole at $s=1$.  In the subsequent work \cite{morimoto}, using the exceptional isomorphism $\mathrm{PGU}_{2,2}\simeq \mathrm{PGSO}_{4,2}$, Morimoto characterized the existence of Shalika models for $\mathrm{PGU}_{2,2}$ in terms of the theta correspondence for the pair $(\GSp_4,\mathrm{GSO}_{4,2})$, proving in particular that a cuspidal automorphic representation on $\mathrm{PGU}_{2,2}$ has a non-trivial Shalika model if and only if it comes from a generic cuspidal automorphic representation of $\mathrm{PGSp}_4$. \\ 

It is natural to investigate whether a local analogue of this statement is true and to determine the dimension of the space of Shalika functionals of a given representation.  When $F$ is a non-archimedean local field of characteristic zero, a combination of some of the results of Morimoto in \cite{morimoto} shows that generic unramified representations of ${\rm PGU}_{2,2}(F)$ always admit a non-trivial Shalika functional and that are contained in the image of the theta lift from $\mathrm{PGSp}_4(F)$. 
In this article, we give an alternative proof of Morimoto's result by means of Mackey theory for infinite dimensional representations. We then prove the uniqueness of Shalika functionals for any generic unramified representation $\Pi$ of ${\rm PGU}_{2,2}(F)$. With this in mind, we establish an explicit formula for the value of the Shalika functional at the translates of a spherical vector of $\Pi$ in terms of finite dimensional complex representations of the dual group of $\mathrm{PGSp}_4$. We conclude by showing how this formula has applications to the study of $L$-functions. Namely, we use it to show how a Rankin-Selberg integral on ${\rm PGU}_{2,2}$ represents an $L$-function on ${\rm PGSp}_{4}$.

\subsection{On local Shalika functionals on ${\rm PGU}_{2,2}$}

Fix a non-archimedean local field $F$  of characteristic zero and let $E/F$ be a quadratic \'etale $F$-algebra which defines ${\rm PGU}_{2,2}$. Denote by $S$ the Shalika subgroup defined by \eqref{shalgrp} and let $\chi_{S}$ be the character on $S$ defined in \S \ref{ss:Shalikamodels}. A Shalika functional for an irreducible admissible representation $\Pi$ of ${\rm PGU}_{2,2}(F)$ is a non-zero linear functional $\Lambda_{\mathcal{S}}:\Pi \to \C$, such that for $s \in S(F)$ and $v$ in the space of $\Pi$, we have
\[\Lambda_{\mathcal{S}}(\Pi(s)v) = \chi_{S}(s)\Lambda_{\mathcal{S}}(v).\]
The existence of Shalika functionals for a given representation $\Pi$ detects certain functorial properties of $\Pi$. When $E = F \times F$, the group ${\rm PGU}_{2,2}(F)$ is isomorphic to ${\rm PGL}_4(F)$ and this question has been explored in many works in the literature. For instance, when $\Pi$ is a generic and unramified representation of ${\rm PGL}_4(F)$, the results of \cite{JacquetRallis} and \cite{Sakellaridis} show that the dimension of the space of Shalika functionals on $\Pi$ is at most one and is equal to one if and only if  $\Pi$ is a functorial lift from ${\rm PGSp}_4 (F)$. Note that the latter is equivalent to saying that a representative of the Frobenius conjugacy class of $\Pi$ is in the image of ${\rm Sp}_4(\C)$. Moreover, using the exceptional isomorphism between ${\rm PGL}_4(F)$ and ${\rm PGSO}_{3,3}(F)$, these conditions are equivalent to say that $\Pi$ is in the image of the local theta correspondence from ${\rm PGSp}_4(F)$. \\

The first main result of the manuscript aims to give an analogous characterization of the space of Shalika functionals of generic unramified representations of ${\rm PGU}_{2,2}(F)$ when $E$ is the unique unramified quadratic field extension of $F$, which we assume to be from now on. We show the following. 
\begin{theorem}[Theorem \ref{MainSummarizingS3}]\label{theorem1}
    Let $\Pi$ be a generic unramified irreducible representation of ${\rm PGU}_{2,2}(F)$. Then $\Pi$ has a unique non-trivial Shalika functional up to constant and it is the small theta lift of a generic unramified representation of $\mathrm{PGSp}_4^+(F)$.
\end{theorem}

We now describe the strategy and main ingredients of the proof of Theorem \ref{theorem1}. We recall that an admissible representation of ${\rm PGU}_{2,2}(F)$ is unramified (or spherical) if it has a vector which is invariant under the action of a good hyperspecial maximal compact subgroup of ${\rm PGU}_{2,2}(F)$. Each unramified irreducible representation $\Pi$ of ${\rm PGU}_{2,2}(F)$ can be realized as the (unique) irreducible unramified subquotient of an unramified principal series for ${\rm PGU}_{2,2}(F)$. If $\Pi$ is also generic, \textit{i.e.} it has a non-trivial Whittaker model, a result of Li (\cite[Theorem 2.7]{LiUPSpadic}) implies that $\Pi$ is an irreducible principal series for ${\rm PGU}_{2,2}(F)$. 

The first step in the proof of Theorem \ref{theorem1} involves the use of Mackey theory, which studies the restriction of principal series of ${\rm GU}_{2,2}(F)$ to the Shalika subgroup $S(F)$. As a result, we get an existence and uniqueness criterion for Shalika functionals of unramified principal series on ${\rm PGU}_{2,2}(F)$ solely in terms of their induced data. Before stating the result, we introduce some notation. If we let $T_{{\rm GU}_{2,2}}$ be the maximal diagonal torus of ${\rm GU}_{2,2}$, an unramified character $\xi$ on $T_{{\rm GU}_{2,2}}(F) \simeq E^\times \times E^\times \times F^\times$ is $\xi_1 \otimes \xi_2 \otimes \xi_0$, with $\xi_1, \xi_2$, and $\xi_0$ unramified characters on $E^\times$,  $E^\times$, and  $F^\times$ respectively. We then denote by $I_{{\rm GU}_{2,2}}(\xi)$ the normalized induction from the upper-triangular Borel of ${\rm GU}_{2,2}(F)$ attached to $\xi$. 

\begin{theorem}[Theorem \ref{FinalMackey}] \label{subtheorem1}
 Let $I_{{\rm GU}_{2,2}}(\xi)$ be an irreducible unramified principal series of ${{\rm GU}_{2,2}}(F)$, with trivial central character. Then $I_{{\rm GU}_{2,2}}(\xi)$ has at least  one non-trivial Shalika functional. Moreover, $I_{{\rm GU}_{2,2}}(\xi)$ admits a unique non-trivial Shalika functional up to constant if none of the following holds: \begin{enumerate}
    \item $\xi_1(a)\xi_0(ad) = \xi_2^{-1}(a)$, $\forall a,d \in F^\times$;
     \item $\xi_1(a)\xi_0(ad) = \xi_2^{-1}(d)$, $\forall a,d \in F^\times$.
\end{enumerate}
\end{theorem}

We would like to remark that the question on the existence of non-trivial Shalika functionals for any irreducible  generic unramified representation of ${\rm PGU}_{2,2}(F)$ was previously answered by Morimoto in \cite{morimoto}. Precisely, Morimoto's  proof follows from the computation of the local theta correspondence for the dual pair $(\mathrm{PGSp}_4^+(F),{\rm PGU}_{2,2}(F))$ and of the twisted Jacquet modules of the Weil representation (see \cite[Theorems 6.9, 6.21, and 6.23]{morimoto}). However, the method used in this manuscript offers a direct explanation of the reason why principal series for ${\rm PGU}_{2,2}(F)$ admit at least one non-trivial Shalika functional.

Theorem \ref{subtheorem1} asserts the existence of non-trivial Shalika functionals for any generic unramified irreducible representation of ${\rm PGU}_{2,2}(F)$, but not their uniqueness in the desired generality. 
In order to show that the space of Shalika functionals is of multiplicity one, we rephrase the problem in terms of the theta correspondence.  We can do so because Theorem \ref{subtheorem1} and the explicit description of \cite[Theorem 6.21]{morimoto} of the local theta correspondence for the dual pair $({\rm PGSp}_{4}^+(F),{\rm PGU}_{2,2}(F))$ let us show in Theorem \ref{Shalikaimpliesimagetheta} that any generic unramified representation of ${\rm PGU}_{2,2}(F)$ is the small theta lift of a generic unramified representation $\sigma$ of ${\rm PGSp}_{4}^+(F)$. Interestingly, the principal series which satisfy  \textit{(1)} or \textit{(2)} of Theorem \ref{subtheorem1} are the theta lifts of the generic ${\rm PGSp}_{4}^+(F)$-factor of an unramified principal series $I_{{\rm GSp}_{4}}(\chi)$ of ${\rm PGSp}_{4}(F)$ with the property that \[I_{{\rm GSp}_{4}}(\chi) \otimes \chi_{E/F} \simeq I_{{\rm GSp}_{4}}(\chi),\]
where $\chi_{E/F}$ is the quadratic character of $F^\times$ corresponding to $E$ via local class field theory.
For such principal series, Theorem \ref{subtheorem1} does not directly show  the uniqueness of their Shalika functionals.   Nevertheless, one can exploit certain properties of the theta correspondence to prove the uniqueness of Shalika functionals for any irreducible generic unramified representation of  ${\rm PGU}_{2,2}(F)$. Indeed, Proposition \ref{SUniqueIfTIrr} shows how the calculation in \cite{morimoto} of certain twisted Jacquet modules of the Weil representation can be used to relate the uniqueness of Shalika functionals on ${\rm PGU}_{2,2}(F)$ to the one of Whittaker functionals on ${\rm PGSp}_{4}^+(F)$. This reduces the proof of the uniqueness of Shalika functionals in Theorem \ref{theorem1} to determining whether the big theta lift to ${\rm PGSp}_{4}^+(F)$ of a generic unramified representation of ${\rm PGU}_{2,2}(F)$ is irreducible or not, which is done in Lemma \ref{lemThetatheta}. This completes the proof of Theorem \ref{theorem1}.
 
\subsection{A Casselman--Shalika formula for Shalika functionals on ${\rm PGU}_{2,2}$}\label{section1.2}

Explicit formulas for linear functionals of unramified representations have been proved to be an essential tool to study automorphic $L$-functions. One of the foundational results in the theory is provided by the work of Casselman and Shalika \cite{CasselmanShalika}, who first computed a formula for the Whittaker functional associated to unramified principal series of unramified $p$-adic groups. The significance of that paper extends beyond its final result; the authors introduced a series of ideas that have been widely used to address similar problems including those discussed in this manuscript.

The Casselman--Shalika method provides a way to compute explicit formulas for linear functionals which are unique. Sakellaridis in \cite{Sakellaridis}, using the work of Hironaka \cite{Hironaka}, proposed a variant of the method which has a wider range of applications. The framework introduced in \cite{Sakellaridis} together with Theorem \ref{theorem1} allow us to compute a formula for the Shalika functional of generic unramified representations of ${\rm PGU}_{2,2}(F)$ in terms of representation theoretic invariants of the dual group of ${\rm PGSp}_4(F)$.\\

Let us describe the result in detail. Let $\Pi$ be a generic unramified irreducible representation of ${\rm PGU}_{2,2}(F)$. We fix a spherical vector $\phi_0$ of $\Pi$,  which we normalize by setting $\phi_0(1)=1$. In view of Theorem \ref{theorem1}, we are interested in computing explicitly the values of $\mathcal{S}_\Pi(g):=\Lambda_{\mathcal{S}}(\Pi(g)\phi_0)$ as $g$ varies, where $\Lambda_{\mathcal{S}}$ is the unique non-trivial Shalika functional for $\Pi$. Note that $\mathcal{S}_\Pi(g)$ is the image of $\phi_0$ under the embedding to ${\rm Ind}_{S(F)}^{{\rm GU}_{2,2}(F)}(\chi_S)$ defined by $\Lambda_{\mathcal{S}}$. As $\phi_0$ is spherical, it is sufficient to evaluate $\mathcal{S}_\Pi$ at a set of double coset representatives of
 $S(F)\backslash {\rm GU}_{2,2}(F) / {\rm GU}_{2,2}(\O)$. Using the Iwasawa and Cartan decompositions, we reduce it to calculate the value of $\mathcal{S}_\Pi$ at the matrices \[g_n:=\left(\begin{smallmatrix} \varpi^n I_2 &  \\  & I_2 \end{smallmatrix} \right),\] with $\varpi$ a uniformizer of $F$ and $n$ a non-negative integer. 

Before stating the formula, we introduce some of the notation used in it and refer to \S \ref{S:CSformulasplitandinert} for further clarifications. Firstly, by Theorem \ref{theorem1}, we know that $\Pi$ is the small theta lift of a generic unramified representation $\sigma$ of ${\rm PGSp}_{4}^+(F)$, which can be realized as the generic ${\rm PGSp}_{4}^+(F)$-factor of an unramified principal series $I_{{\rm GSp}_{4}}(\chi)$ of ${\rm PGSp}_{4}(F)$. We then denote by  $g_\chi \in {\rm Sp}_4(\C)$ a representative of the Frobenius conjugacy class of $I_{{\rm GSp}_{4}}(\chi)$. Moreover, for each root $\alpha \in \Phi_{\rm Sp_4}$, $\check{\alpha}$ denotes the corresponding coroot and the notation $e^{\check{\alpha}}(g_{\chi})$ stands for $\chi(g_\alpha)$, where, if $\alpha: {\rm diag}(t_1,t_2,t_3,t_4) \mapsto t_it_j^{-1}$,  $g_\alpha$ is equal to the diagonal matrix with $\varpi$ on the $i$-th line, $\varpi^{-1}$ on the $j$-th line, and $1$'s otherwise. Finally, we let $\rho$ denote the half sum of the positive roots.

\begin{theorem}[{Theorem \ref{CasselmanShalikaformula}}]\label{theorem2}
Let $\sigma$ and $\Pi$ be irreducible generic unramified representations of $\mathrm{PGSp}_4^+(F)$ and ${\rm PGU}_{2,2}(F)$, for which $\Pi$ is the small theta lift of $\sigma$. For all $n \geq 0$, a suitably normalized spherical Shalika functional satisfies
   \[\mathcal{S}_\Pi( g_n) = \frac{q^{-2n}}{ 1 + q^{-1}} \mathcal{A}\left( (-1)^n e^{\check{\rho}+n(\check{\alpha}_1 + \check{\alpha}_2)   }\prod_{\alpha \in \Phi^{+,s}_{\Sp_4}}(1 + q^{-1}e^{-\check{\alpha}})\right)(g_{\chi}) (\mathcal{A}(e^{\check{\rho}}) (g_{\chi}))^{-1},\] 
   where  $q$ is the cardinality of the residue field, $\mathcal{A}(\cdot) = \sum_{\omega\in W_{\Sp_4}}(\-1)^{\ell(\omega)} (\cdot)^{\omega}$ is the alternating sum over the Weyl group of $\Sp_4$, and $\Phi^{+,s}_{\Sp_4}$ is the set of positive short roots.
\end{theorem}
 One key feature of our formula is that the use of the theta correspondence seems indispensable for giving such an elegant and useful presentation of $\mathcal{S}_\Pi( g_n)$, which in \S \ref{sec_application_to_Lvalues} let us connect $\mathcal{S}_\Pi$ to  a degree 5 $L$-factor of $\sigma$. We believe that this strengthens the connection between the Shalika functionals on $\mathrm{PGU}_{2,2}(F)$ and the theta correspondence.

An analogous formula for Shalika models on $\mathrm{PGL}_4(F)$ was given in \cite{Sakellaridis}. Our proof is based on adapting the methods and calculations of \textit{loc.cit.} to our setting. While the strategy for proving Theorem \ref{theorem2} follows verbatim the one of \cite[Theorem 2.1]{Sakellaridis}, Theorem \ref{theorem2} presents a few technical challenges due to the fact that $\mathrm{PGU}_{2,2}$ is not split. 

Let us outline the idea of the proof in more detail. 
First, in Proposition \ref{relation} we use the work of \cite{Hironaka} (and \cite{Sakellaridis}) to reduce the computation of the Shalika functional at a spherical vector of $\Pi$ to the evaluation of the associated Shalika distribution at certain Weyl translates of the characteristic function of the Iwahori subgroup. We remark that at this stage the uniqueness result obtained in Theorem \ref{theorem1} is essential. To compute explicitly the distribution at these functions, we want to use the direct period formula given by Lemma \ref{formulasupp}, which however can be a priori applied on a subspace of $\Pi$ that does not contain those Weyl translates. To overcome this, in Proposition \ref{proponconv} we first show that the period can be evaluated at all the elements of unramified principal series whose induced data varies in an open region of the space of unramified characters of $T_{{\rm GU}_{2,2}}(F)$. Furthermore, in Corollary \ref{rationalitySm}, we show how we can extend the formula to all the unramified principal series. With this and the explicit calculation of the theta correspondence in hand, in \S \ref{ThirdReductionSection} we can finally compute each term of the right hand side of Proposition \ref{relation} and conclude in Corollary \ref{Coronrightnormalizationinert} the proof of Theorem \ref{theorem2}. We would like to remark that the computations of Proposition \ref{proponconv} and Proposition \ref{funct2} (e.g. Lemma \ref{comp2aux}) are the technical novelties in the proof.

\subsection{An integral representation for the standard $L$-function of ${\rm PGSp}_4$}

The formula of Theorem \ref{theorem2} has applications to the calculation of Rankin--Selberg integrals which unfold to the Shalika model. To this extent, in \S \ref{sec_application_to_Lvalues} we show how it can be implemented to calculate certain local zeta integrals. As a corollary, we give a more direct and alternative proof of \cite[Theorem 1.2]{CauchiGuti}.\\

Let $\Pi$ be a globally generic cuspidal automorphic representation on $\mathrm{PGU}_{2,2}(\A_F)$, where $F$ is a number field and $\mathrm{PGU}_{2,2}$ is associated to a quadratic field extension $E/F$. In \cite{CauchiGuti}, we considered the integral
\begin{equation*}  J(\varphi,  s) :=\int_{Z_{{\rm GSp}_4}(\A_F){\rm GSp}_4(F)\setminus {\rm GSp}_4(\A_F)} E^*_{P_{{\rm GSp}_4}}(h,s) \varphi(h) d h,\end{equation*}
where $\varphi$ is a cusp form in $\Pi$ and $E^*_{P_{{\rm GSp}_4}}(h,s)$ is the normalized Siegel Eisenstein series for ${\rm GSp}_4$. In Theorem 1.2 of \textit{loc.cit.}, we showed that $J(\varphi, s)$ calculates the degree 5 standard $L$-function of $\mathrm{PGSp}_4$ twisted by the quadratic Hecke character $\chi_{E/F}$, by global means. Precisely, we realized it as the residue of a two variable Rankin--Selberg integral which calculates a product of exterior square $L$-functions of $\Pi$.

Since $J(\varphi, s)$ unfolds to the Shalika model of $\varphi$, Theorem \ref{theorem1} let us imply that $J(\varphi,s)$ is Eulerian. The corresponding local zeta integral at a place $v$ which is unramified for all the data can be written as

\begin{align*} 
      J_v( \phi_0, f_{v,s}, s) = \int_{\GL_2(F_v)N_{{\rm GSp}_4}(F_v)\setminus {\rm GSp}_4(F_v)} f_{v,s}(h) \mathcal{S}_{\Pi_v}(h) d h,
 \end{align*}
where $\phi_0$ is a spherical vector in $\Pi_v$, $f_{v,s}$ is the $v$-component of the section defining the Eisenstein series, and $\mathcal{S}_{\Pi_v}$ is as in \S \ref{section1.2}. In Theorems \ref{unrcompsplit} and \ref{unrcompinert}, we show the following.
\begin{theorem}\label{theorem3}
We have \[ J_v( \phi_0, f_{v,s}, s) = L(s, \sigma_v, {\rm std} \otimes \chi_{E_v/F_v}), \]
where \begin{enumerate}
    \item If $v$ is inert in $E$, $\sigma_v$ is a representation of ${\rm PGSp}_4(F_v)$ such that $\Pi_v$ is the small theta lift of the generic irreducible ${\rm PGSp}_4^+(F_v)$-factor of $\sigma_v$ and $\chi_{E_v/F_v}$ is the quadratic character associated by local class field theory to the quadratic field extension $E_v/F_v$;
    \item If $v$ splits in $E$, $\Pi_v$ is the small theta lift of $\sigma_v$ via the theta correspondence for the pair $({\rm PGSp}_4,{\rm PGL}_4)$ and  $\chi_{E_v/F_v}$ is trivial.
\end{enumerate}
\end{theorem}

We remark that Theorem \ref{theorem3} is entirely new, as in \cite{CauchiGuti} the authors did not calculate these local zeta integrals. As the reader might guess, the proof is based on the application of the Casselman--Shalika formulas of \cite{Sakellaridis} and the one proved in Theorem \ref{theorem2}.

\subsection{Structure of the manuscript}
In Section $2$, we introduce some of the notation used, the relevant groups and recall the exceptional isomorphism between $\mathrm{PGU}_{2,2}$ and $\mathrm{PGSO}_{4,2}$. Section $3$ is devoted to the study of the multiplicity of local Shalika functionals for generic unramified representations of $\mathrm{PGU}_{2,2}(F)$. There, we describe some of the properties of the local theta correspondence studied by \cite{morimoto}, we use Mackey theory to show that the above mentioned multiplicity is always positive, and we use various features of the local theta correspondence to deduce that it is equal to one. In Section $4$, we describe the Casselman--Shalika formula for local spherical Shalika functionals for $\mathrm{PGL}_{4}$ of \cite{Sakellaridis} as well as the new Casselman--Shalika formula for local spherical Shalika functionals for $\mathrm{PGU}_{2,2}$, which we then prove in Section $5$. Finally, in Section $6$, we show how the Casselman--Shalika formulas of Section $4$ have applications to the calculation of Rankin--Selberg integrals.

\subsection{Acknowledgements}
The project started when the two authors were both working at the Universitat Politecnica de Catalunya; we thus kindly thank Victor Rotger for his support. We are indebted to Wee Teck Gan, who explained to us how to prove uniqueness of Shalika models by means of the local theta correspondence. The first named author would also like to thank Patrick Allen, Mathilde Gerbelli-Gauthier, Aaron Pollack, Mart\'i Roset Juli\`a, and Giovanni Rosso for discussions related to this project. Finally, we would like to thank the anonymous referee for the valuable corrections and comments that helped us to considerably improve this article. A.C.'s research in this publication was conducted with the financial support  of the NSERC grant RGPIN-2018-04392 and Concordia Horizon postdoc fellowship n.8009, of the JSPS Postdoctoral Fellowship for Research in Japan, and of Taighde \'{E}ireann -- Research Ireland under Grant number IRCLA/2023/849 (HighCritical). A.G.T. was supported by the Morningside Center of Mathematics (CAS). Both authors were also supported by the European Research Council (ERC) under the European Union's Horizon 2020 research and innovation programme (grant agreement No. 682152).

\section{Notation}\label{prelim}

We let $F$ be a non-archimedean local field  of characteristic zero with ring of integers $\mathcal{O}$. We denote by $q$ the size of the quotient of $\mathcal{O}$ modulo its maximal ideal $\mathfrak{p}$ and fix a uniformizer $\varpi$ of $\mathcal{O}$. Accordingly, we normalize the norm $|\cdot |$ of $F$ so that $|\varpi| = q^{-1}$. 
Here and throughout the article, if $G_{/F}$ is an algebraic group with a Borel subgroup $B_G$, we will adopt the notation $I_G(\chi)$ for the normalized induction ${\rm Ind}_{B_G(F)}^{G(F)}(\chi)$. Also, if $G$ is any algebraic group, we let $PG$ denote the quotient of $G$ by its center.

\subsection{Groups}\label{15}

 \subsubsection{Definitions}\label{ss:defgroups}

Denote by $J_2$ the $2 \times 2$ anti-diagonal matrix with entries 1 in the anti-diagonal and let $J=\left( \begin{smallmatrix}  & J_2 \\ -J_2 & \end{smallmatrix} \right)$. Let $E$ be either $F\times F$ or a quadratic field extension of $F$. When $E$ is a field, choose $d \in F^\times / (F^\times)^2$ and let $E=F(\delta)$ be the quadratic extension with ring of integers $\O_E$ such that $\overline{\delta}=-\delta$, with $\delta = \sqrt{d}$. Define $\G$ to be the group scheme over $F$ given by  
\[ \G(R):=\GU_{2,2}(R) = \{ (g,m_g) \in \GL_4(R \otimes_{F}  E ) \times R^\times :\bar{g}^t J g = m_g J\}, \]
where $\bar{\bullet}$ denotes the non-trivial automorphism of order 2 of $E/F$ and $R$ denotes any $F$-algebra.  Denote by $\nu:\G \to \Gm, \; g \mapsto m_g$ the similitude character.  Its kernel is the group ${\rm U}_{2,2}$. 
Define the subgroup $\G^\star$ of $\G$ given by 
\[ \G^\star(R) = \{(g,m_g) \in {\GU}_{2,2} \; : \; {\rm det}(g)=m_g^2 \}. \]
Note that $\G^\star$ is a form of $\GSpin_6$. Inside $\G^\star$, we have the subgroup \[\H(R) := \GSp_4(R)=\{ (g,m_g) \in \GL_4(R) \times R^\times :g^t J g = m_g J\}. \]
Denote by $\iota$ the natural embedding \[ \iota: \H \hookrightarrow \G,\; (g,m_g) \mapsto (g,m_g),\]
induced by the inclusion $\GL_4 \hookrightarrow {\rm Res}_{E/F} \GL_4$.

\begin{remark}\label{IntegralModelsRemark}
   The groups defined above arise as the generic fiber of group schemes over $\O$ by replacing $E$ and $F$ by $\O_E$ and $\O$ in the definition. If $G$ is any of these groups, we denote by $G(\O)$ the $\O$-points of the integral model obtained in this way.
\end{remark}

When $E/F$ is a field extension, define the similitude orthogonal group $\GO_{2+i,i}$ for $i=0,1,2$ as the group scheme over $F$ whose points are
\[ \GO_{2+i,i}(R) = \{(g,\nu_g) \in \GL_{2+2i}(R) \times R^\times : g^t S_i g = \nu_g S_i\},\]
where \[ S_0 = \left( \begin{smallmatrix}  2&  \\  & -2 d \end{smallmatrix} \right),  \; S_1 = \left( \begin{smallmatrix}  &  & 1 \\ & S_0 & \\ 1 & & \end{smallmatrix} \right),\; S_2= \left( \begin{smallmatrix}  &  & 1 \\ & S_1 & \\ 1 & & \end{smallmatrix} \right).\]
Moreover, let $\GSO_{2+i,i}$ be the subgroup

\[ \GSO_{2+i,i}(R) = \{ (g, \nu_g)  \in \GO_{2+i,i}(R) : {\rm det} (g) = \nu_g^{1+i}\}.\]

\noindent Notice that $\GSO_{2,0}(F) \simeq E^\times$; we identify these groups in a way so that the character $(g,\nu_g) \mapsto \nu_g$ is given by the norm map $N_{E/F}$. Throughout the paper, we denote $\G' := \GSO_{4,2}$. Finally, for any $F$-algebra $R$, define \[ \H^+(R) : = \{ h \in \H(R)\,:\, m_h \in N_{E/F}((R \otimes_F E)^\times)\}.\]
The group $\H^+(F)$ is normal in $\H(F)$ with quotient isomorphic to $F^\times / N_{E/F}(E^\times)$.

\subsubsection{Subgroups}

Let $B_\G=T_\G U_\G$ denote the upper-triangular Borel subgroup of $\G$, with $T_\G$ the diagonal torus \[\left \{\left(\begin{smallmatrix}a& &&\\&b&& \\ & &\nu \bar{b}^{-1} & \\ & &&\nu \bar{a}^{-1}\end{smallmatrix}\right)\;:  a,b\in {\rm Res}_{E/F}\Gm ,\;\nu  \in {\Gm}_{/F} \right \}.\]
The modulus character  $\delta_{B_\G}: B_\G(F) \to \C^\times$ is given by \[ \left(\begin{smallmatrix}a& &&\\&b&& \\ & &\nu \bar{b}^{-1} & \\ & &&\nu \bar{a}^{-1}\end{smallmatrix}\right) \mapsto \tfrac{|a\bar{a}|^3|b\bar{b}|}{|\nu|^4}.\]
We let $B_\H = T_\H U_\H = B_\G \cap \H$ be the upper triangular Borel subgroup of $\H$ and let $P_G = M_G N_G$ be the standard Siegel parabolic of $G \in \{ \H, \G \}$. For instance,  \[M_\H = \left\{ \left(\begin{smallmatrix} g & \\ & \mu J_2 {}^tg^{-1}J_2 \end{smallmatrix} \right),\; g \in \GL_{2/F}, \mu \in {\Gm}_{/F} \right\}, \]
\[ N_\H = \left\{ \left(\begin{smallmatrix} 1 & & a & x\\ & 1 & y & a \\ & & 1 & \\ & & & 1\end{smallmatrix} \right),\; \text{ with } x,y,a \in \mathbb{G}_{{\rm a} /F} \right\}.\]
Denote by $\delta_{P_\H}: P_\H(F) \to \C^\times$  the modulus character of $P_\H$ given by \begin{align*}
    \delta_{P_\H} &:\left(\begin{smallmatrix} g & \star \\ & \mu J_2 {}^tg^{-1} J_2 \end{smallmatrix} \right)  \mapsto \left|\tfrac{{\rm det}(g)}{\mu}\right|^3.
\end{align*}
We let $B_{\G'}$ be the upper-triangular Borel of $\G'$ with Levi the maximal diagonal torus \[T_{\G'} = \left \{\left(\begin{smallmatrix}a& && & \\&b&& & \\& & x & &  \\ & & & \nu b^{-1} & \\ & & &&\nu a^{-1}\end{smallmatrix}\right)\; :\,a,b \in \Gm,\; x \in {\rm Res}_{E/F}\Gm ,\; \nu = N_{E/F}(x)  \right \}.\]
Finally, the Shalika subgroup $S$ of $\G$ is given by
\begin{equation}\label{shalgrp}S = \left\{\left(\begin{smallmatrix}a&b&&\\c&d&& \\ & &a&-b\\ & &-c&d\end{smallmatrix}\right) \left(\begin{smallmatrix} 1& &\alpha & x\\& 1&y&\overline{\alpha}\\ & & 1& \\ & & &1 \end{smallmatrix}\right)\;:\; \left(\begin{smallmatrix}a&b\\c&d\end{smallmatrix}\right) \in \GL_{2/F}, x,y\in \mathbb{G}_{a/F},\;\alpha\in {\rm Res}_{E/F}\mathbb{G}_a \right\}.\end{equation}

\noindent The Shalika group $S$ is isomorphic to $\GL_2 N_{\G}$, where $\GL_2$ embeds into the Levi of the Siegel parabolic $M_\H$ of $\H$ via the embedding $h \mapsto h^*:= \left(\begin{smallmatrix}h& \\ & {\rm det}(h) J_2 {}^th^{-1} J_2\end{smallmatrix}\right)$. Notice that the modulus character $\delta_S$ of $S$ is trivial.

\subsubsection{An exceptional isomorphism}

Let $E/F$ be a quadratic field extension. We have an exceptional isomorphism $j: \mathbf{P}\G\simeq \mathbf{P}\G'$, where recall that the letter $\mathbf{P}$ stands for taking the quotient by the center. In what follows, we sketch the construction of this isomorphism as done in \cite[pp.  33--34]{morimoto} and refer to \emph{loc.cit.} for further details. 
Define the six dimensional $F$-vector space\footnote{This differs from the one defined in \cite{morimoto} because of the different choices of Hermitian forms defining ${\rm GU}_{2,2}$. Conjugating by ${\matrix {J_2} {} {} {I}}$, one gets the vector space introduced in \emph{loc.cit.}. }
\[\mathcal{V} := \left\{v(x_1,x_2,x_3,x_4,x_5,x_6) := \begin{pmatrix}0& -x_1 & x_2 & - x_3+\delta x_4\\x_1&0 &x_3+\delta x_4 &x_5\\ -x_2& -x_3-\delta x_4&  0&x_6\\ x_3-\delta x_4 &-x_5& -x_6 & 0\end{pmatrix}\;|\;x_i\in F \right\},\]
endowed with the bilinear form  $\Psi: \mathcal{V} \times \mathcal{V} \to F$ defined by 
\[ \Psi(v_1,v_2) = {\rm Tr}(v_1 J' \overline{v_2}^t J'),\;\text{with } J' ={\matrix {} {J_2} {J_2} {}}.\]
Choose the basis of $\mathcal{V}$ given by $\{ v_i \}_i$, where $v_i = v(e_i)$  for $i=1, \cdots, 6$, with $\{ e_i \}_i$ denoting the standard basis of $F^6$. This induces an isomorphism ${\rm GO}(\mathcal{V}, \Psi) \simeq {\rm GO}_{4,2}$.
The homomorphism 
\begin{align}\label{actione2}
    \phi:\;\G^{\star}&\to  {\rm GO}(\mathcal{V}, \Psi),\\
    g&\mapsto (v\mapsto gv^{t}g),\nonumber
\end{align}
is well defined and satisfies the relation  $\nu_{\phi(g)} = m_g^2$. Hence, $\phi$ induces a homomorphism $\Phi:\;\G^{\star}\to \G'$.
Given $\alpha\in E^{\times}$,  define $r_{\alpha} = \left(\begin{smallmatrix}\alpha& & & \\ & 1& & \\ & &1& \\ & & &\overline{\alpha}^{-1}\end{smallmatrix}\right)$. There is an action of $E^\times$ on $\mathcal{V}$ via the rule  
\begin{equation}\label{actione} \alpha\cdot v = \overline{\alpha}r_{\alpha}vr_{\alpha}, \, \text{for }\alpha\in E^{\times}, v\in \mathcal{V}.\end{equation}
On the other hand, $E^{\times}$ acts on $g\in \G^{\star}$ by $\alpha\cdot g = r_{\alpha}gr_{\alpha}^{-1}$. This action allows us to consider the semidirect product $\G^{\star}\rtimes E^{\times}$. Then, using the action \eqref{actione}, $\Phi$ extends to a surjective homomorphism $\Gamma:\;\G^{\star}\rtimes E^{\times}\to \mathbf{PG}'$. Finally, the surjective homomorphism 
\begin{align*}\Xi:\;\G^{\star}\rtimes E^{\times}&\to \mathbf{P}\G\\ (g,\alpha)&\mapsto gr_{\alpha}\nonumber\end{align*}
satisfies the property that $\mathrm{ker}(\Gamma) = \mathrm{ker}(\Xi)$ and hence we can define $j:=\Gamma^*\circ (\Xi^*)^{-1}:\;\mathbf{P}\G\simeq \mathbf{P}\G'$. By lifting any diagonal element of $\mathbf{P}\G$ to $\G^{\star}\rtimes E^{\times}$ and calculating the action (via \eqref{actione2} and \eqref{actione}) of its lift on the basis $\{v_i\}_i$ of $\mathcal{V}$, we obtain the following.

\begin{lemma}\label{excepisom}
The isomorphism $j: \mathbf{P}\G\simeq \mathbf{P}\G'$ sends \[j: \left(\begin{smallmatrix}a& &&\\&b&& \\ & &\nu \bar{b}^{-1} & \\ & &&\nu \bar{a}^{-1}\end{smallmatrix}\right) \mapsto \left(\begin{smallmatrix} N_{E/F}(ab)  & && & \\&  \nu N_{E/F}(a )  && & \\& & \nu \bar{a} b & &  \\ & & &  \nu N_{E/F}(b)  & \\ & & && \nu^2  \end{smallmatrix}\right),  \] 
 with $a,b\in {\rm Res}_{E/F}\Gm $ and $\nu  \in \Gm $.
\end{lemma}

\subsection{Root system and Weyl groups}

\subsubsection{$L$-groups}

Recall that the $L$-group of $\G$ is \[ {}^L\G = ( \GL_4(\C) \times \GL_1(\C)) \rtimes {\rm Gal}(E/F),  \]
with the action of the non-trivial element $\tau \in {\rm Gal}(E/F)$ given by (cf. \cite[\S 1.8(c)]{BlasiusRogawski})
\[ (g,\lambda) \mapsto (\Phi_4 {}^tg^{-1} \Phi_4 , \lambda {\rm det}(g)), \text{ where }  \Phi_4 = \left( \begin{smallmatrix}  & & &1\\&&-1& \\ & 1&&\\ -1& & & \end{smallmatrix} \right). \] 
The $L$-group of $\mathbf{P}\G$ is the derived subgroup of ${}^L\G$, namely
${}^L\mathbf{P}\G = {\rm SL}_4(\C) \rtimes {\rm Gal}(E/F)$. Recall also that $ ^L\mathbf{P}\H = {\rm Spin}_5(\C) \simeq \Sp_4(\C)$. Under this exceptional isomorphism, the standard representation of $\Sp_4(\C)$ coincides with the spin representation of ${\rm Spin}_5(\C)$.
 
\subsubsection{Root system}
Write any element of $T_\H$ as
$t(a,b, \nu)=\left(\begin{smallmatrix}a& &&\\&b&& \\ & &\nu b^{-1} & \\ & &&\nu a^{-1}\end{smallmatrix}\right)$ and let $\alpha_1, \alpha_2, \alpha_0 \in X^*(T_\H)$ be the characters \[ \alpha_1(t(a,b, \nu))=a,\, \alpha_2(t(a,b, \nu))=b,\, \alpha_0(t(a,b, \nu))=\nu.\]
Furthermore,
let $\check{\alpha}_1, \check{\alpha}_2, \check{\alpha}_0 \in X_*(T_\H)$ be the cocharacters \[ \check{\alpha}_1(x)=t(x,1, 1),\, \check{\alpha}_2(x)=t(1,x, 1),\, \check{\alpha}_0(x)=t(1,1, x).\]
We denote by $\langle \, , \, \rangle: X^*(T_\H) \times X_*(T_\H) \to \Z$ the canonical pairing given by composition. Our choice of Borel subgroup $B_\H$ fixes the set of  positive roots for $\mathbf{P}\H$ (and $\H$); explicitly,  \[ \Phi_{\mathbf{P}\H}^+ = \{ \alpha_1 - \alpha_2, \alpha_1 + \alpha_2 - \alpha_0, 2 \alpha_1 - \alpha_0,  2 \alpha_2 - \alpha_0  \}.\]
The set of simple positive roots consists of the short root $\alpha_1 - \alpha_2$ and the long one $2 \alpha_2 - \alpha_0$. We also denote by $\Phi_{\mathbf{P}\H}^{+,s}$, resp. $\Phi_{\mathbf{P}\H}^{+,l}$, the set of positive short, resp. long, roots. We have $\Phi_{\mathbf{P}\H}^{+,l} = \{ 2 \alpha_1 - \alpha_0,  2 \alpha_2 - \alpha_0 \}$.

To each root $\alpha$, we denote by $\check{\alpha}$ the corresponding coroot. There is a homomorphism $x_\alpha: {\rm SL}_2 \to \mathbf{P}\H$ such that $x_\alpha \left(\left( \begin{smallmatrix}
    1 & c \\ & 1
\end{smallmatrix}\right) \right) \subseteq N^\alpha$, the latter being the one parameter unipotent subgroup of $\alpha$. Then, $\check{\alpha} = x_\alpha \circ \iota$, with $\iota : \GL_1 \to \mathrm{SL}_2$ the embedding of the diagonal torus, is such that $\langle \alpha, \check{\alpha} \rangle =2$. Explicitly, \[\Phi_{\mathbf{P}\H}^{\vee,+} = \{ \check{\alpha}_1 - \check{\alpha}_2, \check{\alpha}_1 + \check{\alpha}_2, \check{\alpha}_1, \check{\alpha}_2 \} \]
is the set of positive coroots. We can identify the set of positive roots of $\mathbf{P}\H$ with the set of positive coroots of $\Sp_4$  as (see the proof of \cite[Lemma 2.3.1]{RobertsSchmidt}) 
\begin{align}\label{Identify_roots_and_coroots}
    \{\alpha_1 - \alpha_2, \alpha_1 + \alpha_2 - \alpha_0, 2 \alpha_1 - \alpha_0 , 2 \alpha_2 - \alpha_0 \} \to \{ \check{\alpha}_2, \check{\alpha}_1 ,  \check{\alpha}_1 +  \check{\alpha}_2,  \check{\alpha}_1 -  \check{\alpha}_2  \}.
\end{align}
Note that the set $\Phi_{\mathbf{P}\H}^{+,l}$ is sent to the set $\Phi^{\vee, +,s}_{\Sp_4} = \{ \check{\alpha}_1 +  \check{\alpha}_2,  \check{\alpha}_1 -  \check{\alpha}_2 \}$ of coroots corresponding to the short roots for $\Sp_4$.

\subsubsection{Weyl groups}\label{ss:Weylgroup}
Let $G_{/F}$ be an algebraic group. Given a maximal split subtorus $A \subset G(F)$, we let the Weyl group be $W_G = N_{G}(A)/C_{G}(A)$, where $N_{G}(A)$, resp. $C_{G}(A)$, denotes the normalizer, resp. centralizer, of $A$ in $G(F)$. When $G$ is either $\H_{/F}$ or $\G_{/F}$ with $E/F$ a quadratic field extension, $A$ can be taken to be $T_\H(F)$. We then have that $W_\H \simeq W_\G$. The Weyl group $W_\H$ is generated by the reflections $s_1,s_2$ associated to the positive roots $\alpha_1-\alpha_2$, $2\alpha_2 - \alpha_0$ respectively. We choose coset representatives 
\[s_1  =\left(\begin{smallmatrix} &1& & \\1& & & \\ & & & 1\\ & & 1& \end{smallmatrix}\right),\;\;s_2 =  \left(\begin{smallmatrix}1& & & \\ & & 1& \\ & -1& & \\ & & &1\end{smallmatrix}\right).\]
Then $W_\H$ is of order eight and its elements are $\{1, s_1, s_2, s_1 s_2, s_2 s_1, s_1 s_2 s_1, s_2 s_1 s_2 , s_1 s_2 s_1 s_2 \}$. We will denote the longest Weyl element by $\omega_0$. Finally, given a character $\chi$ of $T_G(F)$, we let $^\omega \chi(t) := \chi(\omega^{-1}  t \omega)$, for any $\omega \in W_G$.

\section{On the multiplicity of local Shalika functionals}\label{S:UniquenessMackey}

\subsection{Shalika models}\label{ss:Shalikamodels}
Let $F$ be a non-archimedean local field of characteristic zero and let $E$ be either $F\times F$ or a quadratic field extension of $F$. When $E$ is a field, choose $d \in F^\times / (F^\times)^2$ and let $E=F(\delta)$ such that $\overline{\delta}=-\delta$, with $\delta = \sqrt{d}$. We   let $\G$ be the group over $F$ associated to $E/F$ defined as in \S \ref{ss:defgroups}. In particular, when $E = F \times F$, the projection to the first component induces an isomorphism between $\G(F) \simeq \GL_4(F) \times \GL_1(F)$. We conveniently compose this isomorphism with conjugation by the diagonal matrix $w=({\rm diag}(1,1,-1,1),-1) \in \GL_4(F) \times \GL_1(F)$. This has the effect of identifying the subgroups \[ \GL_2(F) =\left\{ \left(\begin{smallmatrix} g &  \\  &  {\rm det}(g) J_2 {}^tg^{-1}J_2 \end{smallmatrix} \right)\right\} \simeq  \left\{\left( \left(\begin{smallmatrix} g &  \\  &  g \end{smallmatrix} \right), {\rm det}(g) \right)\right\}, \]
\[N_\G(F) \simeq \left\{u(X) = \left(\begin{smallmatrix} I & X \\  & I \end{smallmatrix} \right), X \in M_{2\times 2}(F) \right\}.\]
Let $\psi: F \to \C^\times$ be a non-trivial additive character. Define a character of $N_{\G}(F)$ as follows.  If $E = F(\delta)$, let
$\chi: N_{\G}(F) \longrightarrow \C^{\times}, \, \left( \begin{smallmatrix}1& &\alpha&x\\&1 &y&\overline{\alpha}\\ & & 1& \\ & & &1 \end{smallmatrix}\right)  \mapsto \psi(\delta\alpha-\delta\overline{\alpha})$. If $E= F \times F$, let $\chi: N_{\G}(F) \longrightarrow \C^{\times}, \,  u(X) \mapsto \psi( {\rm Tr}(X))$. We then define $\chi_{S}$ to be the character on $S(F)$ given by $ h^* n \mapsto \chi(n)$.  Let $\Pi$ be any irreducible admissible representation of $\G(F)$ with trivial central character. Note that, when $E=F \times F$, the isomorphism above identifies $\Pi$ with a representation of ${\rm PGL}_4(F)$. 

\begin{definition}
The representation $\Pi$ is said to have a Shalika model if
\[\dim_\C \mathrm{Hom}_{S(F)}(\Pi,\chi_S)\neq 0.\]

\end{definition}
Any element $\Lambda_{\mathcal{S}} \in \mathrm{Hom}_{S(F)}(\Pi,\chi_S)$, which we refer as Shalika functional, has the property that for any $s \in S(F)$ and $v$ in the space of $\Pi$,
\[\Lambda_{\mathcal{S}}(\Pi(s)v) = \chi_{S}(s)\Lambda_{\mathcal{S}}(v).\]
By Frobenius reciprocity, having a non-trivial Shalika functional is equivalent to having a $\G(F)$-equivariant injection \[ \Pi \hookrightarrow \mathrm{Ind}_{S(F)}^{\G(F)}\chi_{S}.\] \\

When $E= F \times F$, the isomorphism $\mathbf{PG}(F) = \mathrm{PGL}_4(F)$ let us identify the space of Shalika functionals on $\Pi$ with the space of Shalika functionals on admissible representations of $\mathrm{PGL}_4(F)$. Then, by a result of Jacquet--Rallis, we have the following. 

\begin{proposition}[{\cite[Proposition 6.1]{JacquetRallis}}]\label{uniquenesssplit} If $\Pi$ is an irreducible admissible representation of $\mathrm{PGL}_4(F)$, then \[{\rm dim}_\C \, {\rm Hom}_{S(F)}(\Pi, \chi_{S}) \leq 1. \]
\end{proposition} 
\noindent By the results of \cite{Sakellaridis}, if $\Pi$ is unramified, the space of Shalika functionals is one-dimensional if and only if $\Pi$ is generic and a representative of its Frobenius conjugacy class, which a priori is in ${\rm SL}_4(\C)$, belongs to the subgroup ${\rm Sp}_4(\C)$ (\textit{cf}. \cite[Corollary 5.4]{Sakellaridis} and the discussion at the end of \S 8 of {\emph{loc.cit.}}), namely we have the following.

\begin{proposition}\label{ProponShalikasplitmultiplicity}
 If $\Pi$ is a generic, unramified irreducible representation of $\mathrm{PGL}_4(F)$, then \[{\rm dim}_\C \, {\rm Hom}_{S(F)}(\Pi, \chi_{S}) = 1\] 
if and only if  a representative of its Frobenius conjugacy class is in ${\rm Sp}_4(\C)$.
\end{proposition}

Suppose now that $E$ is the unique unramified quadratic field extension of $F$. In this case, the analogue of Proposition \ref{ProponShalikasplitmultiplicity} does not seem to appear in the literature. The objective of this chapter is to prove it; namely we study the existence and uniqueness of Shalika functionals for generic, unramified irreducible representations of $\mathbf{PG}(F) = \mathrm{PGU}_{2,2}(F)$.
We remark that Morimoto in \cite{morimoto} proved the existence of non-trivial Shalika functionals for those representations. We would like to remark that Morimoto's proof follows from the computation of the local theta correspondence for the dual pair $(\mathbf{P}\H^+(F),\mathbf{P}\G(F))$ and of the twisted Jacquet modules of the Weil representation (see \cite[Theorems 6.9, 6.21, and 6.23]{morimoto}). In this section we provide an alternative proof of their result. 

Since the group $\mathbf{PG}(F)$ acts transitively on the set of principal characters of its maximal unipotent subgroup, \cite[Theorem 2.7]{LiUPSpadic} implies that $\Pi$ is generic and unramified if and only if it is an irreducible principal series. In sections \S \ref{MackeyTheory}--\ref{UniquenessShalikaFunctionals} we proceed as follows: \begin{enumerate}
    \item Using Mackey theory, we show that a generic, unramified irreducible representation $\Pi$ of $\mathbf{PG}(F)$ has always a non-trivial Shalika functional. 
    \item We then show, by employing the calculations on the local theta correspondence for the dual pair $(\mathbf{P}\H^+(F),\mathbf{P}\G(F))$ of \cite{morimoto}, that a generic unramified representation $\Pi$ of $\mathbf{PG}(F)$ always lies in the image of the theta correspondence from $\mathbf{P}\H^+(F)$. 
    \item With this in hand, we can prove that the space of Shalika functionals for a generic, unramified  representation $\Pi$ of $\mathbf{PG}(F)$ is of dimension one by relating it, via the theta correspondence, to the space of Whittaker functionals of a generic representation of $\mathbf{P}\H^+(F)$. 
\end{enumerate}
In sections \S \ref{genunrphplus} and \S \ref{ThetaCorrespIntroSection}, we give the necessary representation theoretic background on the local theta correspondence studied by \cite{morimoto}, which is crucially used in the points (2) and (3) above, and might be skipped at a first glance.

\subsection{Generic unramified representations of \texorpdfstring{$\mathbf{P}\H^+(F)$}{(PH+(F)}}\label{genunrphplus}

From now on, we let $E$ be the unique unramified quadratic field extension of $F$ and let $\chi_{E/F}$ be the character on $F^\times$ corresponding to $E$ via local class field theory. An irreducible generic unramified representation $\sigma$ of $\mathbf{P}\H(F)$ is necessarily a principal series $I_\H(\chi)$, for a unique $W_\H$-orbit of characters $\chi = \chi_1 \otimes \chi_2 \otimes \chi_0$, where  $\chi_1,\chi_2,\chi_0$ are unramified characters such that $\chi_1 \chi_2 \chi_0^2 = 1$ and $|\cdot |^{\pm 1} \not\in \{ \chi_1,\chi_2, \chi_1\chi_2, \chi_1\chi_2^{-1}\}$ (\textit{cf}. \cite[Tables A.1 \& A.13]{RobertsSchmidt}).

\begin{lemma}\label{FromGSp4toGSp4+}
The restriction $I_{\H^+}(\chi)$ of $I_\H(\chi)$  to $\H^+(F)$ is irreducible if and only if  \begin{align*}
I_\H(\chi) \otimes \chi_{E/F} \not \simeq I_\H(\chi), 
\end{align*}
where $\chi_{E/F}$ is regarded as a representation of $\H(F)$ factoring through the symplectic multiplier $\nu:\H(F) \to F^\times$.
\end{lemma}
\begin{proof}
   By \cite[Lemma 2.1]{GelbartKnapp}, the restriction $I_{\H^+}(\chi)$ of $I_\H(\chi)$ is a finite direct sum of irreducible representations \[I_{\H^+}(\chi) = \oplus_{i=1}^M \pi_i^m, \]
   with each factor $\pi_i$ irreducible and inequivalent, of same multiplicity $m$ such that $m^2 M$ equals the number of one-dimensional characters of $\H(F)$ in the collection \[ X_{\H^+}(\chi) := \{ \text{ one-dimensional character }\psi \text{ of }\H(F)\;:\; \psi_{|\H^+(F)}=1\text{ and }I_\H(\chi) \otimes \psi \simeq I_\H(\chi)  \}. \] By \cite[Theorem\;1.4]{adlerPrasad}, $m=1$ and thus, \cite[Corollary 2.2]{GelbartKnapp} implies that $\H(F)/\H^+(F) \simeq \Z/2\Z$ acts transitively on the set $\{ \pi_1 , \pi_M \}$. Thus $M \leq 2$, with $M=2$ if $\chi_{E/F} \in X_{\H^+}(\chi)$.
\end{proof}

By \cite[(H3) \& Proposition 2.2]{GanTakedaSp4}, the condition of Lemma \ref{FromGSp4toGSp4+} can be checked at the level of Satake parameters. Precisely, we have the following.
\begin{lemma}\label{NDtoSP}
  We have that $I_\H(\chi) \otimes \chi_{E/F} \simeq I_\H(\chi)$ if and only if $\chi_{E/F} \in \{ \chi_1, \chi_2\}$.
\end{lemma}
\begin{proof}
    It follows from explicitly comparing the Weyl orbits of the two Frobenius conjugacy classes.
\end{proof}

Now suppose that $I_{\H^+}(\chi)$ is reducible. By the proof of Lemma \ref{FromGSp4toGSp4+}, it decomposes as the direct sum of two irreducible representations, which can be characterized as follows. Contrary to what happens for $\H(F)$, there are two orbits of the maximal torus of $\H^{+}(F)$ on the generic characters of $U_\H(F)$. Let $\psi$, $\psi_-$ be representatives of these two orbits, with $\psi$ equal to the character $\psi_N$ of \cite[\S 6.2]{morimoto}. Then, $I_{\H^+}(\chi)$ decomposes as a direct sum \[ I_{\H^+}(\chi) = \sigma_+ \oplus \sigma_-,\] with $\sigma_+$, resp. $\sigma_-$, generic with respect to $\psi$, resp. $\psi_-$  (\textit{cf}. \cite[ Lemma 4.2(3)]{morimoto}). In \textit{loc}.\textit{cit}. these components are conveniently described as inductions of the Klingen parabolic of $\H^+(F)$. If we let $Q_\H$ denote standard Klingen parabolic, with Levi $\GL_2 \times \GL_1$, then its restriction $Q_{\H^+}$ to $\H^+(F)$ has Levi $\GL_2^+(F) \times F^\times$, where 
\[\GL_2^+(F) : = \{ g \in \GL_2(F)\,:\, {\rm det}(g) \in {N}_{E/F}(E^\times) \}. \]
Let $B_2$ denote the upper-triangular Borel of $\GL_2^+(F)$ and denote by $\psi_2$, resp. $\psi_{2,-}$, the restriction of $\psi$, resp. $\psi_-$, to the unipotent radical of $B_2$. Following \cite[\S 4.2]{morimoto} (or adapting Lemmas \ref{FromGSp4toGSp4+} and \ref{NDtoSP} to $\GL_2$), any irreducible principal series of $\GL_2(F)$, whose restriction to $\GL_2^+(F)$ is reducible, is of the form $\varphi \cdot {\rm Ind}_{B_{\GL_2}(F)}^{\GL_2(F)}(\chi_{E/F}, 1 )$, with $\varphi$ a character of $F^\times$. Let ${\rm Ind}_{B_{\GL_2}(F)}^{\GL_2(F)}(\chi_{E/F}, 1 )^+$, resp. ${\rm Ind}_{B_{\GL_2}(F)}^{\GL_2(F)}(\chi_{E/F}, 1 )^-$, be the irreducible $\GL_2^+(F)$-constituent of ${\rm Ind}_{B_{\GL_2}(F)}^{\GL_2(F)}(\chi_{E/F}, 1 )$, which is generic with respect to $\psi_2$, resp. $\psi_{2,-}$. Then, by induction in stages we have that \[ I_\H(\chi) = {\rm Ind}_{Q_\H(F)}^{\H(F)}(\chi_1, \chi_0 \cdot {\rm Ind}_{B_{\GL_2}(F)}^{\GL_2(F)}(\chi_2, 1 )),\]
and, if $\chi_2 = \chi_{E/F}$, $I_{\H^+}(\chi) = \sigma_+ \oplus \sigma_-$, with 
\begin{align*}
    \sigma_+ &= {\rm Ind}_{Q_{\H^+}}^{\H^+(F)}(\chi_1,\chi_0 \cdot {\rm Ind}_{B_{\GL_2}(F)}^{\GL_2(F)}(\chi_{E/F}, 1 )^+),\\ 
     \sigma_- &=  {\rm Ind}_{Q_{\H^+}}^{\H^+(F)}(\chi_1,\chi_0 \cdot {\rm Ind}_{B_{\GL_2}(F)}^{\GL_2(F)}(\chi_{E/F}, 1 )^-).
\end{align*}

\subsection{Local theta correspondence for \texorpdfstring{$(\mathbf{P}\H^+(F),$}{(PH+(Qp),}\texorpdfstring{$\mathbf{P}\G(F))$}{PG(Qp))}}\label{ThetaCorrespIntroSection}
By the exceptional isomorphism of Lemma \ref{excepisom},  we consider the theta correspondence studied in \cite{morimoto} as a correspondence for $(\mathbf{P}\H^+(F), \mathbf{P}\G(F))$. Let $\Omega$ be the induced Weil representation for this pair of groups (\textit{cf}. \cite[p. 68]{morimoto}). 

\begin{definition}
    For any irreducible admissible representation $\sigma$ of $\mathbf{P}\H^+(F)$, we let the big theta lift $\Theta(\sigma)$ of $\sigma$ be the smooth representation of $\mathbf{P}\G(F)$ such that the maximal $\sigma$-isotypic quotient of $\Omega$ has the form $\sigma \boxtimes \Theta(\sigma)$. Similarly, if $\Pi$ is any irreducible admissible representation of $\mathbf{P}\G(F)$, we denote by $\Theta(\Pi)$ the big theta lift of $\Pi$, i.e. the smooth representation of $\mathbf{P}\H^+(F)$ such that the maximal $\Pi$-isotypic quotient of $\Omega$ has the form $\Theta(\Pi) \boxtimes \Pi$.
\end{definition}

By \cite[Theorem 6.21]{morimoto}, for any irreducible admissible representation $\sigma$, resp. $\Pi$, of $\mathbf{P}\H^+(F)$, resp. of $\mathbf{P}\G(F)$,  the representation $\Theta(\sigma)$, resp. $\Theta(\Pi)$, has a unique irreducible quotient $\theta(\sigma)$, resp. $\theta(\Pi)$, which is referred to as the small theta lift of $\sigma$, resp. of $\Pi$, if it is non-zero. In \textit{ibid.} the small theta lift of principal series of $\mathbf{P}\H^+(F)$ is calculated explicitly. Before stating this result, we need the following lemma. 

\begin{lemma}\label{changechar}
Let $\xi' = \xi'_1 \otimes \xi'_2 \otimes \xi'_0$ be a character of $T_{\mathbf{P} \G'}(F)$ defined by \[ 
\xi': \left(\begin{smallmatrix}a& && & \\&b&& & \\& & x & &  \\ & & & \mu b^{-1} & \\ & & &&\mu a^{-1}\end{smallmatrix}\right) \mapsto \xi'_1(a) \xi'_2(b) \xi'_0(x),\; \forall a,b \in F^\times, x \in E^\times\,(\text{and }\mu=N_{E/F}(x)).\] 
The character $\xi : = \xi' \circ j$ of  $T_{\mathbf{P} \G}(F)$ is given by $\left(\begin{smallmatrix}a& &&\\&b&& \\ & &\nu \bar{b}^{-1} & \\ & &&\nu \bar{a}^{-1}\end{smallmatrix}\right)  \mapsto \xi_1(a) \xi_2(b) \xi_0(\nu)$, where
\begin{align*}
    \xi_1(a) &= (\xi'_1\xi'_2 \circ N_{E/F})(a) \xi'_0(\bar{a}),\\
    \xi_2(b) &= \xi'_1 \circ N_{E/F} (b) \xi'_0(b),\\
    \xi_0(\nu)&= \xi'_2 \xi'_0(\nu).
\end{align*}

\end{lemma}

\begin{proof}
Recall that, by Lemma \ref{excepisom}, we have that
\[j: \left(\begin{smallmatrix}a& &&\\&b&& \\ & &\nu \bar{b}^{-1} & \\ & &&\nu \bar{a}^{-1}\end{smallmatrix}\right) \mapsto \left(\begin{smallmatrix} N_{E/F}(ab)  & && & \\&  \nu N_{E/F}(a )  && & \\& & \nu \bar{a} b & &  \\ & & &  \nu N_{E/F}(b)  & \\ & & && \nu^2  \end{smallmatrix}\right), \] 
 with $a,b \in E^\times $ and $\nu  \in F^\times$. Thus, $\xi : = \xi' \circ j$ is explicitly given by \[ \xi: \left(\begin{smallmatrix}a& &&\\&b&& \\ & &\nu \bar{b}^{-1} & \\ & &&\nu \bar{a}^{-1}\end{smallmatrix}\right) \mapsto \xi'_1(N_{E/F}(ab)) \xi'_2(\nu N_{E/F}(a) ) \xi'_0(\nu \bar{a}b). \]

Regrouping the terms of the latter as  \[\xi'_1(N_{E/F}(a))  \xi'_2(N_{E/F}(a)) \xi'_0(\bar{a}) \xi'_1(N_{E/F}(b))\xi'_0(b) \xi'_2(\nu) \xi'_0(\nu),   \] 
we get the result.
\end{proof}

\begin{proposition}\label{OnthetaliftsofPS}
Let $\sigma$ be an irreducible $\H^{+}(F)$-factor of an irreducible unramified principal series $I_{\H}(\chi)$ of $\mathbf{P}\H(F)$. Then, we have the following: \begin{enumerate}
    \item If $I_{\H^+}(\chi)$ is irreducible, then $\Theta(\sigma) =I_{\G}(\xi)$, with \[\xi_1 = \chi_0 \circ N_{E/F},\, \xi_2 = (\chi_2 \chi_0) \circ N_{E/F},\, \xi_0 = \chi_1 \chi_{E/F}.\] 
    Moreover, $\theta(\sigma)$ is the unique irreducible unramified quotient of $I_{\G}(\xi)$ and it equals to $ \Theta(\sigma)$ if it is generic. 
    
    \item If $I_{\H^+}(\chi)$ is reducible then, up to translating by the Weyl action, we fall in one of the following two cases: \begin{enumerate}
        \item Suppose that $\chi_2 = \chi_{E/F}$: \begin{enumerate}
            \item If $\sigma$ is generic with respect to $\psi$, then $ \Theta(\sigma) = I_{\G}(\xi)$, with \[\xi_1 = \xi_2 =  \chi_0 \circ N_{E/F},\, \xi_0 = \chi_1 \chi_{E/F}.\] 
        Moreover, $\theta(\sigma)$ is the unique irreducible unramified quotient of $I_{\G}(\xi)$ and it equals to $ \Theta(\sigma)$ if it is generic. 
        \item If $\sigma$ is generic w.r.t. $\psi_-$ (but not w.r.t. $\psi$), then $\theta(\sigma)=0$.
        \end{enumerate} 
        \item Suppose that $\chi_1 = \chi_{E/F}$ and $\chi_2 = \mathbf{1}$:  then
\[\Theta(\sigma) = \theta(\sigma) = \begin{cases} I_\G(\xi)_{\rm gen} & \text{ if } \sigma \text{ is generic w.r.t. } \psi,\\
    I_\G(\xi)_{\rm ng} & \text{ if } \sigma \text{ is generic w.r.t. } \psi_-,
\end{cases}
  \]      
  where $\xi_1 = \xi_2 =  \chi_0 \circ N_{E/F},\, \xi_0 = \mathbf{1},$
   and $I_\G(\xi)_{\rm gen}$ (resp. $I_\G(\xi)_{\rm ng}$) is the generic (resp. non-generic) constituent of $I_\G(\xi)$. In particular, $\theta(\sigma)$ is not unramified when $\sigma$ is generic w.r.t. $\psi$.
    \end{enumerate}
    \end{enumerate}
\end{proposition}

\begin{proof}
    The theta lift of $\sigma$ is calculated explicitly in \cite[Theorem 6.21(3)]{morimoto}. According to the classification of the non-discrete series of $\H^+(F)$ given in \cite[Lemma 4.2]{morimoto}, $\sigma$ falls in one of the cases (3a), (3b), or (3d) (the remaining case (3c) is excluded as we are supposing $I_{\H}(\chi)$ to be irreducible).   Let's treat the cases separately. \begin{itemize}
        \item[(3d)] Let $\sigma=I_{\H^+}(\chi)$, \textit{i.e.} $I_{\H^+}(\chi)$ is irreducible. It is generic with respect to both $\psi$ and $\psi_-$. By the last case of \cite[Theorem 6.21(3)]{morimoto}, $\theta(\sigma)$ is the unique irreducible unramified quotient of
\[  \mathrm{Ind}^{\G'(F)}_{B_{\G'}(F)}\left(\chi_1^{-1}\chi_{E/F},\chi_2^{-1}\chi_{E/F},(\chi_1\chi_2\chi_0)\circ N_{E/F}\right).\]
 Using the exceptional isomorphism  $j:\mathbf{P}\G(F) \simeq \mathbf{P}\G'(F)$ and Lemma \ref{changechar}, this principal series can be written as $I_\G(\xi)$, with
 \begin{align*}
    \xi_1 &= \chi_0 \circ N_{E/F},\\
    \xi_2 &= (\chi_2 \chi_0) \circ N_{E/F},\\
    \xi_0&= \chi_1 \chi_{E/F}.
\end{align*}
Moreover, if we assume that $\theta(\sigma)$ is generic, by \cite[Theorem 2.7]{LiUPSpadic}, $\Theta(\sigma)=I_\G(\xi)$ is irreducible (and hence equal to $\theta(\sigma)$). 
 \item[(3a)] Suppose that $I_{\H^+}(\chi) = \sigma_+ \oplus \sigma_-$, with $\chi_2 = \chi_{E/F}$. By the case of \cite[Lemma 4.2 3(a)]{morimoto} in \cite[Theorem 6.21(3)]{morimoto}, $\theta(\sigma_-)=0$, while $\theta(\sigma_+)$ is the unique irreducible unramified quotient of $I_\G(\xi)$, with $\xi = (\chi_0 \circ N_{E/F},\chi_0 \circ N_{E/F},\chi_1 \chi_{E/F})$. As in the previous case, if $\theta(\sigma_+)$ is generic then \cite[Theorem 2.7]{LiUPSpadic} implies that $\Theta(\sigma_+)=I_\G(\xi)$ is irreducible and so equal to $\theta(\sigma_+)$.
 \item[(3b)] Suppose that $I_{\H^+}(\chi) = \sigma_+ \oplus \sigma_-$, with $\chi_1 = \chi_{E/F}$ and $\chi_2 = \mathbf{1}$. By the case of \cite[Lemma 4.2 3(b)]{morimoto} in \cite[Theorem 6.21(3)]{morimoto}, we have that, if $\xi_1 = \xi_2 =  \chi_0 \circ N_{E/F},\, \xi_0 = \mathbf{1},$ 
 \begin{equation}\label{decompositionPS}
     I_\G(\xi) = I_\G(\xi)_{\rm gen} \oplus I_\G(\xi)_{\rm ng}
 \end{equation}
 with $I_\G(\xi)_{\rm gen}$ (resp. $I_\G(\xi)_{\rm ng}$) denoting the generic (resp. non-generic) constituent of $I_\G(\xi)$, and 
 \begin{align*}
     \Theta(\sigma_+) &= \theta(\sigma_+) = I_\G(\xi)_{\rm gen}, \\ 
     \Theta(\sigma_-) &= \theta(\sigma_-) = I_\G(\xi)_{\rm ng}.
 \end{align*}
 Notice that $I_\G(\xi)_{\rm gen}$ is not unramified as, by \cite[Theorem 2.7]{LiUPSpadic}, it would contradict \eqref{decompositionPS} otherwise.
\end{itemize}
\end{proof}

In the next section, we relate the uniqueness of the Shalika functional for $\theta(\sigma)$ to the irreducibility of the the big theta lift $\Theta(\theta(\sigma))$, which is calculated explicitly in the proof of \cite[Theorem 6.21]{morimoto}. In the following Lemma, we summarize the result for the cases relevant to us.

\begin{lemma}\label{lemThetatheta}
    Let $\sigma$ be the irreducible $\psi$-generic factor of an irreducible unramified principal series $I_{\H}(\chi)$ of $\H(F)$, with trivial central character. Suppose that $\theta(\sigma)$ is generic and unramified. Then $\Theta(\theta(\sigma)) = \sigma$.
\end{lemma}
\begin{proof}
    According to the classification of the irreducible non-discrete series representations given in \cite[Lemma 4.2]{morimoto},  Proposition \ref{OnthetaliftsofPS} described the local theta correspondence from $\mathbf{P}\H(F)$ to $\mathbf{P}\G'(F)$ for representations falling in cases $3(a)$, $3(b)$ and $3(d)$ of \textit{loc.cit}, with case $3(c)$ excluded because of the irreducibility of $I_\H(\chi)$. By the hypothesis we impose here on $\theta(\sigma)$, the Proposition above rules out the case $3(b)$ and shows that $\theta(\sigma)$ is an irreducible principal series in the remaining two cases. Before we go on, notice the following:  
    \begin{enumerate}
        \item Via the exceptional isomorphism \[\GSO_{3,1}(F)\simeq \left(\GL_2(E)\times  F^\times \right) / \{(a^{-1},N_{E/F}(a))\,:\, a\in E^\times\},\] 
        an irreducible representation of $\GSO_{3,1}(F)$ can be denoted by $\pi(\rho,\varphi)$ for an irreducible representation $\rho$ of $\GL_2(E)$ and a character $\chi$ of $F^\times$ such that the central character of $\rho$ equals to $\varphi \circ N_{E/F}$. Then, if $\tau$ and $\varphi$ are characters of $E^\times$ and $F^\times$ respectively,  we have the equality 
        \[\pi\left ( {\rm Ind}_{B_{\GL_2}(E)}^{\GL_2(E)}\left ( \tau,\tau^\sigma \cdot (\varphi \circ N_{E/F}) \right) , \varphi \cdot  \tau_{|_{F^\times}} \right) = {\rm Ind}_{B_{\GSO_{3,1}}(F)}^{\GSO_{3,1}(F)}( \varphi , \tau), \]
        with $\tau^\sigma (a) = \tau(\bar{a})$ (\textit{cf.} \cite[p. 51]{morimoto}).
        \item If we let $P_{\G'}$ be the standard parabolic of $\G'$ with Levi $\GSO_{3,1} \times \GL_1$, combining (1) with \cite[Theorem 6.21]{morimoto}, we have
        \begin{align}\label{indstageQ}\theta(\sigma) = \mathrm{Ind}_{P_{\G'}(F)}^{{\G'}(F)}\left(\chi_1^{-1}\chi_{E/F}, \pi \left ( ((\chi_1\chi_2\chi_0)\circ N_{E/F}) \cdot  \mathrm{Ind}_{B_{\GL_2}(E)}^{\GL_2(E)}(1, \chi_2^{-1} \circ N_{E/F} ), \chi_1 \chi_{E/F} \right ) \right).\end{align}
Indeed, this follows by doing induction in stages and by the fact that $\theta(\sigma)$ is equal to \[  \mathrm{Ind}^{\G'(F)}_{B_{\G'}(F)}\left(\chi_1^{-1}\chi_{E/F},\chi_2^{-1}\chi_{E/F},(\chi_1\chi_2\chi_0)\circ N_{E/F}\right).\]
\item The representation $\pi \left ( ((\chi_1\chi_2\chi_0)\circ N_{E/F}) \cdot  \mathrm{Ind}_{B_{\GL_2}(E)}^{\GL_2(E)}(1, \chi_2^{-1} \circ N_{E/F} ), \chi_1 \chi_{E/F} \right )$, appearing in \eqref{indstageQ}, is in the image of the theta lift from $\GL_2^+(F)$. Indeed, by \cite[Corollary 6.11 \& Lemma 6.12]{morimoto}, it is equal to $\theta(\sigma_0^+)$ with $\sigma_0^+$ the $\psi_2$-generic irreducible component of
\[\sigma_0 = {\rm Ind}_{B_{\GL_2}(F)}^{\GL_2(F)}(\chi_1 \chi_2 \chi_0 , \chi_1 \chi_0  ). \]
\end{enumerate}
We can now calculate the big theta lift of $\theta(\sigma)$. Using \cite[p. 75 (1)]{morimoto}, we get 
        \[\Theta(\theta(\sigma))^* = \mathrm{Hom}_{\G'(F)}\left(\Omega,\theta(\sigma)\right),\]
 and the latter space has been studied in \cite[Propositions 6.19 \& 6.20]{morimoto} and calculated explicitly in the proof of Proposition 6.21 of \textit{loc.cit}. for the cases of interest here, \textit{i.e.} 3$(a)$ and 3$(d)$ (see p. 80 of \cite{morimoto}). For the convenience of the reader, we sketch the calculation.
Suppose that both $\chi_1 = \chi_2 =1$. Then, because of the trivial central character condition, $\chi_0$ is quadratic. As, by hypothesis, it is also unramified, $\chi_0$ is forcely $\chi_{E/F}$. In this case, we now use \cite[Proposition 6.20]{morimoto} to deduce the result. Firstly, if $Q_{\G'}$ is the standard parabolic of $\G'$ with Levi $\GL_2 \times {\rm Res}_{E/F}(\GL_1)$, we can write $\theta(\sigma)$ as 
\begin{align}\label{indstageP}\mathrm{Ind}_{Q_{\G'}(F)}^{{\G'}(F)}\left( \chi_{E/F} \cdot  \mathrm{Ind}_{B_{\GL_2}(F)}^{\GL_2(F)}(1, 1), 1 \right).\end{align}
Using \eqref{indstageP} and \cite[Proposition 6.20]{morimoto}, we get an injection $\Theta(\theta(\sigma))^* \hookrightarrow \sigma^*$, which induces an isomorphism on the contragradients $\Theta(\theta(\sigma))^\vee \simeq \sigma^\vee$ as $\sigma$ is irreducible. Suppose that either $\chi_1$ or $\chi_2$ is not trivial. Without loss of generality, up to using the Weyl action, we assume $\chi_1 \ne 1$. By \cite[Proposition 6.19(1)]{morimoto} and point (3) above, we have that 
 \[ \Theta(\theta(\sigma)) = {\rm Ind}_{Q_{\H^+}}^{\H^+(F)}(\chi_1,\chi_0 \cdot {\rm Ind}_{B_{\GL_2}(F)}^{\GL_2(F)}(\chi_2, 1 )^+), \]
which is equal to $\sigma$ as explained in \S \ref{genunrphplus}.
\end{proof}

\begin{remark}
    Note that if we relax the hypotheses of Lemma \ref{lemThetatheta} by allowing either $\sigma$ or $\theta(\sigma)$ to be non-generic, then $\Theta(\theta(\sigma))$ is not necessarily $\sigma$.
\end{remark}

\subsection{Mackey Theory}\label{MackeyTheory}

Recall that $E$ denotes the unique unramified quadratic field extension of $F$. In this section,   we study the space of Shalika functionals of a generic unramified representation $\Pi$ of $\mathbf{PG}(F)$ by means of Mackey theory, which involves the study of the restriction of $\Pi$ to the Shalika subgroup $S(F)$.

\subsubsection{An auxiliary computation}

In order to study the restriction of the representation $\Pi$ of $\mathbf{PG}(F)$ to $S(F)$ we need to consider the double quotient $S(F) \setminus \G(F)/B_\G(F)$. We now calculate it explicitly, proving it to be finite. In particular, we show that the group $S(F)$ acts on the flag variety $\G(F) /B_\G(F)$ with an open orbit. 

\begin{lemma}\label{lemma:decomposition}
We have the following decomposition: \[ \G(F)= \bigsqcup_{w \in{W_{\delta}}} \bigsqcup_{\tilde{w} \in {}^MW_\H} S(F) w \tilde{w} B_\G(F),\]
where $W_{\delta} := \left\{{\rm id},w_\delta = \left(\begin{smallmatrix}1 & & & \\ \delta &1 & &\\ & & 1& \\ & & -\overline{\delta} & 1\end{smallmatrix}\right)\right\}$ and  ${}^MW_\H := \{id, s_2, s_2s_1,s_2 s_1 s_2\}$, with $s_1 =(1\,2),\, s_2 =(2\,3)$ introduced in \S \ref{ss:Weylgroup}, is the set of minimal length representatives (Kostant representatives) of $W_{M_\H} \backslash W_\H$.
\end{lemma}
\begin{proof}
Using the Bruhat decomposition, we have that 
\[\G(F) =  \bigsqcup_{\tilde{w} \in \tilde{W}}P_\G(F) \tilde{w} B_\G(F),\]
where $\tilde{W}$ is the set of Kostant representatives of  $W_{M_\G} \backslash W_{\G}$, with $W_{M_\G}$ the Weyl group associated to the Levi of the Siegel subgroup $P_\G$. Since the root datum of $\H_{/F}$ equals to the relative root datum of $\G_{/F}$, the disjoint union above is indexed by the elements in the set of Kostant representatives for the quotient $ W_{M_\H} \backslash W_{\H}$. Thus any $g \in \G(F)$ can be expressed as $m \cdot u \cdot \tilde{w} \cdot b$ for a unique $\tilde{w} \in  {}^MW_\H$ and some $b \in B_\G(F)$,  $m \cdot u =\left(\begin{smallmatrix} h & \\  & \nu J {}^t\bar{h}^{-1} J\end{smallmatrix}\right) \left(\begin{smallmatrix}I & X \\ & I\end{smallmatrix}\right)  \in P_\G(F)$. Write $m$ as the product of $m' \lambda_\nu  = \left(\begin{smallmatrix} h & \\ &  J {}^t\bar{h}^{-1} J \end{smallmatrix}\right)\left(\begin{smallmatrix} I & \\ & \nu I \end{smallmatrix}\right)$, then we have that \begin{align*}
    g &= m' \lambda_\nu u \tilde{w} b = m'\lambda_\nu u \lambda_\nu^{-1} \lambda_\nu \tilde{w} b = m' u' \lambda_\nu \tilde{w} b = m' u' \tilde{w} \lambda_\nu' b \in \GL_2(E) N_\G(F) \tilde{w} B_\G(F),
\end{align*}
where we have denoted $u' =\lambda_\nu u \lambda_\nu^{-1} \in N_\G(F)$ and $ \lambda_\nu' = \tilde{w} \lambda_\nu \tilde{w}^{-1} \in T_\G(F)$. This implies that \begin{align*}
    S(F) \backslash \G(F) /B_\G(F) &= \bigsqcup_{\tilde{w} \in {}^MW_\H}  \GL_2(F)  N_\G(F)\backslash \GL_2(E) N_\G(F) \tilde{w} B_\G(F)/B_\G(F).
\end{align*}
Finally, recall that $\GL_2(F)$ acts on $\mathbb{P}^1_{E} = \GL_2(E) / B_{\GL_2}(E)$ with two orbits, one closed and one open with generators given by the identity and $ \left(\begin{smallmatrix} 1 & \\ \delta & 1\end{smallmatrix}\right)$  respectively. Since any $\tilde{w}$ is defined by the condition $\tilde{w}^{-1} \Phi_{M_\G}^+ \subset \Phi^+$, i.e. $\tilde{w}^{-1} B_{\GL_2}(E)\tilde{w} \subset  B_\G(F)$, we use this fact to deduce that 
\[\G(F)  =\bigsqcup_{w \in  W_\delta } \bigsqcup_{\tilde{w} \in {}^MW_\H}  S(F) w \tilde{w}B_\G(F). \]
\end{proof}

Following Lemma \ref{lemma:decomposition}, we now calculate the stabiliser of each orbit for the action of $S(F)$ on the flag variety $\G(F) /B_\G(F)$.  Before doing so, we denote by $T_\delta$ the maximal non-split torus of $\GL_2(F)$ given by \[T_\delta := \{ \left(\begin{smallmatrix}a & b \\b \delta^2 & a\end{smallmatrix}\right)\,:\, (a,b) \in F^2 \smallsetminus \{(0,0) \}\}, \]
which we see inside $S(F)$ via the embedding of $\GL_2(F) \hookrightarrow S(F)$.
\begin{proposition}\label{Prop:stabilisers}
The group $S(F)$ acts on the flag $\G(F) /B_\G(F)$, with an open orbit $\mathcal{O}_{w_\delta s_2 s_1 s_2}$ and seven closed ones, where \begin{itemize}
    \item $\mathcal{O}_{\rm id} = {\rm Stab}_{id} \backslash S(F)$, with ${\rm Stab}_{\rm id} = B_{\GL_2}(F) N_\G(F)$.   
    \item $\mathcal{O}_{s_2} = {\rm Stab}_{s_2} \backslash S(F)$, with \[{\rm Stab}_{s_2} =B_{\GL_2}(F) \left \{\left( \begin{smallmatrix} 1 & & \alpha & x \\ & 1 &  & \bar{\alpha} \\ &  & 1  & \\ & & & 1   \end{smallmatrix}\right)\,:\, x \in F, \alpha \in E \right \} . \]
    \item $\mathcal{O}_{s_2 s_1} = {\rm Stab}_{s_2 s_1} \backslash S(F)$, with \[{\rm Stab}_{s_2 s_1} = B_{\GL_2}(F) \left \{\left( \begin{smallmatrix} 1 & &  & x \\ & 1 &  &  \\ &  & 1  & \\ & & & 1   \end{smallmatrix}\right)\,:\, x \in F \right \}. \]
    \item $\mathcal{O}_{s_2 s_1 s_2} = {\rm Stab}_{s_2 s_1s_2} \backslash S(F)$, with ${\rm Stab}_{s_2 s_1 s_2} = B_{\GL_2}(F) $. 
    \item $\mathcal{O}_{w_\delta} = {\rm Stab}_{w_\delta} \backslash S(F)$, with ${\rm Stab}_{w_\delta} = T_\delta N_\G(F)$.  
    \item $\mathcal{O}_{w_\delta s_2} = {\rm Stab}_{w_\delta s_2} \backslash S(F)$, with \[{\rm Stab}_{w_\delta s_2} = T_\delta \left\{ \left(\begin{smallmatrix} 1 & & \alpha & x \\   &  1  & \delta \alpha + \bar{\delta} \bar{\alpha} - \delta \bar{\delta} x  & \bar{\alpha} \\ & & 1 &  \\ & &  & 1 \end{smallmatrix}\right)\,:\, x \in F, \alpha \in E  \right \}. \]
    \item $\mathcal{O}_{w_\delta s_2 s_1} = {\rm Stab}_{w_\delta s_2 s_1} \backslash S(F)$, with \[{\rm Stab}_{w_\delta s_2 s_1} = T_\delta \left\{ \left(\begin{smallmatrix} 1 & & \bar{\delta} x & x \\   &  1  & \delta \bar{\delta} x  & \delta x \\ & & 1 &  \\ & &  & 1 \end{smallmatrix}\right)\,:\, x \in F \right \}. \]
    \item $\mathcal{O}_{w_\delta s_2 s_1 s_2} = {\rm Stab}_{w_\delta s_2 s_1s_2} \backslash S(F)$, with ${\rm Stab}_{w_\delta s_2 s_1 s_2} = T_\delta$.
\end{itemize}
\end{proposition} 

\begin{proof}
By Lemma \ref{lemma:decomposition}, a set of representatives for the $S(F)$-orbits of $\G(F) /B_\G(F)$ is given by  $\{ w\tilde{w} \}$ with  $w$ and $\tilde{w}$ varying in $W_\delta$ and ${}^MW_\H$ respectively. The stabilizer of each element $v$ of this set is given by $\mathrm{Stab}_{v} := S(F) \cap v B_\G(F) v^{-1}$. We start by calculating the stabiliser of each element  $\tilde{w} \in {}^MW_\H$. Since $\Phi_{M_\G}^+ \subset \tilde{w} \Phi^+$, we have $\mathrm{Stab}_{\tilde{w}} = B_{\GL_2}(F) (N_\G \cap \tilde{w} U_\G \tilde{w}^{-1})(F)$, where recall that $U_\G$ denotes the unipotent radical of the upper triangular Borel parabolic of $\G$ and $B_{\GL_2}$ embeds into $S$ via the map $b \mapsto \left( \begin{smallmatrix} b & \\ & {\rm det}(b)J {}^tb^{-1} J  \end{smallmatrix} \right)$. Therefore  $\mathrm{Stab}_{\rm id} = B_{\GL_2}(F) N_\G(F)$. Since $s_2$ sends the $\mathrm{Sp}_4$-positive roots $\Phi^+_{\mathrm{Sp}_4}=\{\alpha_1 - \alpha_2, \alpha_1 + \alpha_2, 2 \alpha_1 , 2 \alpha_2 \}$ to $\{\alpha_1 +\alpha_2,\alpha_1 -\alpha_2, 2\alpha_1, - 2\alpha_2 \}$, $\mathrm{Stab}_{s_2} = B_{\GL_2}(F) \left \{\left( \begin{smallmatrix} 1 & & \alpha & x \\ & 1 &  & \bar{\alpha} \\ &  & 1  & \\ & & & 1   \end{smallmatrix}\right)\,:\, x \in F, \alpha \in E \right \}$. Similarly, $s_2s_1$, resp.  $s_2s_1s_2$, sends $\Phi^+_{\mathrm{Sp}_4}$ to $\{-\alpha_1-\alpha_2,\alpha_1 -\alpha_2, -2\alpha_2,  2\alpha_1 \}$, resp. to $\{\alpha_1-\alpha_2, -\alpha_1 -\alpha_2, -2\alpha_2, - 2\alpha_1 \}$, hence $\mathrm{Stab}_{s_2 s_1} = B_{\GL_2}(F) \left \{\left( \begin{smallmatrix} 1 & &  & x \\ & 1 &  &  \\ &  & 1  & \\ & & & 1   \end{smallmatrix}\right)\,:\, x \in F \right \},$ while $\mathrm{Stab}_{s_2 s_1 s_2} = B_{\GL_2}(F)$. We are now left with calculating the stabiliser of the orbits associated to the elements $w_\delta \tilde{w}$. Conjugating by $w_\delta^{-1}$ gives that  $w_\delta^{-1} \mathrm{Stab}_{w_\delta \tilde{w}} w_\delta =  w_\delta^{-1} S(F)w_\delta \cap \tilde{w} B_\G(F) \tilde{w}^{-1}$. Given $g = m u =\left(\begin{smallmatrix}a&b&&\\c&d&& \\ & &a&-b\\ & &-c&d\end{smallmatrix}\right) \left(\begin{smallmatrix} 1& &\alpha & x\\& 1&y&\overline{\alpha}\\ & & 1& \\ & & &1 \end{smallmatrix}\right)$, \[w_\delta^{-1} g w_\delta = w_\delta^{-1} m w_\delta w_\delta^{-1} u w_\delta = \left(\begin{smallmatrix}a+b \delta &b & & \\c' & d - b \delta & & \\ & & a+b \bar{\delta} & -b \\ & & -\bar{c}' & d -\bar{\delta} b \end{smallmatrix}\right) \left(\begin{smallmatrix} 1& & \alpha - \bar{\delta} x & x \\ & 1 & y' &\overline{\alpha} - \delta x \\ & & 1& \\ & & &1 \end{smallmatrix}\right), \]
with $c' = c - b \delta^2 + \delta (d - a)$ and $y' = y + \delta \bar{\delta} x - \delta \alpha - \bar{\delta} \bar{ \alpha}$. Using again that $\Phi_{M_\G}^+ \subset \tilde{w} \Phi^+$, $w_\delta^{-1} m w_\delta \in \tilde{w} B_\G(F) \tilde{w}^{-1}$ if $c'= 0$, i.e. $c = b \delta^2 + \delta (a - d)$.
Since $c \in F$, this further implies that $a-d = 0$, hence $m$ equals to $\left(\begin{smallmatrix}a & b & & \\b \delta^2 & a && \\ & & a & -b\\ & &-b \delta^2 & a\end{smallmatrix}\right)$. When $\tilde{w} = id$,  $w_\delta^{-1} u w_\delta \in B_\G(F)$, which implies that \[{\rm Stab}_{w_\delta} = \left\{ \left(\begin{smallmatrix}a & b& &\\b \delta^2 & a && \\ & & a & -b\\ & &-b \delta^2 & a\end{smallmatrix} \right)\,:\, \left(\begin{smallmatrix}a & b \\b \delta^2 & a\end{smallmatrix}\right) \in \GL_2(F) \right \} N_\G(F) = T_\delta  N_\G(F). \]
When $\tilde{w} = s_2$, $w_\delta^{-1} u w_\delta \in \tilde{w} U_\G \tilde{w}^{-1}$ if $y'=0$, because $N_\G \cap \tilde{w} U_\G \tilde{w}^{-1} =  \left(\begin{smallmatrix} 1 & & \star & \star  \\  & 1 & & \star \\ & & 1 & \\ & &  & 1 \end{smallmatrix} \right)$. Thus \[{\rm Stab}_{w_\delta s_2} = T_\delta \left\{ \left(\begin{smallmatrix} 1 & & \alpha & x \\   &  1  & \delta \alpha + \bar{\delta} \bar{\alpha} - \delta \bar{\delta} x  & \bar{\alpha} \\ & & 1 &  \\ & &  & 1 \end{smallmatrix}\right)  \right \}. \]
Finally, when $\tilde{w} = s_2 s_1$, resp.  $\tilde{w} = s_2 s_1 s_2$, $N_\G \cap \tilde{w} U_\G \tilde{w}^{-1}$ equals to  $\left(\begin{smallmatrix} 1 & &  & \star  \\  & 1 & &  \\ & & 1 & \\ & &  & 1 \end{smallmatrix} \right)$, resp. $\{I_4\}$. From this, one deduces the last two cases.
\end{proof}

\subsubsection{Mackey theory}\label{sss:Mackeytheory}

Following the same ordering as in Proposition \ref{Prop:stabilisers}, we denote by $\{w_i\}_{i \in I}$, with $I = \{1,...,7\}$, the set of representatives of the closed orbits, each with stabiliser $\mathrm{Stab}_i$. Moreover, we let $w_{\rm op}$ be the representative $w_\delta s_2s_1 s_2$ of the open orbit with stabiliser $T_\delta$. Recall the following result.

\begin{prop}[{\cite[Lemma 5.1]{Prasad}}]\label{exactsequencePrasad}
Let $\mathscr{V}$ be a complex vector bundle on a locally compact totally disconnected topological space $X$ and let $Z$ be a closed subspace. Let $\Gamma_c(X,\mathscr{V})$ denote the space of locally constant compactly supported sections of $\mathscr{V}$. Then we have the exact sequence
\[0\to \Gamma_c(X-Z,\mathscr{V}_{|_{X-Z}}) \to \Gamma_c(X,\mathscr{V})\to \Gamma_c(Z,\mathscr{V}_{|_Z})\to 0.\]
\end{prop}

\begin{remark}
Let $X$ be a l.c.t.d. topological group and let $Y$ be a closed subgroup of $X$. If $(V,\rho)$ is a smooth representation of $Y$, then \[\Gamma_c(X/Y,\mathscr{V}) \simeq \text{c-Ind}_Y^X(\rho), \]
where $\mathscr{V}$ denotes the vector bundle associated to $V$ and $\text{c-Ind}_Y^X(\rho)$ denotes the space of smooth functions $f: X \to V$ which are compactly supported modulo $Y$ and satisfy the following condition: \[ f(yx) = \delta_X^{-1/2}(y)\delta_Y^{1/2}(y) \rho(y) f(x)\, \text{for every } y \in Y, x\in X. \]
\end{remark}

Let now $\Pi$ be an irreducible generic unramified representation of $\G(F)$ with trivial central character. Note that, as remarked in \S 3.1, $\Pi$ is necessarily a principal series $I_\G(\xi)$, with $\xi$ an unramified character of $T_\G(F)$ such that \[\xi_1(\alpha)\xi_2(\alpha)\xi_0(\alpha \bar{\alpha})=1,\, \text{ for every }\alpha \in E^\times.\] 
We apply Proposition \ref{exactsequencePrasad} for $X = \G(F) / B_\G(F)$, and the representation $(V,\xi)$ of $T_\G(F)$ given by the character $\xi$. Then, if  $Z = X - \mathcal{O}_{w_{\rm op}}$, we have an exact sequence of $S(F)$-modules, with $S(F)$ acting on the right, \begin{equation}\label{ourexact} 0\to \text{c-}\mathrm{Ind}_{T_\delta}^{S(F)} \xi_{\rm op} \to I_\G(\xi) \to \Gamma_c\left(\bigsqcup_{i\in I}\mathcal{O}_i, \mathscr{V}_{|_{\bigsqcup_{i\in I}\mathcal{O}_i}}\right)\to 0,\end{equation}
with $\xi_{\rm op}$ the representation of $T_\delta$ given by  \[ \xi_{\rm op}(g) = \delta_{B_\G}^{1/2}(w_{\rm op}^{-1} g w_{\rm op}) \delta_{T_\delta}^{-1/2}(g) \xi(w_{\rm op}^{-1} g w_{\rm op}) = \delta_{B_\G}^{1/2}\xi(w_{\rm op}^{-1} g w_{\rm op}), \]
as $T_\delta$ is unimodular. Notice that the restriction of $\mathscr{V}$ to each orbit $\mathcal{O}_i$ can be identified with $\text{c-}\mathrm{Ind}_{\mathrm{Stab}_i}^{S(F)} \xi_{w_i}$, where 
\[\xi_{w_i}(g) := \delta_{B_\G}^{1/2}(w_{i}^{-1}g w_{i})\delta^{-1/2}_{{ \rm Stab}_{w_{i}}}(g)\xi(w_{i}^{-1} g w_{i}).\]

Let $\chi_{S}$ be the character on $S(F)$ introduced in \S \ref{ss:Shalikamodels}. Applying the functor $\mathrm{Hom}_{S(F)}(- ,\chi_{S})$ to the exact sequence \eqref{ourexact} we obtain

\begin{small}
\begin{align}\label{exactsequence} 0 \to \mathrm{Hom}_{S(F)}\left(\Gamma_c\left(\bigsqcup_{i\in I}\mathcal{O}_i, \mathscr{V}_{|_{\bigsqcup_{i\in I}\mathcal{O}_i}}\right),\chi_{S}\right) \to \mathrm{Hom}_{S(F)}\left(I_\G(\xi),\chi_{S}\right)\to \mathrm{Hom}_{S(F)} \left(\text{c-}\mathrm{Ind}_{T_\delta}^{{S(F)}} \xi_{\rm op},\chi_{S}\right)\to  \end{align} 
\[\to \mathrm{Ext}^1_{S(F)}\left(\Gamma_c\left(\bigsqcup_{i\in I}\mathcal{O}_i, \mathscr{V}_{|_{\bigsqcup_{i\in I}\mathcal{O}_i}}\right),\chi_{S}\right) \to \mathrm{Ext}^1_{S(F)}\left(I_\G(\xi),\chi_{S}\right)\to \mathrm{Ext}^1_{S(F)}\left(\text{c-}\mathrm{Ind}_{T_\delta}^{{S(F)}} \xi_{\rm op},\chi_{S}\right) \to...\]
\end{small}
We use this exact sequence to study the space of Shalika functionals $\mathrm{Hom}_{S(F)}\left(I_\G(\xi),\chi_{S}\right)$. To do so, we first study the contributions of the open and closed orbits separately.

\begin{lemma}\label{propopen} 
Let $\xi$ be an unramified character of $T_\G(F)$ such that $\xi_1(\alpha)\xi_2(\alpha)\xi_0(\alpha \bar{\alpha})=1$ for every $\alpha \in E^\times$. Then
\[{\rm dim}\, \mathrm{Hom}_{S(F)} \left(\text{c-}\mathrm{Ind}_{T_\delta}^{{S(F)}} \xi_{\rm op},\chi_{S}\right) =1 .\]
\end{lemma}

\begin{proof}
By Frobenius reciprocity 
\[\mathrm{Hom}_{S(F)} \left(\text{c-}\mathrm{Ind}_{T_\delta}^{{S(F)}} \xi_{\rm op},\chi_{S}\right) \simeq  \mathrm{Hom}_{T_\delta} \left( \delta_{S}^{1/2} \xi_{\rm op}, \mathbf{1} \right).\]
We now explicit $\delta_{S}^{1/2} \xi_{\rm op}(g) = \delta_{S}^{1/2}(g)\delta_{B_\G}^{1/2}\xi(w_{\rm op}^{-1} g w_{\rm op})$. Firstly, recall that \[s_2 = \left(\begin{smallmatrix}1& & & \\ & & 1& \\ &  -1& & \\ & & &1\end{smallmatrix}\right),\; s_1 = \left(\begin{smallmatrix} &1& & \\1& & & \\ & & & 1\\ & & 1& \end{smallmatrix}\right),\; w_\delta = \left(\begin{smallmatrix}1 & & & \\ \delta &1 & &\\ & & 1& \\ & & \delta & 1\end{smallmatrix}\right).\]
Hence $w_{\rm op } = \left(\begin{smallmatrix} & & 1 & \\ & & \delta & 1 \\- 1 & & & \\ -\delta & - 1 & & \end{smallmatrix}\right)$ and, if $g=\left(\begin{smallmatrix}a & b& &\\b \delta^2 & a && \\ & & a & -b\\ & &-b \delta^2 & a\end{smallmatrix} \right) \in T_\delta$, we have $w_{\rm op }^{-1} g w_{\rm op }= \left(\begin{smallmatrix} a - b \delta & - b & &\\ & a + b \delta  & & \\ & & a + b \delta  & b\\ & &  & a - b \delta \end{smallmatrix} \right) $.
A simple check shows that $S(F)$ is unimodular, hence $\delta_S(g)=1$; also one checks explicitly that $\delta_{B_\G}(w_{\rm op}^{-1} g w_{\rm op})=1$. This implies that the space $\mathrm{Hom}_{T_\delta} \left( \delta_{S}^{1/2} \xi_{\rm op}, \mathbf{1} \right)$ is trivial unless $\xi(w_{\rm op}^{-1} g w_{\rm op}) = 1$ for all $g \in T_\delta$, in which case it is 1 dimensional. This condition translates into asking that $\xi_1(\alpha)\xi_2(\bar{\alpha}) \xi_0(\alpha \bar{\alpha}) = 1$ for all $\alpha \in E^\times$. As $\xi_2$ is unramified, it factors through $N_{E/F}$, hence this condition is equivalent to the trivial central character condition $\xi_1(\alpha)\xi_2(\alpha) \xi_0(\alpha \bar{\alpha}) = 1$.
\end{proof}

\begin{lemma}\label{lemmaclosedorbits}
We have the following: \leavevmode \begin{itemize}
    \item $\mathrm{Hom}_{S(F)}(\text{c-}\mathrm{Ind}_{{ \rm Stab}_{i}}^{S(F)}\xi_{w_i},\chi_{S})$ is trivial if $i =1,2,5,6,7$; 
    \item ${\rm dim}\, \mathrm{Hom}_{S(F)}(\text{c-}\mathrm{Ind}_{{ \rm Stab}_{3}}^{S(F)}\xi_{w_3},\chi_{S}) \leq 1$ and it equals to one if and only if \[\xi_1(a) \xi_2(a) \xi_0(ad) = 1,\,\, \forall a,d \in F^\times.\]
    The latter is  equivalent to asking that $\xi_0 = 1$ and $\xi_1 = \xi_2^{-1}$.
    \item ${\rm dim}\, \mathrm{Hom}_{S(F)}(\text{c-}\mathrm{Ind}_{{ \rm Stab}_{4}}^{S(F)}\xi_{w_4},\chi_{S}) \leq 1$ and it equals to one if and only if \[\xi_1(a) \xi_2(d) \xi_0(ad) = 1,\,\, \forall a,d \in F^\times.\]
    This is equivalent to asking that $\xi_1|_{F^\times} = \xi_0^{-1}$ and $\xi_2(\alpha) = \xi_1(\bar{\alpha})$.
\end{itemize}  
\end{lemma}

\begin{proof}
By Frobenius reciprocity
\[\mathrm{Hom}_{S(F)}(\text{c-}\mathrm{Ind}_{{ \rm  Stab}_{i}}^{S(F)}\xi_{w_i},\chi_{S})  \simeq  \mathrm{Hom}_{{ \rm Stab}_{i}} \left( \delta_{S}^{1/2}\delta_{{ \rm Stab}_{i}}^{-1/2} \xi_{w_i}, {{\chi_{S}}}_{|_{{ \rm Stab}_{i}}} \right) =\mathrm{Hom}_{{ \rm Stab}_{i}} \left( \delta_{{ \rm Stab}_{i}}^{-1/2} \xi_{w_i}, {{\chi_{S}}}_{|_{{ \rm Stab}_{i}}} \right).\]
In the case where $w_i \in \{id, s_2, w_\delta, w_\delta s_2, w_\delta s_2 s_1 \}$, this space is trivial since ${{\chi_{S}}}_{|_{{ \rm Stab}_{i}}}$ is not trivial while $\delta_{{ \rm Stab}_{i}}^{-1/2} \xi_{w_i}$ is trivial on the unipotent part of ${ \rm Stab}_{i}$. We are left with examining the cases of $w_3 = s_2 s_1$ and $w_4 = s_2 s_1 s_2$. Let's start with $w_3$. In this case \[{\rm Stab}_{3} = B_{\GL_2}(F) \left \{\left( \begin{smallmatrix} 1 & &  & x \\ & 1 &  &  \\ &  & 1  & \\ & & & 1   \end{smallmatrix}\right)\,:\, x \in F \right \},\; \;\delta_{{\rm Stab}_{3}} \left ( \left( \begin{smallmatrix} a & b &  &  \\ & d &  &  \\ &  & a  & -b \\ & & & d   \end{smallmatrix}\right) \left( \begin{smallmatrix} 1 & &  & x \\ & 1 &  &  \\ &  & 1  & \\ & & & 1   \end{smallmatrix}\right)\right) =  \tfrac{|a|^2}{|d|^2},\] hence, if $g = \left( \begin{smallmatrix} a & b &  &  \\ & d &  &  \\ &  & a  & -b \\ & & & d   \end{smallmatrix}\right) \left( \begin{smallmatrix} 1 & &  & x \\ & 1 &  &  \\ &  & 1  & \\ & & & 1   \end{smallmatrix}\right)$, we have
\begin{align*}\delta_{{ \rm Stab}_{3}}^{-1/2} \xi_{w_3}(g) &= \delta^{-1}_{{ \rm Stab}_{3}}(g)\delta_{B_\G}^{1/2}\xi(w_{3}^{-1} g w_{3}) \\
&=\tfrac{|d|^2}{|a|^2}\cdot \delta_{B_\G}^{1/2}\xi \left( \left( \begin{smallmatrix} a & & b &  \\ & a & x & b \\  &  & d  & \\ &  & & d  \end{smallmatrix}\right)\right) \\ 
&=\tfrac{|d|^2}{|a|^2} \cdot \tfrac{|a|^2}{|d|^2} \xi_1(a) \xi_2(a) \xi_0(ad)
\\ 
&=\xi_1(a) \xi_2(a) \xi_0(ad).
\end{align*}
Therefore $\mathrm{Hom}_{{ \rm Stab}_{3}} \left( \delta_{{ \rm Stab}_{3}}^{-1/2} \xi_{w_3}, {{\chi_{S}}}_{|_{{ \rm Stab}_{3}}} \right)$ is trivial unless $\xi_1(a) \xi_2(a) \xi_0(ad) = 1\,\, \forall a,d \in F^\times$, in which case it is one-dimensional. Similarly, ${\rm Stab}_4 = B_{\GL_2}(F)$ and $\delta_{{\rm Stab}_4} \left ( \left( \begin{smallmatrix} a & b &  &  \\ & d &  &  \\ &  & a  & -b \\ & & & d   \end{smallmatrix}\right) \right) =  \tfrac{|a|}{|d|}$ which implies that
\begin{align*}
\delta_{{ \rm Stab}_4}^{-1/2} \xi_{w_4}(g) &= \delta^{-1}_{{ \rm Stab}_4}(g)\delta_{B_\G}^{1/2}\xi(w_4^{-1} g w_4 ) \\
&=\tfrac{|d|}{|a|}\cdot \delta_{B_\G}^{1/2}\xi \left( \begin{smallmatrix} a & -b &  &  \\ & d &  &  \\ &  & a  & b \\ & & & d   \end{smallmatrix}\right) \\ 
&=\tfrac{|d|}{|a|} \cdot \tfrac{|a|}{|d|} \xi_1(a) \xi_2(d) \xi_0(ad)
\\ 
&=\xi_1(a) \xi_2(d) \xi_0(ad).
\end{align*}
Hence $\mathrm{Hom}_{{ \rm Stab}_4} \left( \delta_{{ \rm Stab}_4}^{-1/2} \xi_{\omega_4}, {{\chi_{S}}}_{|_{{ \rm Stab}_4}} \right)$ is trivial unless $\xi_1(a) \xi_2(d) \xi_0(ad) = 1\,\, \forall a,d \in F^\times$, in which case it is one-dimensional.
\end{proof}

The following technical Lemma is an adaptation to the current setting of \cite[Proposition 5.9]{Prasad}.

\begin{lemma}\label{vanishhomGOOD}\leavevmode 
\begin{enumerate}
    \item Let $\xi$ be such that $ \mathrm{Hom}_{S(F)}(\text{c-}\mathrm{Ind}_{{ \rm Stab}_{i}}^{S(F)}\xi_{w_i},\chi_{S})=0$  for every $1 \leq i \leq 7$, 
where we recall that $\{w_i\}_{i\in I}$ is the set of representatives of the closed $S(F)$-orbits of $ \G(F) / B_\G(F)$. Then
\[ \mathrm{Hom}_{S(F)}\left(\Gamma_c\left(\bigsqcup_{i\in I}\mathcal{O}_i, \mathscr{V}_{|_{\bigsqcup_{i\in I}\mathcal{O}_i}}\right),\chi_{S}\right) = \mathrm{Ext}^1_{S(F)}\left(\Gamma_c\left(\bigsqcup_{i\in I}\mathcal{O}_i, \mathscr{V}_{|_{\bigsqcup_{i\in I}\mathcal{O}_i}}\right),\chi_{S}\right)=0.\]
\item Suppose that there exists a closed orbit $w_i$, with $1 \leq i \leq 7$, for which $ \mathrm{Hom}_{S(F)}(\text{c-}\mathrm{Ind}_{{ \rm Stab}_{i}}^{S(F)}\xi_{w_i},\chi_{S})$ is not trivial, then
\[ {\rm dim}\, \mathrm{Hom}_{S(F)}\left(\Gamma_c\left(\bigsqcup_{i\in I}\mathcal{O}_i, \mathscr{V}_{|_{\bigsqcup_{i\in I}\mathcal{O}_i}}\right),\chi_{S}\right) \geq 1.\]
\end{enumerate}

\end{lemma}
\begin{proof}
We start with verifying \textit{(1)}. We will prove the statement by applying Proposition \ref{exactsequencePrasad} inductively on the number of closed orbits. 

Firstly, let $X = \mathcal{O}_1\bigsqcup \mathcal{O}_2$, $Z = \mathcal{O}_1$ and let $\mathscr{V}|_{X}$ be the restriction to $X$ of the vector bundle  $\mathscr{V}$ associated to the representation $I_{\G}(\xi)$. By Proposition \ref{exactsequencePrasad}, we have the exact sequence 
\[0\to \text{c-}\mathrm{Ind}_{\mathrm{Stab}_1}^S\xi_{w_1}\to \Gamma_c\left(X, \mathscr{V}|_X\right)\to \text{c-}\mathrm{Ind}_{\mathrm{Stab}_{2}}^S\xi_{w_2}\to 0.\]
We apply the functor $\mathrm{Hom}_{S(F)}\left(\cdot,\chi_{S}\right)$ to this exact sequence, obtaining 
\begin{equation}\label{exactsecaux1}0\to \mathrm{Hom}_{S(F)}\left(\text{c-}\mathrm{Ind}_{\mathrm{Stab}_2}^S\xi_{w_2},\chi_{S}\right)\to \mathrm{Hom}_{S(F)}\left(\Gamma_c\left(X, \mathscr{V}|_{X}\right),\chi_{S}\right)\to \mathrm{Hom}_{S(F)}\left(\text{c-}\mathrm{Ind}_{\mathrm{Stab}_1}^{S}\xi_{w_1},\chi_{S}\right)\to...\end{equation}
By hypothesis, $\mathrm{Hom}_{S(F)}(\text{c-}\mathrm{Ind}_{{ \rm Stab}_{i}}^{S(F)}\xi_{w_i},\chi_{S})=0$ for every $i$, which shows that
\[\mathrm{Hom}_{S(F)}\left(\Gamma_c\left(X, \mathscr{V}|_{X}\right),\chi_{S}\right)= 0.\]
We now apply an inductive argument on the number of closed orbits; given $J = \{1,...,j\}\subset I$ such that $j+1\in I$, suppose that \begin{equation}\label{inducthyp} \mathrm{Hom}_{S(F)}\left(\Gamma_c\left(\bigsqcup_{i\in J}\mathcal{O}_i, \mathscr{V}|_{J}\right),\chi_{S}\right) = 0,\end{equation}
where $\mathscr{V}|_{J}$ denotes the restriction of $\mathscr{V}$ to $\bigsqcup_{i\in J}\mathcal{O}_i$. 
Applying Proposition \ref{exactsequencePrasad} with $X = \bigsqcup_{i \in J\cup \{j+1\}}\mathcal{O}_i$, $Z = \mathcal{O}_{j+1}$ and with the line bundle $\mathscr{V}|_{J\cup\{j+1\}}$ we get an analogous exact sequence to \eqref{exactsecaux1}. Using the inductive hypothesis \eqref{inducthyp} and our hypothesis on the vanishing of  $\mathrm{Hom}_{S(F)}(\text{c-}\mathrm{Ind}_{{ \rm Stab}_{j+1}}^{S(F)}\xi_{w_{j+1}},\chi_{S})$, we obtain that
\[\mathrm{Hom}_{S(F)}\left(\Gamma_c\left(\bigsqcup_{i \in J\cup\{j+1\} }\mathcal{O}_i, \mathscr{V}|_{J\cup\{j+1\}}\right),\chi_{S}\right) = 0.\]
By induction, this proves that \[ \mathrm{Hom}_{S(F)}\left(\Gamma_c\left(\bigsqcup_{i\in I}\mathcal{O}_i, \mathscr{V}_{|_{\bigsqcup_{i\in I}\mathcal{O}_i}}\right),\chi_{S}\right) = 0.\]

The vanishing of the $\mathrm{Ext}^1$ is proved similarly. Firstly, notice that $\mathrm{Hom}_{S(F)}(\text{c-}\mathrm{Ind}_{{ \rm Stab}_{i}}^{S(F)}\xi_{w_i},\chi_{S})=0$ if and only if $\mathrm{Ext}^1_{S(F)}(\text{c-}\mathrm{Ind}_{{ \rm Stab}_{i}}^{S(F)}\xi_{w_i},\chi_{S})=0$. This fact is proved in the exact same way as for \cite[Proposition 5.9]{Prasad}.
Then, the exact sequence \eqref{exactsecaux1} becomes
\[0\to \mathrm{Ext}^1_{S(F)}(\text{c-}\mathrm{Ind}_{{ \rm Stab}_{2}}^{S(F)}\xi_{w_2},\chi_{S})\to \mathrm{Ext}^1_{S(F)}\left(\Gamma_c\left(X, \mathscr{V}|_{X}\right),\chi_{S}\right)\to \mathrm{Ext}^1_{S(F)}\left(\text{c-}\mathrm{Ind}_{\mathrm{Stab}_1}^S\xi_1,\chi_{S}\right)\to...,\]
thus, by hypothesis, $\mathrm{Ext}^1_{S(F)}\left(\Gamma_c\left(X, \mathscr{V}|_{X}\right),\chi_{S}\right) = 0$. Using the same inductive argument as above we obtain \textit{(1)}.  

The proof of \textit{(2)} follows similarly. Indeed, let $w_i$ be a closed orbit for which $ \mathrm{Hom}_{S(F)}(\text{c-}\mathrm{Ind}_{{ \rm Stab}_{i}}^{S(F)}\xi_{w_i},\chi_{S})$ is not trivial. Applying Proposition \ref{exactsequencePrasad} with $X = \bigsqcup_{j \in I} \mathcal{O}_j$, $Z = \mathcal{O}_i$, and $\mathscr{V}|_{X}$ the restriction to $X$ of the vector bundle  $\mathscr{V}$ associated to the representation $I_{\G}(\xi)$, we have a surjection 
\[ \Gamma_c\left(X, \mathscr{V}|_X\right)\twoheadrightarrow\text{c-}\mathrm{Ind}_{\mathrm{Stab}_{i}}^S\xi_{w_i}.\]
Applying the functor $\mathrm{Hom}_{S(F)}\left(\cdot,\chi_{S}\right)$, we get an injection
\[\mathrm{Hom}_{S(F)}(\text{c-}\mathrm{Ind}_{{ \rm Stab}_{i}}^{S(F)}\xi_{w_i},\chi_{S})  \hookrightarrow \mathrm{Hom}_{S(F)}\left(\Gamma_c\left(X, \mathscr{V}|_X\right),\chi_{S}\right),\]
implying that the dimension of $\mathrm{Hom}_{S(F)}\left(\Gamma_c\left(X, \mathscr{V}|_X\right),\chi_{S}\right)$ is at least one, as desired.
\end{proof}

\begin{theorem}\label{FinalMackey}
Let $\Pi=I_\G(\xi)$ be an irreducible generic unramified representation of $\G(F)$ with trivial central character. Then $\Pi$ has at least one non-trivial Shalika functional. Moreover, $\Pi$ admits a unique non-trivial Shalika functional up to constant if none of the following holds: \begin{enumerate}
    \item $\xi_1(a)\xi_0(ad) = \xi_2^{-1}(a)$, $\forall a,d \in F^\times$;
     \item $\xi_1(a)\xi_0(ad) = \xi_2^{-1}(d)$, $\forall a,d \in F^\times$.
\end{enumerate}
\end{theorem}

\begin{proof}
Recall that we have the long exact sequence \eqref{exactsequence}

\begin{small}
\begin{align*} 0 \to \mathrm{Hom}_{S(F)}\left(\Gamma_c\left(\bigsqcup_{i\in I}\mathcal{O}_i, \mathscr{V}_{|_{\bigsqcup_{i\in I}\mathcal{O}_i}}\right),\chi_{S}\right) \to \mathrm{Hom}_{S(F)}\left(I_\G(\xi),\chi_{S}\right)\to \mathrm{Hom}_{S(F)} \left(\text{c-}\mathrm{Ind}_{T_\delta}^{{S(F)}} \xi_{\rm op},\chi_{S}\right)\to  \end{align*} 
\[\to \mathrm{Ext}^1_{S(F)}\left(\Gamma_c\left(\bigsqcup_{i\in I}\mathcal{O}_i, \mathscr{V}_{|_{\bigsqcup_{i\in I}\mathcal{O}_i}}\right),\chi_{S}\right) \to \mathrm{Ext}^1_{S(F)}\left(I_\G(\xi),\chi_{S}\right)\to \mathrm{Ext}^1_{S(F)}\left(\text{c-}\mathrm{Ind}_{T_\delta}^{{S(F)}} \xi_{\rm op},\chi_{S}\right) \to...\]
\end{small}
We divide the proof in two cases. \begin{enumerate}
    \item We first suppose that $\xi$ does not satisfy \textit{(1)} and \textit{(2)}. By Lemma \ref{lemmaclosedorbits}, we have that  
\[ \mathrm{Hom}_{S(F)}(\text{c-}\mathrm{Ind}_{{ \rm Stab}_{i}}^{S(F)}\xi_{w_i},\chi_{S})=0,\, \, \text{ for every }  1 \leq i \leq 7.  \]
Hence Lemma \ref{vanishhomGOOD} implies that \[ \mathrm{Hom}_{S(F)}\left(\Gamma_c\left(\bigsqcup_{i\in I}\mathcal{O}_i, \mathscr{V}_{|_{\bigsqcup_{i\in I}\mathcal{O}_i}}\right),\chi_{S}\right) = \mathrm{Ext}^1_{S(F)}\left(\Gamma_c\left(\bigsqcup_{i\in I}\mathcal{O}_i, \mathscr{V}_{|_{\bigsqcup_{i\in I}\mathcal{O}_i}}\right),\chi_{S}\right)=0.\]
Therefore, the short exact sequence gives an isomorphism \[ \mathrm{Hom}_{S(F)}\left(I_\G(\xi),\chi_{S}\right) \simeq \mathrm{Hom}_{S(F)} \left(\text{c-}\mathrm{Ind}_{T_\delta}^{{S(F)}} \xi_{\rm op},\chi_{S}\right).\]
By Lemma \ref{propopen}, the latter is exactly of dimension one as $\Pi$ has trivial central character. This shows both uniqueness and existence of Shalika functionals for unramified principal series whose character does not satisfy \textit{(1)} and \textit{(2)}.
\item Let $\xi$ be such that \textit{(1)} or \textit{(2)} are satisfied. By \eqref{exactsequence}, we have an injection

\begin{align*}  \mathrm{Hom}_{S(F)}\left(\Gamma_c\left(\bigsqcup_{i\in I}\mathcal{O}_i, \mathscr{V}_{|_{\bigsqcup_{i\in I}\mathcal{O}_i}}\right),\chi_{S}\right) \hookrightarrow \mathrm{Hom}_{S(F)}\left(I_\G(\xi),\chi_{S}\right).
\end{align*}
Then the result follows from Lemma \ref{vanishhomGOOD}.
\end{enumerate} 
\end{proof}

\begin{remark}
    Note that the existence of non-trivial Shalika functionals for any irreducible  generic unramified representation of ${\rm PGU}_{2,2}(F)$ was previously proved by Morimoto (see \cite[Theorems 6.9, 6.21, and 6.23]{morimoto}). 
\end{remark}

The existence of at least one non-trivial Shalika functional for $\Pi$ has the following important consequence:

\begin{theorem}\label{Shalikaimpliesimagetheta}
    Let $\Pi$ be a generic unramified representation of $\mathbf{PG}(F)$. There is a generic unramified representation $\sigma$ of $\mathbf{P}\H^+(F)$ so that $\Pi = \theta(\sigma)$.
\end{theorem}
\begin{proof}
By Theorem \ref{FinalMackey}, $\Pi$ has a non-trivial Shalika functional. Thus,  by \cite[Theorem 6.9]{morimoto}, $\theta(\Pi)$ is a non-zero generic representation of $\mathbf{P}\H^+(F)$. Moreover, by \cite[Corollary 6.23]{morimoto}, $\theta(\Pi)$ is also unramified.
Proposition \ref{OnthetaliftsofPS} then shows that $\Theta(\theta(\Pi))$ is a principal series representation. According to the definition of the local theta correspondence, the representation $\theta(\Pi)\boxtimes \Theta(\theta(\Pi))$ is the maximal $\theta(\Pi)$-isotypic component of the Weil representation $\Omega$. Since $\theta(\Pi)\boxtimes \Pi$ is a quotient of $\Omega$, the representation $\Pi$ must be either a quotient or a sub-representation of $\Theta(\theta(\Pi))$. We recall that, by \cite[Theorem 2.7]{LiUPSpadic}, since $\Pi$ is generic and unramified, it is an irreducible principal series representation and so $\Pi \simeq \Theta(\theta(\Pi))$. Therefore, if $\sigma : = \theta(\Pi)$, then
    $\Theta(\sigma ) = \theta(\sigma ) =\Pi$, as desired.
\end{proof}

 In view of Theorem \ref{Shalikaimpliesimagetheta}, a direct check from Proposition \ref{OnthetaliftsofPS} shows that either condition \textit{(1)} or \textit{(2)} of Theorem \ref{FinalMackey} hold if and only if $\Pi$ is the small theta lift of $\sigma$, with $\sigma$ the generic factor of a dihedral principal series $I_\H(\chi)$ of $\mathbf{PH}(F)$ as in  Proposition \ref{OnthetaliftsofPS}(2)(a). Here the term dihedral means that $I_\H(\chi) \otimes \chi_{E/F} \simeq I_\H(\chi)$.
For such $\Pi$'s, Theorem \ref{FinalMackey} does not immediately prove uniqueness of Shalika functionals. Nevertheless, one can show it in a rather elegant way by using the properties of the local theta correspondence, as we show below.

\subsection{Uniqueness of Shalika functionals}\label{UniquenessShalikaFunctionals}

We are thankful to Wee Teck Gan for explaining to us how to relate the local multiplicity of Shalika functionals to the one of Whittaker models by means of the local theta correspondence.

\begin{proposition}\label{SUniqueIfTIrr}
Let $\sigma$ and $\Pi$ be irreducible admissible representations of $\mathbf{P}\H^+(F)$ and  $\mathbf{P}\G(F)$, for which \[{\rm Hom}_{\mathbf{P}\H^+(F) \times \mathbf{P}\G(F)}(\Omega, \sigma \boxtimes \Pi) \ne 0 .\] Then we have the following inclusion
\[ {\rm Hom}_{U_\H(F)}(\sigma, \psi) \subseteq {\rm Hom}_{U_\H(F)}(\Theta(\Pi), \psi)  \simeq {\rm Hom}_{S(F)}(\Pi,\chi_{S}). \]
In particular, if $\Theta(\Pi)$ is irreducible
\[ {\rm dim}\,{\rm Hom}_{S(F)}(\Pi,\chi_{S}) \leq 1 .\]
\end{proposition} 

\begin{proof}
As $\Theta(\Pi) \twoheadrightarrow \sigma$, we have the first inclusion. Moreover,
\[ {\rm Hom}_{U_\H(F)}(\Theta(\Pi), \psi) \simeq {\rm Hom}_{U_\H(F) \times \mathbf{P}\G(F)}( \Omega, \psi \boxtimes \Pi) \simeq {\rm Hom}_{\mathbf{P}\G(F)}( \Omega_\psi, \Pi), \]
where $\Omega_\psi$ denotes the twisted Jacquet module of $\Omega$ with respect to the generic datum $(U_\H(F), \psi)$ of $ \mathbf{P}\H^+(F)$. By \cite[Proposition 6.6(2)]{morimoto}, $\Omega_{\psi}$ is isomorphic to ${\rm Ind}_{S(F)}^{\G(F)}(\chi_{S})$ as a $\G(F)$-representation\footnote{By the exceptional isomorphism introduced above, the Shalika subgroup of $\mathbf{P}\G'(F)$ and its character $\psi'$ of Definition \cite[Definition 6.5]{morimoto} agree with $S(F)$ and $\chi_{S}$ (\textit{cf}. \cite[(3.5)]{morimoto}).}, hence 
\[{\rm Hom}_{\mathbf{P}\G(F)}( \Omega_\psi, \Pi) \simeq {\rm Hom}_{\mathbf{P}\G(F)}({\rm Ind}_{S(F)}^{\G(F)}(\chi_{S}), \Pi) \simeq {\rm Hom}_{S(F)}(\chi_S, \Pi)\simeq {\rm Hom}_{S(F)}(\Pi^\vee,\chi_{S}^{-1}). \]
It remains to show that \[ {\rm Hom}_{S(F)}(\Pi,\chi_{S}) \simeq {\rm Hom}_{S(F)}(\Pi^\vee,\chi_{S}^{-1}).\]
This fact, which is already implicit in \cite[Theorem 6.9]{morimoto}, can be proved as follows. Moeglin, Vigneras, and Waldspurger in \cite[Ch. 4]{MVW87} show that, if $V$ is the $E$-module with Hermitian form $\langle \,,\,\rangle_V$ that defines $\mathbf{P}\G(F)$, there exists an element $h \in \GL_{F}(V)$ such that $\langle h x, h y \rangle_V = \langle y,x\rangle_V$ and such that $\overline{g} = h g h^{-1}$. Let $\Pi^\dagger$ be the representation of $\G(F)$ given by the same space as $\Pi$ and with action $\Pi^\dagger(x) = \Pi (h x h^{-1})$. Then a linear form $\Lambda: \Pi \to \C$ such that for $s \in S(F)$ and $v$ in the space of $\Pi$
\[\Lambda( \Pi(s) v) = \chi_{S}(s) \Lambda(v), \]
induces a linear form on $\Pi^\dagger$ such that 
\[\Lambda( \Pi^\dagger(s) v) = \chi_{S}(h s h^{-1}) \Lambda(v) = \chi_{S}^{-1}(s) \Lambda(v). \]
This identifies \[ {\rm Hom}_{S(F)}(\Pi,\chi_{S}) \simeq {\rm Hom}_{S(F)}(\Pi^\dagger,\chi_{S}^{-1}).\]
The desired isomorphism follows from the fact that $\Pi^\vee \simeq \Pi^\dagger$ (\textit{cf}.  \cite[Ch. 4, \S II.1]{MVW87}). Finally, when $\Theta(\Pi)$ is irreducible, $\Theta(\Pi)=\sigma$ and the upper bound on the dimension of the space of Shalika functionals follows from the uniqueness of Whittaker models.
\end{proof}

\begin{prop}\label{Uniquenessinert}
    Let $\sigma$ and $\Pi$ be irreducible generic unramified representations of $\mathbf{P}\H^+(F)$ and  $\mathbf{P}\G(F)$, for which $\Pi = \theta(\sigma)$. Then 
\[ {\rm Hom}_{U_\H(F)}(\sigma, \psi) \simeq {\rm Hom}_{S(F)}(\Pi,\chi_{S}). \]
In particular,  $\Pi$ admits a unique  non-trivial Shalika functional up to constant.
\end{prop}
\begin{proof}
    It follows from Proposition \ref{SUniqueIfTIrr} and Lemma \ref{lemThetatheta}.
\end{proof}

We summarize the results obtained in this section in the following statement, which constitutes the main first result of the article. 

\begin{theorem}\label{MainSummarizingS3}
    Let $\Pi$ be a generic unramified representation of $\mathbf{PG}(F)$. Then, $\Pi$ has a unique non-trivial Shalika functional up to constant and it is the small theta lift of a generic unramified representation of $\mathbf{PH}^+(F)$.
\end{theorem}
\begin{proof}
It follows from Theorems \ref{FinalMackey}, \ref{Shalikaimpliesimagetheta}, and Proposition \ref{Uniquenessinert}.
\end{proof}

\section{A Casselman-Shalika formula for the Shalika model of \texorpdfstring{${\rm PGU}_{2,2}$}{GU({2,2})}}\label{S:CSformulasplitandinert}

Let $F$ be a non-archimedean local field  of characteristic zero. In what follows, we describe the second main result (Theorem \ref{CasselmanShalikaformula}) of the manuscript, namely the Casselman--Shalika formula for the Shalika model of generic unramified representations of  ${\rm PGU}_{2,2}(F)$. The formula takes inspiration from the one proved in \cite{Sakellaridis} for spherical Shalika functionals for $\mathrm{GL}_4(F)$, which we recall
 in \S \ref{subsec:onsplitsak}.
 
\subsection{The case of $E=F\times F $}\label{subsec:onsplitsak}

Let $\Pi$ be an irreducible unramified smooth representation of $\mathbf{P}\G(F)$. In \S \ref{ss:Shalikamodels} we have fixed an isomorphism $\mathbf{P}\G(F) \simeq \mathrm{PGL}_4(F)$ which identifies the subgroup $S(F)$ with the classical Shalika subgroup of $\GL_4(F)$. In particular, recall that we can regard $\Pi$ as a representation of $\mathrm{PGL}_4(F)$ and any Shalika functional on it as a functional $\Lambda_{\mathcal{S}} : \Pi \to \C$, such that, for all $v \in \Pi$, we have \[ \Lambda_{\mathcal{S}}( \left(\begin{smallmatrix} g &  \\  &  g \end{smallmatrix} \right) u(X) \cdot v) =\psi({\rm Tr}(X))\Lambda_{\mathcal{S}}(v),\, \forall g \in \GL_2(F), \forall X \in M_{2\times2}(F). \] 
In this context, spherical Shalika functionals and their Casselman--Shalika formula have been studied in detail by \cite{Sakellaridis}. Since we use the formula (in the case of $\GL_4$) of \textit{loc.cit.} in the last section of the article, we recall its statement and notation in what follows. 

Suppose that the representation $\Pi$ appears as the unique unramified irreducible factor of a principal series $I_{\GL_4}(\chi)$, where $\chi=(\chi_1,\chi_2,\chi_3,\chi_4)$ is an unramified character of the diagonal torus of $\GL_4(F)$ such that $\prod_i \chi_i = \mathbf{1}$. By Proposition \ref{uniquenesssplit}, there exists at most one non-trivial Shalika functional $ \Lambda_{\mathcal{S}}$ on $\Pi$, up to constant. By \cite[Corollary 5.4]{Sakellaridis} and the discussion at the end of \S 8 of {\emph{loc.cit.}}, there exists a non-trivial one if and only if $\Pi = I_{\GL_4}(\chi)$ and $\chi$ is a regular character of the form (or a Weyl translate of)  $(\chi_1,\chi_2,\chi_2^{-1},\chi_1^{-1})$. 
Under these assumptions, the Shalika functional $ \Lambda_{\mathcal{S}}$ is uniquely defined up to a scaling factor. We therefore normalize it according to \cite[Theorem 2.1]{Sakellaridis}. Before explaining this, we introduce some extra notation. We let $\Gamma$, resp. $W$, denote the Weyl groups $W_{\Sp_4}$, $W_{\GL_4}$ of $\Sp_4(F)$, $\GL_4(F)$ respectively; we also let $\Phi_{\GL_4}$, $\Phi_{\Sp_4}$ denote the root systems for $\GL_4$ and $\Sp_4$. The embedding $\Sp_4 \hookrightarrow \GL_4$ identifies $\Gamma$ as a subgroup of the Weyl group of $\GL_4$ and induces a surjection 
\[ \Phi_{\GL_4} \twoheadrightarrow \Phi_{\Sp_4},\]
which is one-to-one onto the set of long roots $\Phi_{\Sp_4}^l$ and two-to-one on the set of short roots $\Phi_{\Sp_4}^s$. For each $\alpha \in \Phi_{\Sp_4}$, we denote by $\alpha$, $\tilde{\alpha}$ the (possibly equal) elements in the pre-image of it.  Finally, a superscript $+$ over $\Phi$ will denote positive roots. For each root $\alpha$, $\check{\alpha}$ denotes the corresponding coroot. Then, the notation $e^{\check{\alpha}}(g_{\chi})$ stands for $\chi(g_\alpha)$, where, if $\alpha: {\rm diag}(t_1,t_2,t_3,t_4) \mapsto t_it_j^{-1}$,  $g_\alpha$ is equal to the diagonal matrix with $\varpi$ on the $i$-th line, $\varpi^{-1}$ on the $j$-th line, and $1$'s otherwise.
Let $g_{\chi} \in \Sp_4(\C)$ be a representative of the Frobenius conjugacy class of $\Pi$.

\begin{definition}
We define
\[ \mathcal{S}_{\Pi}(g) := \Lambda_{\mathcal{S}} ( g \cdot \phi_0), \]
where $\phi_0$ is the spherical vector of $\Pi$ such that $\phi_0(1)=1$ (with $\phi_0$ seen as an element of $I_{\GL_4}(\chi)$), and $\Lambda_{\mathcal{S}}(\cdot)$ is the Shalika functional such that \[\Lambda_{\mathcal{S}}(\phi_0) = \frac{\mathcal{Q}}{1+ q^{-1}} \cdot \frac{ e^{- \check{\rho}} \prod_{\beta \in \Phi_{\GL_4}^+} (1 - e^{\check{\beta}}) }{\prod_{\alpha \in \Phi_{\Sp_4}^+} (1 - q^{-1} e^{\check{\alpha}})\mathcal{A}\left(e^{\check{\rho}}\right)} (g_{\chi}),\]
where $\mathcal{Q} = \sum_{\omega \in W} ({\rm Iw}_{\GL_4} \omega {\rm Iw}_{\GL_4} : {\rm Iw}_{\GL_4})^{-1}$, $\check{\rho} = \tfrac{1}{2} \sum_{\alpha \in \Phi_{\Sp_4}^+} \check{\alpha}$, and $\mathcal{A}$ denotes the alternator $\mathcal{A}( \cdot ) = \sum_{\omega\in \Gamma} (-1)^{\ell(\omega)} \omega(\cdot)$.
\end{definition}

\begin{remark}
As explained in the proof of \cite[Theorem 2.1]{Sakellaridis}, the factor ${\prod_{\alpha \in \Phi_{\Sp_4}^+} (1 - q^{-1} e^{\check{\alpha}})} (g_{\chi})$ appearing in the denominator of the formula above would be zero if $I_{\GL_4}(\chi)$ was reducible.
\end{remark}

The following is a special case of \cite[Theorem 2.1]{Sakellaridis} - the statement in \emph{loc.cit.} contains a small typo, however we invite the reader to consult \cite[(78)]{Sakellaridis} where the correct formula is given.

\begin{theorem}\label{thm:sakellaridis} For all $n \geq 0$, we have \[\mathcal{S}_{\Pi}\left(\left(\begin{smallmatrix} \varpi^n I_2 &  \\  & I_2 \end{smallmatrix} \right)\right) =\frac{q^{-2n}}{1+ q^{-1}} \cdot  \mathcal{A}\Big( e^{\check{\rho}+n(\check{\alpha}_1+\check{\alpha}_2)} \cdot \prod_{\alpha \in \Phi_{\Sp_4}^{+,s}} (1 - q^{-1} e^{-\check{\alpha}}) \Big)(g_{\chi}) ({\mathcal{A}\left(e^{\check{\rho}}\right)}(g_{\chi}))^{-1}.  \]
\end{theorem}

\noindent We'd like to note that a similar formula was already proved in \cite{BFGsplitorthogonal}, as, via the isomorphism between $\mathrm{PGL}_4$ and $\mathrm{PGSO}_{3,3}$, Shalika models coincide with the ``Whittaker-orthogonal'' models considered by \cite{BFGsplitorthogonal}.

\subsection{The case of 
the unramified quadratic field extension $E/F$}\label{ss:inertcase}

Let $\Pi$ be an irreducible generic unramified representation of $\mathbf{P}\G(F)$; then recall that $\Pi$ is isomorphic to a principal series $I_{\G}(\xi)$ by \cite[Theorem 2.7]{LiUPSpadic}. By Theorem \ref{Shalikaimpliesimagetheta}, we know that $\Pi$ is the small theta lift of a $\psi$-generic unramified representation $\sigma$ of $\mathbf{PH}^+(F)$. From now on, we identify $\sigma$ with the $\psi$-generic component of $I_{\H^+}(\chi)$, with $\chi$ satisfying that $\chi_1\chi_2 \chi_0^2=1$ and \begin{align}
\!\!\!\! &|\cdot |^{\pm 1} \not\in \{ \chi_1,\chi_2, \chi_1\chi_2, \chi_1\chi_2^{-1}\} .\label{nondegeneracy}
\end{align}

\noindent By Proposition \ref{OnthetaliftsofPS}, we have 
\begin{align}\!\!\!\! &\xi = ( \chi_0 \circ N_{E/F}, \chi_2\chi_0 \circ N_{E/F}, \chi_1  \chi_{E/F}).\label{shapeofcharacter}
\end{align}
Proposition \ref{Uniquenessinert} asserts that there exists a unique Shalika functional $\Lambda_{\mathcal{S}}$ for $\Pi$ up to constant.  Let $\phi_0$ be the spherical vector of $\Pi$ such that $\phi_0(1)=1$. Denote $\mathcal{S}_\Pi(g) := \Lambda_{\mathcal{S}}(g \cdot \phi_0)$. We now describe a new explicit formula for $\mathcal{S}_\Pi(g)$, which reflects the fact that $\Pi$ is in the image of the local theta correspondence from $\mathbf{P}\H^+(F)$ and has applications to the calculation of Rankin--Selberg integrals (\textit{cf}. \S \ref{sec_application_to_Lvalues}). 

Recall that the relative root system of $\mathbf{P}\G_{/F}$ is equal to the root system of $\mathbf{P}\H_{/F}$, allowing us to express the resulting formula in terms of the root system of $\Sp_4$. This, together with the equality \eqref{shapeofcharacter} given by the theta correspondence, let us relate spherical Shalika functionals on $\mathbf{P}\G$ to $L$-factors for $\mathbf{P}\H$. Recall that the Frobenius conjugacy class of $I_\H(\chi)$ can be represented by
\begin{align}\label{Satake_parameters_GSp4}
g_\chi={\rm diag}((\chi_1\chi_2\chi_0)(\varpi), (\chi_1\chi_0)(\varpi),(\chi_2\chi_0)(\varpi),\chi_0(\varpi)) \in {}^L\mathbf{P}\H. 
\end{align}
If we let $i: X^*(T_{\Sp_4}(\C)) \simeq X_*(T_{\mathbf{P}\H}(\C))$ be the isomorphism induced by the duality between $\mathbf{P}\H(\C)$ and $\Sp_4(\C)$, then for every root $\alpha$ for $\Sp_4(\C)$, with $\check{\alpha}$ the corresponding coroot, we have that $e^{\check{\alpha}}(g_\chi) = \chi(i(\alpha)(\varpi))$.

\begin{theorem}\label{CasselmanShalikaformula}
For all $n \geq 0$, a suitably normalized spherical Shalika functional satisfies
   \[\mathcal{S}_\Pi\left(\left(\begin{smallmatrix} \varpi^n I_2 &  \\  & I_2 \end{smallmatrix} \right)\right) = \frac{q^{-2n}}{ 1 + q^{-1}} \mathcal{A}\left( (-1)^n e^{\check{\rho}+n(\check{\alpha}_1 + \check{\alpha}_2)   }\prod_{\alpha \in \Phi^{+,s}_{\Sp_4}}(1 + q^{-1}e^{-\check{\alpha}})\right)(g_{\chi}) (\mathcal{A}(e^{\check{\rho}}) (g_{\chi}))^{-1}.\] 
\end{theorem}
\noindent Here, we have again denoted $\mathcal{A}(\cdot) = \sum_{\omega \in W_{{\rm Sp}_4}}(-1)^{\ell(\omega)} \omega(\cdot)$. Thanks to Theorem \ref{CasselmanShalikaformula} and the Iwasawa and Cartan decompositions we can obtain a formula for $\mathcal{S}_\Pi(g)$ for any $g \in \G(F)$. The proof of Theorem \ref{CasselmanShalikaformula}, which occupies all the next section, follows the strategy of the proof of \cite[Theorem 2.1]{Sakellaridis}. We would like to point out that most of the steps follow verbatim the ones of \emph{loc.cit.}, with main differences the proof of Proposition \ref{proponconv} and the calculations around Proposition \ref{funct2} (e.g. Lemma \ref{comp2aux}). 

\section{Proof of Theorem \ref{CasselmanShalikaformula}}\label{proofofCSf}

Let $E = F(\delta)$ be the unique unramified quadratic field extension of $F$ with ring of integers $\mathcal{O}_E$. In the whole section, we keep the notation introduced in \S \ref{ss:inertcase}. In particular, we let $\Pi=I_{\G}(\xi)$ be a generic unramified principal series of $\mathbf{P}\G(F)$, with $\xi$ of the form \eqref{shapeofcharacter}. 

\subsection{Shalika functionals and distributions}
Recall that, by Proposition \ref{Uniquenessinert}, there exists a unique non-trivial Shalika functional $\Lambda_{\mathcal{S}}$ for $I_{\G}(\xi)$. Without loss of generality, we suppose that $I_{\G}(\xi)$ is unramified with respect to the good hyperspecial maximal compact $\G(\mathcal{O})$ introduced in Remark \ref{IntegralModelsRemark}. Moreover, recall that $w_{\rm op}$ denotes the representative of the open orbit of $S(F)\setminus \G(F)/B_{\G}(F)$. It was described explicitly in Proposition \ref{Prop:stabilisers}. Consider the group
$\tilde{S}(F) := w_{\rm op}^{-1}S(F)w_{\rm op}$. Then  $B_{\G}(F) \cdot \tilde{S}(F)$ is Zariski open in $G(F)$. 
Set the character \begin{align*}\chi_{\tilde{S}}:\;\tilde{S}(F) \to\C^{\times},\; s \mapsto \chi_{S}(w_{\rm op} s w_{\rm op}^{-1}).\end{align*}

\noindent The map  $\mathrm{Ind}_{S(F)}^{\G(F)}\chi_{S} \to \mathrm{Ind}_{\tilde{S}(F)}^{\G(F)}\chi_{\tilde{S}}$, given by sending $ f \mapsto f_{\tilde{S}}(g) := f(w_{\rm op} g)$, is an isomorphism of $\G(F)$-representations. The composition of the above map with the map $I_{\G}(\xi) \to \mathrm{Ind}_{S(F)}^{\G(F)} \chi_{S}$ associated to the Shalika functional $\Lambda_{\mathcal{S}}$ results in a $\G(F)$-intertwining map of the form
\[I_{\G}(\xi) \to \mathrm{Ind}_{\tilde{S}(F)}^{\G(F)} \chi_{\tilde{S}}.\]
As in \cite{Sakellaridis}, working with this realisation of the Shalika model will clarify some of the computations. The associated Shalika functional is given by $\Lambda_{\tilde{S}}(g \cdot v) := \Lambda_{\mathcal{S}}((w_{\rm op}g) \cdot v)$. 
The associated spherical vector in this Shalika model is then 
\begin{equation}\label{relmodels} \tilde{\mathcal{S}}(g) := \Lambda_{\tilde{S}}(g \cdot \phi_0) = \Lambda_{ \mathcal{S}}((w_{\rm op} g w_{\rm op}^{-1}) \cdot \phi_0),\end{equation}
where the last equality follows since $w_{\rm op}\in \G(\mathcal{O})$.

Consider the canonical surjective map $ P_{\xi}:\;\mathcal{C}^{\infty}_{c}(\G(F)) \to I_\G(\xi)$ defined by
\begin{align}\label{operator}
    \phi&\mapsto P_{\xi}(\phi)(x):=\int_{B_\G(F)}\xi^{-1}(b)\delta_{B_\G}^{1/2}(b)\phi(bx)db,
\end{align}
where $db$ is the Haar measure of $B_\G(F)$ normalized according to the Iwasawa decomposition, i.e. $dg = dbdk$, with $dg$ the Haar measure of $\G(F)$ satisfying $\mathrm{vol}(\G(\mathcal{O})) = 1$. This induces the injective map
\[P_{\xi}^*:\; I_\G(\xi)^* \to \mathcal{D}(\G(F)),\]
where $I_\G(\xi)^*$ denotes the dual of $I_\G(\xi)$ and $\mathcal{D}(\G(F))$ are the distributions on $\G(F)$. Using $P_{\xi}^*$, we can then associate to $\Lambda_{\tilde{S}}$ a distribution $\Delta_{\tilde{S},\xi}$, with the following property 
\[\Delta_{\tilde{S},\xi}(b g \tilde{s}) = \xi^{-1}(b)\delta_{B_\G}^{1/2}(b) \chi_{\tilde{S}}^{-1}(\tilde{s})\Delta_{\tilde{S},\xi}(g),\]
for all $b \in B_\G(F)$, $g\in \G(F)$, and $\tilde{s} \in \tilde{S}(F)$. If we denote $g_{n} := \left(\begin{smallmatrix} \varpi^n I_2 &  \\  & I_2 \end{smallmatrix} \right)$, where $n \geq 0$, notice that $ \mathcal{S}(g_{n}) = \tilde{\mathcal{S}}(g_{-n})$. To prove Theorem \ref{CasselmanShalikaformula}, we calculate $\tilde{\mathcal{S}}(g_{-n})$  explicitly. As in \cite[\S 4]{Sakellaridis}, using the work of Casselman (\textit{cf}. \cite{CasselmanSpherical}) and Hironaka (\textit{cf}. \cite{Hironaka}), one can express (\textit{cf}. \cite[(38)]{Sakellaridis})
\begin{equation}\label{eq:Casselmanbasis}
\tilde{\mathcal{S}}(g_{-n}) = Q^{-1}\sum_{\omega\in W_{\H}}\prod_{\substack{\alpha\in\Phi_{\mathbf{P}\H}^+\\\omega\alpha>0}}c_{\alpha}(\xi)T^*_{\omega^{-1}}\Delta_{\tilde{S},^{\omega}{\xi}}(R_{g_{-n}}\mathbf{1}_{\rm Iw}), 
\end{equation}
 where $\mathcal{Q} = \sum_{\omega \in W_\H} ({\rm Iw} \omega {\rm Iw} : {\rm Iw})^{-1}$.
Here we have used the following notation:
\begin{enumerate}
    \item For each $\omega \in W_{\G}$, we have the intertwining map $T_{\omega}: I_\G(\xi) \to I_\G({}^{\omega}\xi)$ (\textit{cf}. \cite[\S 3.7]{Cartier}) and its pullback $T_{\omega}^*:\mathcal{D}(\G(F))_{{}^{\omega}\xi^{-1}} \to \mathcal{D}(\G(F))_{\xi^{-1}}$ (\textit{cf}. \cite[\S 1.2]{Hironaka}),
     \item ${\rm Iw}$ is the Iwahori subgroup of $\G(F)$ and $\mathbf{1}_{{\rm Iw}}$ denotes its characteristic function,
    \item $R_{g_{-n}}$ is the right translate action of $g_{-n} \in \G(F)$ on $\mathcal{C}^{\infty}_c(\G(F))$,
    \item the constants $c_{\alpha}(\xi)$ are defined as in \cite[(1.12)]{Hironaka} (\textit{cf}. \S \ref{proofofCSf:theconstants} below).
\end{enumerate}

\noindent Theorem \ref{CasselmanShalikaformula} will follow from the uniqueness of the Shalika functional and from a direct calculation of each term of the right hand side of \eqref{eq:Casselmanbasis} explicitly.

\subsubsection{Definition of $c_{\alpha}(\xi)$}\label{proofofCSf:theconstants}
In what follows, we define the constants $c^G_{\alpha}(\xi)$, where $G_{/F}$ is either $\H$ or $\G,$ $\xi$ is a character of the maximal torus of $G(F)$, and  where $\alpha$ runs through the set of positive roots $\Phi^+_G$ induced by our choice of the upper-triangular Borel subgroup of $G$. Here, we closely follow \cite[\S 3]{CasselmanShalika}.

For the following discussion, we let $M_{\alpha}$ be the $F$-points of the Levi subgroup associated to a positive root $\alpha \in \Phi^+_G$ and denote by $M^{der}_{\alpha}$ the simply connected cover of its derived subgroup. In our case, as $G$ varies in $\{\H,\G \}$, $M^{der}_{\alpha}$ is either $\SL_{2}(F)$ or $\SL_2(E)$.

\begin{definition}\label{defacte}

Let $L$ be either $F$ or $E$, with ring of integers $\mathcal{O}_{L}$ and $\mathfrak{m}$ its maximal ideal. Consider the group $\SL_2$ defined over $L$ with maximal diagonal torus $T_{\SL_{2}}$ and define the following coset of $T_{\SL_{2}}(L)/T_{\SL_{2}}(\mathcal{O}_{L})$
\[a = \left(\begin{smallmatrix}\mathfrak{m}-\mathfrak{m}^2& \\ & (\mathfrak{m}-\mathfrak{m}^2)^{-1}\end{smallmatrix}\right).\]
Given $\alpha \in \Phi^+_G$, $a_{\alpha}$ denotes the image of a representative of $a$ via the map $x_\alpha: M_{\alpha}^{der}\to G(F)$. If $\alpha$ is such that $M_{\alpha}^{der} = \SL_2(L)$, we denote $q_{\alpha} := |\mathcal{O}_L/\mathfrak{m}\mathcal{O}_{L}|$.
\end{definition}

\begin{definition}\label{constantdef}
Given $\alpha \in \Phi^+_G$ and $\omega\in W_G$, we define 
\begin{align*} c_{\alpha}^G(\xi) = \frac{1-q_{\alpha}^{-1}\xi(a_{\alpha})}{1-\xi(a_{\alpha})},\;\;c^G_{\omega}(\xi) = \prod_{\substack{\alpha\in \Phi^+_G\\\omega\alpha<0}}c_{\alpha}^G(\xi).
\end{align*}
If $G=\G$, we simply denote them by $c_{\alpha}(\xi)$ and $c_{\omega}(\xi)$ respectively. 
\end{definition}
\begin{example}\label{examplesofa} \leavevmode
\begin{itemize} 
    \item If $G = \G$, we have that $M_{2\alpha_2-\alpha_0}^{der} \!\!= \! \SL_2(F)$ and $M_{\alpha_1-\alpha_2}^{der}\!\!= \!{\SL}_2(E)$. Then  $a_{\alpha_1-\alpha_2} = \left(\begin{smallmatrix}\varpi& & & \\ &\varpi^{-1}& & \\ & &\varpi& \\ & & &\varpi^{-1}\end{smallmatrix}\right)$ and $a_{2\alpha_2-\alpha_0} = \left(\begin{smallmatrix}1& & & \\ &\varpi& & \\ & &\varpi^{-1}& \\ & & &1\end{smallmatrix}\right)$. Furthermore, by \cite[p.\;219]{CasselmanShalika}, we obtain $q_{2\alpha_2-\alpha_0} = q$ and $q_{\alpha_1-\alpha_2} = q^2$, implying that $[{\rm Iw} s_1  {\rm Iw}:{\rm Iw}] = q^2$ and $[{\rm Iw} s_2  {\rm Iw}: {\rm Iw}] = q$, where $s_1, s_2$ are the Weyl elements corresponding to $\alpha_1-\alpha_2$ and $2\alpha_2-\alpha_0$, respectively. 
    \item When $G = \H $, which is split, for any root $\alpha\in\Phi_G^+$ we have $q_{\alpha} = q$ and $a_{\alpha}$ is the image of $\varpi$ under the coroot $\check{\alpha}$ (\rm {\textit{cf}.} \cite[Remarks 1.1 and 3.3]{CasselmanSpherical}).
\end{itemize}
\end{example}

\subsection{First reduction: removing the dependence on \texorpdfstring{$n$}{n}}\label{subsub:firstreduction}
This subsection follows closely \cite[\S 5 \& \S 6]{Sakellaridis}.
\begin{lemma}\label{formulasupp}
For any $\omega\in W_{\H}$ and $\phi\in I_{\G}({}^{\omega}\xi)$ with $\mathrm{supp}(\phi)\subset B_{\G}(F)\tilde{S}(F)$ we have \begin{align}\label{formulafortheShalikafunctional}\Lambda_{\tilde{S}}(\phi) = \int_{B_{\G}(F)\cap \tilde{S}(F)\setminus \tilde{S}(F)}\phi(h)\chi_{\tilde{S}}^{-1}(h)dh.\end{align}
\end{lemma}
\begin{proof}
The proof is a verbatim translation of the one of \cite[Corollary 5.5]{Sakellaridis}.
\end{proof}

\noindent In the integral above, $dh$ denotes the Haar measure on $\G(F)$ so that the Iwahori subgroup ${\rm Iw}$ has volume equal to $1$. This corresponds to normalizing \begin{align}\label{normalizationofmodel}
    \Lambda_{\tilde{S}}( P_\xi(\mathbf{1}_{\rm Iw})) = 1.
\end{align} 
Notice that, if $f\in \mathcal{C}^{\infty}_c(\G(F))$ is chosen so that $ P_{\xi}(f)\in I_{\G}(\xi)$ has its support contained in $B_{\G}(F) \cdot \tilde{S}(F)$, Lemma \ref{formulasupp} can be rephrased as
\begin{equation}\label{combination}\Delta_{\tilde{S},\xi}(f) = \int_{B_\G(F)\tilde{S}(F)}\xi^{-1}\delta_{B_\G}^{1/2}(b_h)f(h)\chi_{\tilde{S}}^{-1}(\tilde{s}_h)d h,\end{equation}
where we have written $h = b_h \cdot \tilde{s}_h$, with  $b_h \in B_{\G}(F)$ and $\tilde{s}_h \in \tilde{S}(F)$. Let us point out the fact that here we are using that $B_{\G}(F)\tilde{S}(F)$ is dense in $\G(F)$. From now on, we identify the Shalika distribution $\Delta_{\tilde{S},\xi}$ as the one defined by  \eqref{combination}. Using this expression, we obtain the following.

\begin{prop}\label{relation}
For every $n \geq 0$, we have \begin{equation}\label{csformulaintermediario}\tilde{\mathcal{S}}(g_{-n}) = Q^{-1} \sum_{\omega\in W_{\H}} \prod_{\substack{\alpha\in\Phi_{\mathbf{P}\H}^+\\\omega\alpha>0}} c_{\alpha}(\xi) ({}^{\omega}\xi^{-1}\delta^{1/2}_{B_\G})(g_{n})T^*_{\omega^{-1}}\Delta_{\tilde{S},^{\omega}{\xi}}(\mathbf{1}_{\rm Iw}).\end{equation}
\end{prop}
\begin{proof}
The proof of the statement is identical to the one of \cite[(48)]{Sakellaridis}; therefore we sketch it and invite the curious reader to consult either \cite[\S 6]{Sakellaridis} or the second author's PhD thesis for further details.

Firstly, notice that $R_{g_{-n}}\mathbf{1}_{\mathrm{Iw}}$ has support inside $B_{\G}(F) \cdot \tilde{S}(F)$ (as for \cite[Lemma 5.1]{Sakellaridis}). Hence, we can apply the formula \eqref{combination} to obtain
 \[\Delta_{\tilde{S},\xi}(R_{g_{-n}}\mathbf{1}_{\mathrm{Iw}}) = \xi^{-1}\delta^{1/2}_{B_\G}(g_{n}),\] which is proved as in \cite[Corollary 6.1]{Sakellaridis}.  By uniqueness of the Shalika functional (Proposition \ref{Uniquenessinert}), we now have that $T^*_{\omega^{-1}}\Delta_{\tilde{S},{}^{\omega}\xi}$ is a constant multiple of $\Delta_{\tilde{S},\xi}$, which implies, as in \cite[Corollary 6.2]{Sakellaridis}, the equality \[T^*_{\omega^{-1}}\Delta_{\tilde{S},{}^{\omega}\xi}(R_{g_{-n}}\mathbf{1}_{\mathrm{Iw}}) = ({}^\omega\xi^{-1}\delta^{1/2}_{B_\G})(g_{n}) T^*_{\omega^{-1}}\Delta_{\tilde{S},{}^{\omega}\xi}(\mathbf{1}_{\mathrm{Iw}}).\]
Plugging this into \eqref{eq:Casselmanbasis}, we get the desired formula.
\end{proof}

\subsection{Second reduction: Convergence of the Shalika functional on all of its domain}

Thanks to Proposition \ref{relation}, we have reduced the proof of Theorem \ref{CasselmanShalikaformula} to an explicit evaluation of $T^*_{\omega^{-1}}\Delta_{\tilde{S},^{\omega}{\xi}}(\mathbf{1}_{\rm Iw})$ for each $\omega \in W_\H$. Since, by uniqueness of the Shalika functional,  $T^*_{\omega^{-1}}\Delta_{\tilde{S},{}^{\omega}\xi}$ is a constant multiple of $\Delta_{\tilde{S},\xi}$, it is tempting to use \eqref{combination} to evaluate $T^*_{\omega^{-1}}\Delta_{\tilde{S},^{\omega}{\xi}}(\mathbf{1}_{\rm Iw})$. To do so we need to first show that the period integral of Lemma \ref{formulasupp} converges absolutely for all $\phi \in I_\G(\xi)$, so that we can evaluate the formula \eqref{combination} at the characteristic function $\mathbf{1}_{{\rm Iw},\omega}$ of ${\rm Iw} \, \omega \,  {\rm Iw}$.

We first consider the following auxiliary calculation. Recall that we have denoted by $T_\delta$ the maximal non-split torus of $\GL_2(F)$ which is equal to the stabilizer of the open orbit $w_{\rm op}$ (\textit{cf}. Proposition \ref{Prop:stabilisers}).

\begin{lem}\label{quotstab}
We have an isomorphism
\[T_{\delta}\setminus \GL_2(F)\simeq T_1(F)\overline{N}_{\GL_2}(F),\]
where $T_1(F)  \overline{N}_{\GL_2}(F)= \left\{ t(y)\bar{n}(x) = \left(\begin{smallmatrix}y & \\ & 1\end{smallmatrix}\right)\left(\begin{smallmatrix}1& \\ x &1\end{smallmatrix}\right),\;y\in F^{\times},\;x\in F\right\}$. 
\end{lem}

\begin{proof}
Recall that the group $\GL_2(F)$ acts on $\mathbb{P}^1_{E}\simeq \GL_2(E)/B_{\GL_2}(E)$ with two orbits, generated by the elements $(1:0)$ and $(1:\delta)$. Every $g = \left(\begin{smallmatrix}a&b\\c&d\end{smallmatrix}\right)\in T_{\delta}\setminus\GL_2(F)$ acts on $(1:\delta)$ by
\[\left(\begin{smallmatrix}a&b\\c&d\end{smallmatrix}\right)\cdot (1:\delta) = (a+b\delta:c+d\delta).\]
Denote by $O_{\delta}$ the corresponding orbit. It consists of elements $(m:n)$ such that $mn\neq 0$ and $mn^{-1}\not\in F$. Since $T_{\delta}\setminus \GL_2(F)$ acts transitively on $O_{\delta}$, we have $T_{\delta}\setminus \GL_2(F)\simeq O_{\delta}$ as topological spaces. On the other hand, the subgroup $T_1(F)  \overline{N}_{\GL_2}(F)= \left\{\left(\begin{smallmatrix}y & \\ & 1\end{smallmatrix}\right)\left(\begin{smallmatrix}1& \\ x &1\end{smallmatrix}\right)\right\} \subset \GL_2(F)$ acts on $O_{\delta}$ with one orbit, with representative $(1:\delta)$. Concretely, if $\left(\begin{smallmatrix}y & \\ x & 1 \end{smallmatrix}\right)\in T_1(F) \overline{N}_{\GL_2}(F)$, we have
\[\left(\begin{smallmatrix}y&  \\ x &1 \end{smallmatrix}\right)(1:\delta) = (y : x + \delta ),\] 
which implies that the stabilizer of $(1 \!:\!\delta)$ is just the identity matrix; hence $T_1(F)\overline{N}_{\GL_2}(F)\simeq O_{\delta}$.
\end{proof}

We now can study the convergence of the period integral of Lemma \ref{formulasupp} for all $\phi\in I_{\G}(\xi)$. Since $\xi$ is an unramified character, we can write $\xi_1 = |\cdot|^{z_1} \circ N_{E/F}$, $\xi_2 = |\cdot|^{z_2} \circ N_{E/F}$,  and $\xi_0 =  |\cdot|^{z_0}$, where $z_1,z_2, z_0$ are complex numbers such that  $|\cdot|^{\sum 2 z_i}=1$. In particular, ${\rm Re}(z_0) = - {\rm Re}(z_1) - {\rm Re}(z_2)$. 

\begin{prop}\label{proponconv} Suppose that $ {\rm Re}(z_1)>{\rm Re}(z_2)+1 > 1$. The period integral of Lemma \ref{formulasupp} converges absolutely for every $\phi\in I_{\G}(\xi)$ and thus represents a Shalika functional.
\end{prop}
\begin{proof} 
As $I_{\G}(\xi)$ is unramified, it suffices to prove the proposition when $\phi$ is equal to the spherical vector $\phi_0$, which we therefore assume from now on.
In order to simplify the calculation, after conjugating by $w_{\rm op}$ (using \eqref{relmodels}), the convergence of the integral for $\tilde{S}$ follows from the one over the Shalika group $S$ given by  
\begin{align*}\Lambda_{\mathcal{S}}(\phi_0) &:= \int_{ T_\delta \setminus S(F)}\!\!\phi_0(w_{\rm op}^{-1}s)\chi_{S}^{-1}(s)ds\\ &= \int_{F^\times} \int_{F} \int_{N_{\G}(F)} \!\!\!\!\!\!\phi_0( \omega^{-1} w_{\delta}^{-1}  t(y) \bar{n}(x) n ) \chi_{S}^{-1}(n)|y|^{-1} dn d x d^\times y,\end{align*}
where, for the second equality, we have used the isomorphism of Lemma \ref{quotstab} and denoted $\omega := s_2s_1s_2$; we have also used the fact that the right Haar measure $ds$ corresponds to the right Haar measure on $T_1(F)\overline{N}_{\GL_2}(F)$, which is  $|y|^{-1} dx d^\times y $ (up to a volume factor coming from the normalization of $ds$ chosen in \S \ref{subsub:firstreduction}).   Notice that $  w_{\delta}^{-1} t(y) \bar{n}(x) = \iota \left(\begin{smallmatrix} y& \\ x - y \delta  & 1\end{smallmatrix}\right) = t(y) \iota \left(\begin{smallmatrix} 1 & \\ x - y \delta  & 1\end{smallmatrix}\right)$ where $\iota: \GL_2(E) \hookrightarrow P_\G(F)$, so that the integral becomes equal to 
\begin{align*} \Lambda_{\mathcal{S}}(\phi_0) &= \int_{F^\times} \int_{F} \int_{N_{\G}(F)} \!\!\!\!\!\!\phi_0( \omega^{-1} t(y) \iota \left(\begin{smallmatrix} 1 & \\ x - y \delta  & 1\end{smallmatrix}\right) n ) \chi_{S}^{-1}(n)|y|^{-1} dn d x d^\times y\\
&= \int_{F^\times} \int_{F} \int_{N_{\G}(F)} \!\!\!\!\!\!\phi_0( \omega^{-1} \iota \left(\begin{smallmatrix} 1 & \\ x - y \delta  & 1\end{smallmatrix}\right) n ) \chi_{S}^{-1}(n)(\xi_1\xi_0)(y) dn d x d^\times y,
\end{align*}
where in the last equality we have used that $\omega^{-1} t(y) = t(y) \omega^{-1}$ and $\phi_0(t(y) g) = \xi \delta_{B_\G}^{1/2}(t(y)) \phi_0(g)$. As $\phi_0$ is fixed by the right translation action of $\G(\mathcal{O})$ and since conjugation by the Weyl element $s_1$ yields $s_1   \iota \left(\begin{smallmatrix} 1 & \\ x - y \delta  & 1\end{smallmatrix}\right) s_1 = \iota \left(\begin{smallmatrix} 1 & x - y \delta \\  & 1\end{smallmatrix}\right)$, the integral simplifies further as 
\[\int_{F^\times} \int_{F} \int_{N_{\G}(F)} \!\!\!\!\!\!\phi_0( \omega_0^{-1} \iota \left(\begin{smallmatrix} 1 & x - y \delta \\   & 1\end{smallmatrix}\right) n ) \chi_{S}^{-1}(n)(\xi_1\xi_0)(y) dn d x d^\times y,\]
where $\omega_0=s_2s_1s_2s_1$ is a representative of the longest Weyl element of $W_\G$. Since $\xi_1 = |\cdot|^{z_1} \circ N_{E/F}$, $\xi_2 = |\cdot|^{z_2} \circ N_{E/F}$,  and $\xi_0 =  |\cdot|^{z_0}$, we have that $(\xi_1\xi_0)(y) = |y|^{2z_1 + z_0}$ and the integral above is dominated absolutely by
\[\int_{F^\times} \int_{F} \int_{N_{\G}(F)} \!  \left | \phi_0( \omega_0^{-1} \iota \left(\begin{smallmatrix} 1 & x - y \delta \\   & 1\end{smallmatrix}\right) n ) \right | |y|^{{\rm Re}(z_1)-{\rm Re}(z_2)} dn d x d^\times y,\]
where we have used that $|\chi_{S}^{-1}(n)| = 1$. Furthermore, since $\{0\}\in F$ is a measure $0$ set of $F$ and $\phi_0\left(\omega_0  \iota\left(\begin{smallmatrix} 1 & x + y \delta \\  & 1\end{smallmatrix}\right) n \right)$ is a continuous function in all $y\in F$, if ${\rm Re}(z_1) > {\rm Re}(z_2)+1$ the latter extends to 
\begin{equation*}\label{eq:1conv} \int_{F} \int_{F} \int_{N_{\G}(F)} \!  \left | \phi_0( \omega_0^{-1} \iota \left(\begin{smallmatrix} 1 & x - y \delta \\   & 1\end{smallmatrix}\right) n ) \right | |y|^{{\rm Re}(z_1)-{\rm Re}(z_2) -1} dn d x d y.\end{equation*} 

We now show that this integral converges. We write it as the sum of two integrals $I_1 + I_2$ by splitting the domain of $x,y$ into $\{ x,y\,:\, x-\delta y \in \mathcal{O}_E\} \cup \{ x,y\,:\, x-\delta y \not \in \mathcal{O}_E\}$ and analyze each separately. The first integral equals to
\begin{align*}
    I_1:=\iint_{\mathcal{O}_E} \!\!\!  \cdots dxd y &= \sum_{n\geq 0}q^{-n({\rm Re}(z_1)-{\rm Re}(z_2)-1)}\int_{\Omega_{n}} \int_{\mathcal{O}} \int_{N_{\G}(F)} \!  \left | \phi_0( \omega_0^{-1} \iota \left(\begin{smallmatrix} 1 & x - y \delta \\   & 1\end{smallmatrix}\right) n ) \right |  dn d x d y \\ 
    & \leq \sum_{n\geq 0}q^{-n({\rm Re}(z_1)-{\rm Re}(z_2)-1)} \cdot \int_{U_{\G}(F)} \!  \left | \phi_0( \omega_0^{-1} u ) \right |  du,
\end{align*}
where we denoted $\Omega_{n} = \{y\,:\,|y| = q^{-n}\}$. The latter converges because ${\rm Re}(z_1) > {\rm Re}(z_2) + 1$ and because the intertwining operator \[T_{\omega_0}(\phi_0)(1) = \int_{U_{\G}(F)} \!  \phi_0( \omega_0^{-1} u )    du \] 
converges absolutely as long as ${\rm Re}(z_1) > {\rm Re}(z_2) > 0$ (\textit{cf}. \cite[(2.3)]{Shahidi}). The second integral 
\begin{align}
  I_2 &= \iint_{\{x-\delta y \not \in \mathcal{O}_E\}} \int_{N_{\G}(F)} \!  \left | \phi_0( \omega_0^{-1} \iota \left(\begin{smallmatrix} 1 & x - y \delta \\   & 1\end{smallmatrix}\right) n ) \right | |y|^{{\rm Re}(z_1)-{\rm Re}(z_2)-1}   dn d x d y \nonumber\\ 
      &= \iint_{\{ x-\delta y \not \in \mathcal{O}_E\}} \int_{N_{\G}(F)} \!  \left | \phi_0( \omega_0^{-1} \iota \left(\begin{smallmatrix} x-y\delta &    \\ 1 & (x-y\delta)^{-1} \end{smallmatrix}\right)  \iota \left(\begin{smallmatrix} (x-y\delta)^{-1} & 1 \\ -1 &  \end{smallmatrix}\right) n ) \right | |y|^{{\rm Re}(z_1)-{\rm Re}(z_2)-1}   dn d x d y \nonumber \\
       &= \iint_{\{ x-\delta y \not \in \mathcal{O}_E\}} \int_{N_{\G}(F)} \!  \left | \phi_0(  \iota \left(\begin{smallmatrix}(x-y\delta)^{-1} & -1   \\  &  x-y\delta \end{smallmatrix}\right)  \omega_0^{-1}n ) \right | |y|^{{\rm Re}(z_1)-{\rm Re}(z_2)-1}   dn d x d y \nonumber\\
       &= \iint_{\{ x-\delta y \not \in \mathcal{O}_E\}} \int_{N_{\G}(F)} \!  \left | \phi_0( \omega_0^{-1}n ) \right | |y^{-1}x^2-y\delta^2|^{1-{\rm Re}(z_1)+{\rm Re}(z_2)}   dn d x d y, \label{badintegralfinal}
\end{align}
where for the third equality we have made the change of variable $n \mapsto \iota \left(\begin{smallmatrix} (x-y\delta)^{-1} & 1 \\ -1 &  \end{smallmatrix}\right) n\iota \left(\begin{smallmatrix} (x-y\delta)^{-1} & 1 \\ -1 &  \end{smallmatrix}\right)^{-1}$ and we have used that $\left(\begin{smallmatrix} (x-y\delta)^{-1} & 1 \\ -1 &  \end{smallmatrix}\right) \in \GL_2(\mathcal{O}_E)$; for the fourth, we used that \[\phi_0(\iota \left(\begin{smallmatrix}(x-y\delta)^{-1} & -1   \\  &  x-y\delta \end{smallmatrix}\right) g) = \xi_1^{-1}\xi_2(x-y\delta)|x^2-\delta^2 y^2| \phi_0(g).\] 
Notice that 
\[|y^{-1}x^2-y\delta^2| = \begin{cases} | y | & \text{ if } v(y) \leq v(x) \\ 
| x^2y^{-1} | & \text{ if } v(y) > v(x).
\end{cases}\]
In particular, $|y^{-1}x^2-y\delta^2| < 1$ because $x-\delta y \not \in \mathcal{O}_E$. Therefore \[\eqref{badintegralfinal} \leq \sum_{n\geq 1}q^{-n({\rm Re}(z_1)-{\rm Re}(z_2)-1)} \cdot \int_{N_{\G}(F)} \!  \left | \phi_0( \omega^{-1} n ) \right |  dn ,\]
and the latter converges as above because of our hypotheses.
\end{proof}

\begin{remark}
    Recall that in the proof above we have used the explicit representatives $s_1$, $s_2$ fixed in \S \ref{ss:defgroups} of the two Weyl elements in $W_{\G}$ associated to the roots $\alpha_1-\alpha_2$ and $2\alpha_2-\alpha_0$, respectively.
\end{remark}

We now proceed as in \cite[Proposition 7.2]{Sakellaridis} to show the following.
\begin{corollary}\label{rationalitySm}
Let $D$ be the set of unramified characters of $\mathbf{P}\G(F)$ of the form \eqref{shapeofcharacter}, for which \eqref{nondegeneracy} holds.
 There exists a unique non-zero Shalika functional on $I_\G(\xi)$, which satisfies \eqref{normalizationofmodel}, for almost all $\xi \in D$. Moreover, if $\{ \phi_\xi\}_{\xi \in D}$ is a rational family of functions $\phi_\xi \in I_\G(\xi)$, then the period integral $\Lambda_{\tilde{S}}(\phi_\xi)$ of \eqref{formulafortheShalikafunctional} is a rational function on $\xi \in D$.
\end{corollary}

\begin{proof}
We apply Bernstein's Theorem (\textit{cf}. \cite[(12.2)]{GelbartPSRallis}) to the system of equations given by \begin{itemize}
    \item $\Lambda_{\tilde{S}}^\xi$ is a Shalika functional on $I_\G(\xi)$, for all $\xi \in D$,
    \item $\Lambda_{\tilde{S}}^\xi$ satisfies \eqref{normalizationofmodel}.
\end{itemize}

Indeed, by Proposition \ref{proponconv}, there exists a unique normalized Shalika functional $\Lambda_{\tilde{S}}^\xi$, given by the period integral of Lemma \ref{formulasupp}, for $\xi \in D$, which satisfies that ${\rm Re}(z_1)>{\rm Re}(z_2)+1$ and ${\rm Re}(z_2) >0$. The latter conditions define an open subset of $D$ - here we are identifying $D$ with the variety $\C^2$ by sending $\xi \mapsto (z_1,z_2)$. We can thus apply Bernstein's theorem to deduce that there exists a unique non-zero Shalika functional $\Lambda_{\tilde{S}}^\xi$ which satisfies \eqref{normalizationofmodel} for almost all $\xi \in D$. This has the further property that, if $\{ \phi_\xi\}_{\xi \in D}$ is a rational family of functions $\phi_\xi \in I_\G(\xi)$, then $\Lambda_{\tilde{S}}^\xi(\phi_\xi)$ is a rational function of $\xi \in D$. This concludes the proof. \end{proof}

\begin{remark}\label{rationalityintoperator} \leavevmode \begin{enumerate}
\item Let $\C[D]$ denote the algebra of regular functions on $D$ and let $\C(D)$ be its quotient field. Then, in the statement above, by a rational function $f$ on $\xi \in D$ we mean an element of $X \otimes_\C \C(D)$, where $X$ is the space of locally constant functions on $\G(\mathcal{O})$ that are left-invariant under $B_\G(\mathcal{O})$.
\item Notice that the exact same argument as in Corollary \ref{rationalitySm} shows the well-known statement that, if $\{ \phi_\xi\}_{\xi \in D}$ is a rational family of functions $\phi_\xi \in I_\G(\xi)$, then the period integral $T_\omega(\phi_\xi)$ is a rational function on $\xi \in D$, for $\omega \in W_{\G}$.
\end{enumerate}

\end{remark}

\subsection{Third reduction: the right hand side of \texorpdfstring{\eqref{csformulaintermediario}}{(14)}}\label{ThirdReductionSection}

As discussed above, $T^*_{\omega^{-1}}\Delta_{\tilde{S},{}^{\omega}\xi}$ is a constant multiple of $\Delta_{\tilde{S},\xi}$; we will now use the integral expression for the Shalika model in Proposition \ref{proponconv} to calculate the constant explicitly for $\omega$ ranging through a set of generators of $W_\H$. Then Corollary \ref{rationalitySm} will be employed to extend the result outside the region of convergence.

Before we go on with the calculations, we need to introduce some notation. For $\alpha \in \Phi_\G^+$, denote by $N^\alpha \subset U_\G$ the one parameter unipotent subgroup associated to it. Explicitly, if $\alpha$ equals  $2 \alpha_2 - \alpha_0 $, resp. $  \alpha_1 - \alpha_2$, we let \[x_{2 \alpha_2 - \alpha_0}: \mathbf{G}_{\rm a} \to N^{2 \alpha_2 - \alpha_0},\, y \mapsto \left(\begin{smallmatrix}1& & & \\ &1&y& \\ & &1& \\ & & &1 \end{smallmatrix}\right), \]
\[x_{\alpha_1 - \alpha_2}: {\rm Res}_{\mathcal{O}_E/\mathcal{O}}\mathbf{G}_{\rm a} \to N^{ \alpha_1 - \alpha_2},\, y \mapsto \left(\begin{smallmatrix}1& y& & \\ &1& & \\ & &1& -\bar{y}\\ & & &1 \end{smallmatrix}\right). \]

We have the following decomposition. 

\begin{lemma}\label{lemfact1}\leavevmode \begin{enumerate}
\item The element
$s_2 x_{2 \alpha_2 - \alpha_0}(y) \in s_2 N^{2\alpha_2-\alpha_0}(\mathcal{O})$ factors as
\[ \left(\begin{smallmatrix}1& & & \\ && 1& \\ &  -1&  -y & \\ & & &1 \end{smallmatrix}\right) = \left(\begin{smallmatrix}1& & & \\ &-y^{-1} &1
& \\ & &-y& \\ & & &1\end{smallmatrix}\right)\left(\begin{smallmatrix}1& & & \\ &1 &
& \\ & y^{-1} &1& \\ & & &1\end{smallmatrix}\right) \in B_{\G}(F)\tilde{S}(F),\]
if $y \ne 0$.
\item The element $s_1 x_{\alpha_1-\alpha_2}(y) \in s_1 N^{\alpha_1-\alpha_2}(\mathcal{O})$ factors as
\begin{align*}  \left(\begin{smallmatrix}& 1 & & \\ 1 & y & & \\ & & & 1 \\ & & 1 & -\overline{y} \end{smallmatrix}\right) = \left(\begin{smallmatrix} -(y\overline{y}+b_y)^{-1} & \overline{y}(y\overline{y}+b_y)^{-1}& & \\ &1 & & \\ & &  (y\overline{y}+b_y)^{-1} &  y(y\overline{y}+b_y)^{-1}\\ & & &-1\end{smallmatrix}\right) \left(\begin{smallmatrix} \overline{y}& -b_y& & \\ 1& y& & \\ & &y&  b_y\\ & &-1& \overline{y}\end{smallmatrix}\right)\in B_{\G}(F)\tilde{S}(F),\end{align*}
where $y = a_y + \delta b_y$, if $y \overline{y} + b_y \ne 0$.
\end{enumerate}
 \end{lemma}

\begin{proof}
It follows from a direct matrix calculation.
\end{proof}

To compute the ratio of $T^*_{\omega^{-1}}\Delta_{\tilde{S},{}^{\omega}\xi}$ by $\Delta_{\tilde{S},\xi}$, it is enough to evaluate it at the element $\mathbf{1}_{\rm Iw}$. Note that, thanks to our normalization \eqref{normalizationofmodel},  $\Delta_{\tilde{S},\xi}(\mathbf{1}_{\rm Iw})=1$. Recall also that the character $\xi$ is of the form  \eqref{shapeofcharacter}:
\[\xi = ( \chi_0 \circ N_{E/F}, \chi_2\chi_0 \circ N_{E/F}, \chi_1  \chi_{E/F}),\]
for some regular unramified character $\chi$ of $T_\H(F)$.
\begin{prop}\label{funct1}
We obtain the following equality:
\begin{align*}
T^*_{s_2^{-1}}\Delta_{\tilde{S},{}^{s_2}\xi}(\mathbf{1}_{{\rm Iw}}) = -\chi^{-1}(a_{\alpha_1-\alpha_2})c_{2\alpha_2-\alpha_0}(\xi),
\end{align*}
where we recall that $c_{\alpha_1-\alpha_2}(\xi)$ is the constant defined in Definition \ref{constantdef}.
\end{prop}

\begin{proof}

By \cite[Theorem 3.4]{CasselmanSpherical}, we have 
\begin{align}\label{intopfors2}
T_{s_2^{-1}}(P_{{}^{s_2}\xi}(\mathbf{1}_{\rm Iw})) = (c_{2\alpha_2 - \alpha_0}({}^{s_2}\xi) -1)P_{{}^{s_2}\xi}(\mathbf{1}_{{\rm Iw}}) + p^{-1}P_{{}^{s_2}\xi}(\mathbf{1}_{{\rm Iw} s_2{\rm Iw}}).
\end{align}
Applying $\Delta_{\tilde{S},\xi}$ to \eqref{intopfors2} and recalling that $\Delta_{\tilde{S},\xi}(\mathbf{1}_{\rm Iw})=1$, we get \begin{align}\label{formulaintopfors2}
T^*_{s_2^{-1}}\Delta_{\tilde{S},{}^{s_2}\xi}(\mathbf{1}_{{\rm Iw}}) = c_{2\alpha_2 - \alpha_0}({}^{s_2}\xi) -1  + p^{-1} \Delta_{\tilde{S},\xi}(\mathbf{1}_{{\rm Iw} s_2{\rm Iw}}). \end{align}
We are thus reduced to compute $\Delta_{\tilde{S},\xi}(\mathbf{1}_{{\rm Iw} s_2{\rm Iw}})$. To do so, we first express every element of ${\rm Iw} s_2{\rm Iw}$ as a product of the form $B_\G(F)\tilde{S}(F)$. We proceed as in \cite[Proposition 8.1]{Sakellaridis}.  Using the Iwahori decomposition of ${\rm Iw} $ we find that ${\rm Iw} s_2{\rm Iw} = B_{\G}(\mathcal{O})\overline{U}_\G(\varpi \mathcal{O})s_2 B_{\G}(\mathcal{O})\overline{U}_\G(\varpi \mathcal{O})$, with $\overline{U}_\G$ the unipotent radical of lower-triangular matrices. Isolating the one parameter subgroup $\overline{N}^\alpha$ of $\overline{U}_\G$ associated to the root $- \alpha$, with $\alpha = 2 \alpha_2 - \alpha_0$, we can write it as
\begin{align*}
    B_{\G}(\mathcal{O})\overline{N}^{\alpha}(\varpi\mathcal{O})\overline{U}_\G^{\widehat{\alpha}}(\varpi\mathcal{O}) s_2  B_{\G}(\mathcal{O})\overline{N}^{\alpha}(\varpi\mathcal{O})\overline{U}_\G^{\widehat{\alpha}}(\varpi\mathcal{O}),
\end{align*}
where $\overline{U}_\G^{\widehat{\alpha}}$ is the subgroup of $\overline{U}_\G$ generated by the one parameter subgroups of all the negative roots apart from $-\alpha$. Notice that $s_2^{-1}\overline{U}_\G^{\widehat{\alpha}}(\varpi \mathcal{O}) s_2 = \overline{U}_\G^{\widehat{\alpha}}(\varpi\mathcal{O})$, while $s_2^{-1}\overline{N}^{\alpha}(\varpi\mathcal{O}) s_2 = N^\alpha(\varpi\mathcal{O})$, hence
\begin{align*}
    {\rm Iw} s_2{\rm Iw}&=B_{\G}(\mathcal{O})s_2 N^{\alpha}(\varpi\mathcal{O})\overline{U}_\G^{\widehat{\alpha}}(\varpi\mathcal{O}) B_{\G}(\mathcal{O})\overline{N}^{\alpha}(\varpi\mathcal{O})\overline{U}_\G^{\widehat{\alpha}}(\varpi\mathcal{O})\\
    &=B_{\G}(\mathcal{O})s_2 N^{\alpha}(\varpi\mathcal{O})B_{\G}^{\widehat{\alpha}}(\mathcal{O}) N^{\alpha}(\mathcal{O})\overline{U}_\G^{\widehat{\alpha}}(\varpi\mathcal{O}) \\
    &=B_{\G}(\mathcal{O})s_2 N^{\alpha}(\mathcal{O})  \overline{U}_\G^{\widehat{\alpha}}(\varpi\mathcal{O}),
\end{align*}
where we wrote $B_{\G}^{\widehat{\alpha}}(\mathcal{O}) = T_\G(\mathcal{O}) U_\G^{\widehat{\alpha}}(\mathcal{O})$, and used that $\overline{U}_\G^{\widehat{\alpha}}(\varpi\mathcal{O}) B_{\G}(\mathcal{O}) \overline{N}^{\alpha}(\varpi\mathcal{O})= B_{\G}(\mathcal{O}) \overline{U}_\G^{\widehat{\alpha}}(\varpi\mathcal{O})$. 
The above equality respects the Haar measures on both sides up to the volume $[{\rm Iw} s_2 {\rm Iw} :{\rm Iw} ]$, which is equal to $p$ (see Example \ref{examplesofa}). Notice that $\overline{U}_\G^{\widehat{\alpha}}(\varpi\mathcal{O}) \subset B_{\G}^{\widehat{\alpha}}(\mathcal{O}) \tilde{S}(\mathcal{O})$. This fact follows from an explicit computation and the fact that $\tilde{S}(\mathcal{O})$ contains $\overline{N}_\G(\mathcal{O})$, where $\overline{N}_\G$ is the unipotent radical opposite to $N_\G$. Using this inclusion, we further get 
\begin{align*}{\rm Iw} s_2 {\rm Iw} =B_{\G}(\mathcal{O})s_2 N^{\alpha}(\mathcal{O})  \overline{U}_\G^{\widehat{\alpha}}(\varpi\mathcal{O})  \subset B_{\G}(\mathcal{O})s_2 N^{\alpha}(\mathcal{O})  B_{\G}^{\widehat{\alpha}}(\mathcal{O}) \tilde{S}(\mathcal{O})  =  B_{\G}(\mathcal{O}) s_2 N^{\alpha}(\mathcal{O}) \tilde{S}(\mathcal{O}).\end{align*}
Now, assume for a moment that $\xi$ satisfies the hypothesis of Proposition \ref{proponconv}. Then, the Shalika distribution is given by the formula \eqref{combination}. The inclusion above, together with the left $B_\G(\mathcal{O})$-invariance and right $\tilde{S}(\mathcal{O})$-invariance of $\xi^{-1}\delta_{B_\G}^{1/2}(b(x))\chi_{\tilde{S}}^{-1}(\tilde{s}(x))$, gives  
\begin{equation}\label{integraltosolve}\Delta_{\tilde{S}, \xi}(\mathbf{1}_{{\rm Iw} s_2{\rm Iw}}) = q \cdot  \int_{s_2\unipu^{2\alpha_2-\alpha_0}(\mathcal{O})}\xi^{-1}\delta_{B_\G}^{1/2}(b(x))\chi_{\tilde{S}}^{-1}(\tilde{s}(x))dx,\end{equation}
where, because of \eqref{normalizationofmodel}, the Haar measure $dx$ is normalized such that $\mathrm{vol}(B_{\G}(\mathcal{O})) = \mathrm{vol}(\tilde{S}(\mathcal{O})) = 1$, while $q$ comes from the volume factor $[{\rm Iw} s_2 {\rm Iw} :{\rm Iw} ]$. 
 Lemma \ref{comp1aux} below calculates \eqref{integraltosolve}; in particular, we have 
 \[\Delta_{\tilde{S}, \xi}(\mathbf{1}_{{\rm Iw} s_2{\rm Iw}}) = q-1-\xi(a_{2\alpha_2-\alpha_0}).\]
Plugging this into \eqref{formulaintopfors2}, we get
\[T^*_{s_2^{-1}}\Delta_{\tilde{S},{}^{s_2}\xi}(\mathbf{1}_{{\rm Iw}}) = c_{2\alpha_2 - \alpha_0}({}^{s_2}\xi) - q^{-1} (1 +  \xi(a_{2\alpha_2-\alpha_0})). \]

\noindent Since 
$c_{2\alpha_2-\alpha_0}({}^{s_2}\xi) = (1-q^{-1}\xi^{-1}(a_{2\alpha_2-\alpha_0}))(1-\xi^{-1}(a_{2\alpha_2-\alpha_0}))^{-1}$,
we can rearrange the expression as
\[ T^*_{s_2^{-1}}\Delta_{\tilde{S},{}^{s_2}\xi}(\mathbf{1}_{{\rm Iw}}) = -\xi(a_{2\alpha_2-\alpha_0})c_{2\alpha_2-\alpha_0}(\xi) = -\chi^{-1}(a_{\alpha_1-\alpha_2})c_{2\alpha_2-\alpha_0}(\xi) , \]
where we have used \eqref{shapeofcharacter} to identify $\xi(a_{2\alpha_2-\alpha_0})=\chi^{-1}(a_{\alpha_1-\alpha_2})$. This proves the result for the set of characters which satisfy the hypothesis of Proposition \ref{proponconv}. Since 
$T_{s_2^{-1}}(P_{{}^{s_2}\xi}(\mathbf{1}_{\rm Iw}))$ is a rational function on $\xi$ (\textit{cf}. Remark \ref{rationalityintoperator}), we can apply Corollary \ref{rationalitySm} to deduce that $\Lambda_{\tilde{S}}\left(T_{s_2^{-1}}(P_{{}^{s_2}\xi}(\mathbf{1}_{\rm Iw}))\right)$ is a rational function on $\xi$ and we can so extend the expression  to every $\xi$.

\end{proof}

\begin{lem}\label{comp1aux}
Suppose that $\xi$ satifies the hypothesis of Proposition \ref{proponconv}; then
\[\int_{s_2 N^{2\alpha_2-\alpha_0}(\mathcal{O})} (\xi^{-1}\delta^{1/2}_{B_\G})(b(x))\chi_{\tilde{S}}^{-1}(\tilde{s}(x))dx = 1-q^{-1}-q^{-1}\xi(a_{2\alpha_2-\alpha_0}).\]
\end{lem}
\begin{proof}
Given $x =s_2 x_{2\alpha_2 - \alpha_0}(y) \in s_2\unipu^{2\alpha_2-\alpha_0}(\mathcal{O})$, Lemma \ref{lemfact1}(1) shows that, if $y \ne 0$,
\[(\xi^{-1}\delta^{1/2}_{B_\G})(b(x))\chi_{\tilde{S}}^{-1}(\tilde{s}(x)) = |y|^{-1}\xi_2(y)\chi_{S}^{-1}(w_{\rm op} \tilde{s}(x) w_{\rm op}^{-1}) =|y|^{2z_2-1}\psi^{-1}(2\delta^2y^{-1}).\]
Since $N^{2\alpha_2-\alpha_0}(\mathcal{O})\simeq \mathcal{O}$, the integral, away from $\{ 0\}$, reduces to
\[\int_{s_2 N^{2\alpha_2-\alpha_0}(\mathcal{O})} (\xi^{-1}\delta^{1/2}_{B_\G})(b(x))\chi_{\tilde{S}}^{-1}(\tilde{s}(x))dx = \int_{\mathcal{O}}|y|^{2z_2-1}\psi^{-1}(2\delta^2y^{-1})dy,\]
where we recall that $dy$ is the additive Haar measure of $\mathcal{O}$. The previous integral is equal to
\begin{equation}\label{aux2csf1}\sum_{j = 0}^{\infty}q^{-j(2z_2 - 1)}\int_{\mathfrak{p}^{j}-\mathfrak{p}^{j+1}}\psi^{-1}(-2\delta^2y^{-1})dy.\end{equation}
Since the conductor of the character $y\mapsto \psi^{-1}(-2\delta^2y^{-1})$ is $\mathcal{O}$, the only two integrals of \eqref{aux2csf1} that contribute to the sum are the ones with $j = 0$ and $j = 1$. They are equal to $1-q^{-1}$ and $-q^{-2}$ respectively. Thus the integral is equal to $1-q^{-1}-q^{-(2z_2 + 1)}=1-q^{-1}-q^{-1} \xi(a_{2 \alpha_2 - \alpha_0})$, where for the latter equality we have conveniently used the matrix $a_{2 \alpha_2 - \alpha_0}$ introduced in Example \ref{examplesofa}.
\end{proof}

For the following discussion, we fix an isomorphism  $\mathcal{O}_E \simeq \mathcal{O}^2$ of $\mathcal{O}$-modules, for which the topology on $\mathcal{O}_E$ corresponds to the product topology on $\mathcal{O}^2$, in a way that the additive Haar measure $dy$ of $E$ gets mapped to $da db$,
where $da$ and $db$ are the additive Haar measures of $F$ giving $\mathcal{O}$ measure $1$.

\begin{prop}\label{funct2}
We obtain the following equality:
\begin{align*}
T^*_{s_1^{-1}}\Delta_{\tilde{S},{}^{s_1}\xi}(\mathbf{1}_{{\rm Iw}}) =  - \chi^{-1}(a_{2\alpha_2 - \alpha_0})c_{\alpha_1-\alpha_2}(\xi)\frac{1+q^{-1}\chi(a_{2\alpha_2 - \alpha_0})}{1+q^{-1}\chi^{-1}(a_{2\alpha_2 - \alpha_0})}.
\end{align*}
\end{prop}
\begin{proof}
We prove the formula following the same strategy as in Proposition \ref{funct1}. In particular, by \cite[Theorem 3.4]{CasselmanSpherical}, we have  
\begin{align}\label{intopfors1}
T_{s_1^{-1}}(P_{{}^{s_1}\xi}(\mathbf{1}_{\rm Iw})) = (c_{\alpha_1 - \alpha_2}({}^{s_1}\xi) -1)P_{{}^{s_1}\xi}(\mathbf{1}_{{\rm Iw}}) + q^{-2}P_{{}^{s_1}\xi}(\mathbf{1}_{{\rm Iw} s_1{\rm Iw}}).
\end{align}
Applying $\Delta_{\tilde{S},\xi}$ to \eqref{intopfors1}, we get \begin{align}\label{formulaintopfors1}
T^*_{s_1^{-1}}\Delta_{\tilde{S},{}^{s_1}\xi}(\mathbf{1}_{{\rm Iw}}) = c_{\alpha_1 - \alpha_2}({}^{s_1}\xi) -1  + q^{-2} \Delta_{\tilde{S},\xi}(\mathbf{1}_{{\rm Iw} s_1{\rm Iw}}). \end{align}
To calculate $\Delta_{\tilde{S},\xi}(\mathbf{1}_{{\rm Iw} s_1{\rm Iw}})$ explicitly we use two facts. Firstly, we express every element of ${\rm Iw} s_1{\rm Iw}$ as a product of the form $B_\G(F)\tilde{S}(F)$. Secondly, we initially assume that $\xi$ satisfies the hypothesis of Proposition \ref{proponconv}
to use formula \eqref{combination} for the Shalika distribution and then extend the result for all $\xi$ by using Corollary \ref{rationalitySm}. Precisely, in the exact same way as in Proposition \ref{funct1}, we get that ${\rm Iw} s_1{\rm Iw} \subset B_{\G}(\mathcal{O}) s_1 N^{\alpha}(\mathcal{O}) \tilde{S}(\mathcal{O})$. Thus, if $\xi$ satisfies the hypothesis of Proposition \ref{proponconv}, we get  
\begin{equation}\label{integraltosolves1}\Delta_{\tilde{S}, \xi}(\mathbf{1}_{{\rm Iw} s_1{\rm Iw}}) = q^2 \cdot  \int_{s_1\unipu^{\alpha_1-\alpha_2}(\mathcal{O})}\xi^{-1}\delta_{B_\G}^{1/2}(b(x))\chi_{\tilde{S}}^{-1}(\tilde{s}(x))dx,\end{equation}
where the Haar measure $dx$ is normalized such that $\mathrm{vol}(B_{\G}(\mathcal{O})) = \mathrm{vol}(\tilde{S}(\mathcal{O})) = 1$, and $q^2$ comes from the volume factor $[{\rm Iw} s_1 {\rm Iw} :{\rm Iw} ]$. By Lemma \ref{comp2aux} below 
\[ \Delta_{\tilde{S}, \xi}(\mathbf{1}_{{\rm Iw} s_1{\rm Iw}}) =- 1  +( 1 -q ) \frac{1-q\chi(a_{2\alpha_2 - \alpha_0})}{1+\chi(a_{2\alpha_2 - \alpha_0})}=- 1  +q( q-1 ) \frac{1-q^{-1}\chi^{-1}(a_{2\alpha_2 - \alpha_0})}{1+\chi^{-1}(a_{2\alpha_2 - \alpha_0})}.\]
Hence, \eqref{formulaintopfors1} becomes 
\[T^*_{s_1^{-1}}\Delta_{\tilde{S},{}^{s_1}\xi}(\mathbf{1}_{{\rm Iw}}) = c_{\alpha_1 - \alpha_2}({}^{s_1}\xi) -1  - q^{-2}  +q^{-1}( q -1 ) \frac{1-q^{-1}\chi^{-1}(a_{2\alpha_2 - \alpha_0})}{1+\chi^{-1}(a_{2\alpha_2 - \alpha_0})}.\]
Since 
\[c_{\alpha_1 - \alpha_2}({}^{s_1}\xi) - 1 - q^{-2} = \frac{1-q^{-2} \chi^2(a_{2\alpha_2- \alpha_0})}{1-\chi^2(a_{2\alpha_2- \alpha_0})}- 1 - q^{-2}= - \frac{1-q^{-2} \chi^{-2}(a_{2\alpha_2- \alpha_0})}{1-\chi^{-2}(a_{2\alpha_2- \alpha_0})}, \]
we have 
\begin{align*}
T^*_{s_1^{-1}}\Delta_{\tilde{S},{}^{s_1}\xi}(\mathbf{1}_{{\rm Iw}}) &= - \chi^{-1}(a_{2\alpha_2 - \alpha_0})\frac{1 - q^{-2}\chi^{-2}(a_{2\alpha_2 - \alpha_0})}{1-\chi^{-2}(a_{2\alpha_2 - \alpha_0})} \cdot \frac{1+q^{-1}\chi(a_{2\alpha_2 - \alpha_0})}{1+q^{-1}\chi^{-1}(a_{2\alpha_2 - \alpha_0})}.
\end{align*}
As $c_{\alpha_1-\alpha_2}(\xi)=(1-q^{-2}\chi^{-2}(a_{2\alpha_2-\alpha_0}))(1-\chi^{-2}(a_{2\alpha_2-\alpha_0}))^{-1}$, we get the desired formula.
\end{proof}

\begin{lem}\label{comp2aux}
Suppose that $\xi$ satisfies the hypothesis of Proposition \ref{proponconv}. We have 
\[\int_{s_1\unipu^{\alpha_1-\alpha_2}(\mathcal{O})}\xi^{-1}\delta^{1/2}_{B_\G}(b(x))\chi_{\tilde{S}}^{-1}(\tilde{s}(x))dx =-\frac{1}{q^2}\left[ 1 + ( q-1 ) \frac{1-q\chi(a_{2\alpha_2 - \alpha_0})}{1+\chi(a_{2\alpha_2 - \alpha_0})} \right].\]
\end{lem}
\begin{proof}
Given $x =s_1 x_{\alpha_1 - \alpha_2}(y) \in s_1 \unipu^{\alpha_1-\alpha_2}(\mathcal{O})$, with $y=a_y +b_y \delta$ such that $y\overline{y} + b_y \ne 0$, Lemma \ref{lemfact1}(2) shows that 
\[(\xi^{-1}\delta^{1/2}_{B_\G})(b(x)) \chi_{\tilde{S}}^{-1}(\tilde{s}(x)) = \left|y\overline{y}+b_y\right|^{-1}(\xi_1\xi_0)(y\overline{y}+b_y).\]
Since $\xi$ is unramified, we can write 
\[(\xi_1\xi_0)(y\overline{y}+b_y) = \left|y\overline{y}+b_y\right|^{2z_1+z_0},\]
for some complex numbers $z_1,z_0$. Notice that, by hypothesis, using the central character condition, $z_0$ and $z_1$ satisfy $2\mathrm{Re}(z_1)+\mathrm{Re}(z_0) -1 > 0$. The integral, away from the elements $y=a_y + b_y \delta$ such that $a_y^2-\delta^2b_y^2 + b_y \ne 0$, reads as

\begin{equation*}\int_{s_1 N^{\alpha_1-\alpha_2}(\mathcal{O})}\xi^{-1}\delta^{1/2}_{B_\G}(b(x))\chi_{\tilde{S}}^{-1}(\tilde{s}(x))dx = \int_{\mathcal{O}_E}\left|a_y^2-\delta^2b_y^2 + b_y \right|^{2z_1+z_0-1}da_{y}db_{y}, 
\end{equation*}
where we have used the isomorphism $\mathcal{O}_E \simeq \mathcal{O}^2$ which sends $dy$ to $da_ydb_y$. We write the integral as the sum of two by splitting the domain into $\varpi \mathcal{O}_E \cup \mathcal{O}_E^\times$ and study each separately. We first start from the integral over $\mathcal{O}_E^\times$. When $y \in \mathcal{O}_E^\times$, $a_y^2-\delta^2b_y^2 \in \mathcal{O}^\times$, hence by changing variable  $b_y':=b_y (a_y^2-\delta^2b_y^2)^{-1}$, the integral becomes 
\begin{align}    \int_{\mathcal{O}_E^\times}\left|a_y^2-\delta^2b_y^2 + b_y \right|^{2z_1+z_0-1}da_{y}db_{y} &=  \int_{\mathcal{O}_E^\times}\left| 1 + b_y' \right|^{2z_1+z_0-1}da_ydb_{y}'\nonumber \\
&=\int_{\{y \in \mathcal{O}_E^\times\,:\, b_y' \not \equiv -1\,\, \text{mod }\mathfrak{p}\}} \!\!\!\! da_ydb_{y}' + {\rm vol}(\mathcal{O}) \int_{\mathfrak{p}} |t|^{2z_1+z_0-1}dt \nonumber\\
    &= \tfrac{q^2-q-1}{q^2} + \sum_{i \geq 1} q^{-i(2z_1+z_0-1)}{\rm vol}(\mathfrak{p}^i - \mathfrak{p}^{i+1})\nonumber \\
    &= \tfrac{q^2-q-1}{q^2} + (1-q^{-1})\sum_{i \geq 1} q^{-i(2z_1+z_0)} \nonumber\\
    &= \tfrac{q^2-q-1}{q^2} + \tfrac{1-q^{-1}}{ q^{2z_1+z_0} -1},\label{firstcontr}
\end{align}
where in the second equality we have made the change of variable $t=1+b_y'$, and in the third we used the convergence of the geometric series as $2\mathrm{Re}(z_1)+\mathrm{Re}(z_0) > 1$. We now evaluate the integral over $\varpi \mathcal{O}_E \simeq \mathfrak{p} \times \mathfrak{p}$. We may factor the ideal $\mathfrak{p} = \bigsqcup_{i\geq 1}\mathfrak{p}^i-\mathfrak{p}^{i+1}$ to get 
\begin{equation*}\int_{\mathfrak{p} \times \mathfrak{p}}\left|a_y^2-\delta^2b_y^2 + b_y \right|^{2z_1+z_0-1}da_{y}db_{y} = \sum_{i,j\geq 1}\int_{\mathfrak{p}^{i} -\mathfrak{p}^{i+1}}\int_{\mathfrak{p}^{j}-\mathfrak{p}^{j+1} }\left|a_y^2-\delta^2b_y^2 + b_y \right|^{2z_1+z_0-1}db_{y}da_{y}.\end{equation*}
By direct computation 
\[\left|a_y^2-\delta^2b_y^2 + b_y \right||_{\mathfrak{p}^{i} -\mathfrak{p}^{i+1}  \times \mathfrak{p}^{j}-\mathfrak{p}^{j+1}} = \begin{cases}q^{-j}&if\;j< 2i\\ q^{-2i}&if\;j> 2i. \end{cases}
   \]
Writing $\mathrm{vol}_{i,j} = \mathrm{vol}(\mathfrak{p}^{i} -\mathfrak{p}^{i+1}) \cdot \mathrm{vol}(\mathfrak{p}^{j}-\mathfrak{p}^{j+1})$, then the integral is equal to 
\begin{align}
    &\sum_{i \geq 1}\sum_{j =1}^{j=2i-1}\mathrm{vol}_{i,j} \cdot  q^{-j(2z_1+z_0-1)} +\sum_{i \geq 1}\sum_{j =2i+1}\mathrm{vol}_{i,j} \cdot  q^{-2i(2z_1+z_0-1)} \label{twostupidsums} \\ &+\sum_{i\geq 1}\int_{\mathfrak{p}^{i}-\mathfrak{p}^{i+1}}\int_{\mathfrak{p}^{2i}-\mathfrak{p}^{2i+1}}\left|a_y^2-\delta^2b_{y}^2+b_y \right|^{2z_1+z_0-1}db_{y}da_{y} \label{stupidmiddlesum}.
\end{align}
First of all, we compute \eqref{twostupidsums} explicitly: since $\mathrm{vol}(\mathfrak{p}^{i} -\mathfrak{p}^{i+1}) = q^{-i}(1-q^{-1})$, \eqref{twostupidsums}  is equal to
\begin{align*}
    (1-q^{-1})^2\left(\sum_{i \geq 1}\sum_{j = 1}^{2i-1}q^{-j(2z_1+z_0)-i}+\sum_{i \geq  1}\sum_{j \geq 2i+1} q^{-2i(2z_1+z_0-\frac{1}{2})-j}\right).
\end{align*}
We analyze each sum separately. By direct computation,
\[\sum_{j = 1}^{2i-1}q^{-j(2z_1+z_0)-i} = q^{-i}\sum_{j = 1}^{2i-1}q^{-j(2z_1+z_0)} = -q^{-i}\frac{1-q^{2z_1+z_0-2i(2z_1+z_0)}}{1-q^{2z_1+z_0}} = \frac{q^{2z_1+z_0-i(2(2z_1+z_0)+1)}-q^{-i}}{1-q^{2z_1+z_0}}.\] 
Therefore 
\begin{align*}\sum_{i \geq  1}\sum_{j = 1}^{2i-1}q^{-j(2z_1+z_0)-i} &= \frac{1}{1-q^{2z_1+z_0}}\left(q^{2z_1+z_0}\sum_{i \geq 1} q^{-i(2(2z_1+z_0)+1)} - \sum_{i \geq  1} q^{-i}\right) = \frac{q^{2z_1+z_0+1}+1}{(q-1)(q^{2(2z_1+z_0)+1}-1)}, \end{align*}
as long as $\mathrm{Re}(2z_1+z_0)>-1/2$. Regarding the second sum, we get 
\[\sum_{j \geq 2i+1} q^{-2i(2z_1+z_0-\frac{1}{2})-j} = q^{-2i(2z_1+z_0-\frac{1}{2})}\sum_{j \geq  2i+1} q^{-j} = q^{-2i(2z_1+z_0-\frac{1}{2})}\frac{q^{-2i}}{q-1} = \frac{q^{-2i(2z_1+z_0+\frac{1}{2})}}{q-1}.\]
Therefore 
\[\sum_{i \geq  1}\sum_{j \geq  2i+1}q^{-2i(2z_1+z_0-\frac{1}{2})-j} = \frac{1}{q-1}\sum_{i \geq  1} q^{-2i(2z_1+z_0+\frac{1}{2})} = \frac{1}{(q-1)(q^{2(2z_1+z_0)+1}-1)}.\]
Hence, the sum of \eqref{twostupidsums} is equal to 
\begin{align}\label{finalfirstpiece} (q-1) \frac{q^{2z_1 + z_0-1}+2 q^{-2}}{q^{2(2z_1+z_0)+1}-1}.\end{align}
We now evaluate \eqref{stupidmiddlesum}. Writing any element $\alpha \in \mathfrak{p}^{\bullet}-\mathfrak{p}^{\bullet+1}$ as $\alpha  = \varpi^{\bullet} \tilde{\alpha}$, with $\tilde{\alpha} \in \mathcal{O}^{\times}$, we can write it as
\begin{equation}\label{sumintp2}\sum_{i\geq 1}q^{-2i(2z_1+z_0-1)} q^{-3i}\int_{\mathcal{O}^{\times}\times\mathcal{O}^{\times}}\left|\tilde{a}_{y}^2-\delta^2p^{2i}\tilde{b}_{y}^2+\tilde{b}_y \right|^{2z_1+z_0-1}d\tilde{a}_{y}d\tilde{b}_{y},\end{equation}
where the volume factor comes from the Jacobian of the change of variables $(\tilde{a}_y,\tilde{b}_y) = (\varpi^{-i}a_y, \varpi^{-2i}b_y)$, which is exactly \[\frac{\mathrm{vol}(\mathfrak{p}^{i}-\mathfrak{p}^{i+1}) \cdot \mathrm{vol}(\mathfrak{p}^{2i} -\mathfrak{p}^{2i+1} )}{{\rm vol}(\mathcal{O}^\times)^2} = q^{-3i}.\]

\noindent By doing the change of variables $(\alpha_y,\beta_y)= (\tilde{a}_y\tilde{b}_y^{-1},\tilde{b}_y^{-1} - \delta^2\varpi^{2i})$, which has trivial Jacobian, then \begin{align}
    \text{\eqref{sumintp2}} &= \sum_{i\geq 1}q^{-i(2(2z_1+z_0)+1)}\int_{\mathcal{O}^{\times}\times\mathcal{O}^{\times}}\left|\alpha_y^2+ \beta_y \right|^{2z_1+z_0-1}d \alpha_y d \beta_{y} \nonumber\\ &= \sum_{i\geq 1}q^{-i(2(2z_1+z_0)+1)}(1-q^{-1})\int_{\mathcal{O}^{\times}}\left|1+u \right|^{2z_1+z_0-1}du \nonumber\\ 
    &= (1-q^{-1}) \tfrac{1}{q^{2(2 z_1 + z_0) +1}-1}(\tfrac{q-2}{q} +  \tfrac{1-q^{-1}}{q^{2 z_1 + z_0}-1}),\label{finalsecondpiece}
\end{align} 
where for the second equality we did the change of variables $u = \beta_y/\alpha_y^2$, for the third we have divided the integral into the one over $1+u \not \in \mathfrak{p}$ and $1+u \in \mathfrak{p}$, and, for the convergence of the geometric series, we have used that ${\rm Re}(2z_1+z_0) > 1$.

\noindent Summing the contributions of \eqref{firstcontr}, \eqref{finalfirstpiece}, and \eqref{finalsecondpiece}, we obtain 
\begin{align*}
 \frac{1}{q^2}\left[ -1 + (q-1) \frac{q^{2z_1 + z_0+1}+1}{q^{2z_1 + z_0}-1} \right] =-\frac{1}{q^2}\left[ 1 + ( q-1 ) \frac{1-q\chi(a_{2\alpha_2 - \alpha_0})}{1+\chi(a_{2\alpha_2 - \alpha_0})} \right],
\end{align*} 
where, by \eqref{shapeofcharacter} and the fact that $\chi_{E/F}(\varpi)=-1$, in the second equality we have used that $-\chi(a_{2\alpha_2 - \alpha_0}) = q^{2z_1 + z_0}.$
\end{proof}

Recall that $g_\chi$ denotes a representative of the Frobenius conjugacy class of $I_\H(\chi)$.

\begin{theorem}\label{CSformulafinallythm}
For every $n \geq 0$, we have 
\[\tilde{\mathcal{S}}(g_{-n}) = \frac{q^{-2n}\prod_{\substack{\alpha\in\Phi_{\mathbf{P}\H}^+}}c_{\alpha}(\xi)}{Q e^{\check{\rho}}(g_{\chi})\prod_{\alpha \in \Phi_{\mathbf{P}\H}^{+,l}} (1+q^{-1}\chi(a_{-\alpha}))} \mathcal{A}\left((-1)^n e^{\check{\rho}+n(\check{\alpha}_1 + \check{\alpha}_2)}\prod_{\alpha \in \Phi^{+,s}_{\Sp_4}}(1 + q^{-1}e^{-\check{\alpha}})\right)(g_{\chi}),\]
where $\mathcal{A}$ denotes the alternator $\mathcal{A}(\cdot) = \sum_{\omega \in W_{{\rm Sp}_4}}(-1)^{\ell(\omega)} \omega(\cdot)$.
\end{theorem}
\begin{proof}
The theorem follows directly from the discussion in \cite[\S 8]{Sakellaridis} and using the Propositions \ref{funct1} and \ref{funct2}. Let us start by recalling the formula obtained in  \eqref{csformulaintermediario}: for every $n \geq 0$, we have \begin{equation}\label{firstformulathm}\tilde{\mathcal{S}}(g_{-n}) = Q^{-1} \sum_{\omega\in W_{\H}} \prod_{\substack{\alpha\in\Phi_{\mathbf{P}\H}^+\\\omega\alpha>0}} c_{\alpha}(\xi) ({}^{\omega}\xi^{-1}\delta^{1/2}_{B_\G})(g_{n})T^*_{\omega^{-1}}\Delta_{\tilde{S},^{\omega}{\xi}}(\mathbf{1}_{\rm Iw}).\end{equation}

Firstly we look at $T^*_{\omega^{-1}}\Delta_{\tilde{S},^{\omega}{\xi}}(\mathbf{1}_{\rm Iw})$. Applying the expressions of Propositions \ref{funct2} and \ref{funct1} one after the other we get:
\begin{align*}
   T^*_{(s_1s_2)^{-1}}\Delta_{\tilde{S},^{s_1s_2}{\xi}}(\mathbf{1}_{\rm Iw})  &= T^*_{s_2^{-1}} ( T^*_{s_1^{-1}}\Delta_{\tilde{S},^{s_1s_2}{\xi}}(\mathbf{1}_{\rm Iw}))  \\  &=  (-1)^{\ell(s_1)}  c_{\alpha_1 + \alpha_2- \alpha_0}(\xi) \chi(a_{-2\alpha_1 + \alpha_0})\frac{1+q^{-1}\chi(a_{2\alpha_1 - \alpha_0})}{1+q^{-1}\chi  (a_{-2\alpha_1 + \alpha_0})} T^*_{s_2^{-1}} ( \Delta_{\tilde{S},^{s_2}{\xi}}(\mathbf{1}_{\rm Iw}))  \\
   &=  (-1)^{\ell(s_1 s_2)}  c_{\alpha_1 + \alpha_2- \alpha_0}(\xi) c_{2 \alpha_2 - \alpha_0}(\xi) \chi(a_{-2\alpha_1 + \alpha_0}) \chi(a_{\alpha_2- \alpha_1})\frac{1+q^{-1}\chi(a_{2\alpha_1 - \alpha_0})}{1+q^{-1}\chi  (a_{-2\alpha_1 + \alpha_0})}  \\ 
   &=  (-1)^{\ell(s_1 s_2)} \prod_{\substack{\alpha\in\Phi_{\mathbf{P}\H}^+\\ s_1 s_2\alpha<0}}c_{\alpha}(\xi) \prod_{\substack{\alpha\in\Phi_{\mathbf{P}\H}^+\\ s_2 s_1 \alpha<0}}\chi(a_{-\alpha})\prod_{\substack{\alpha\in\Phi^{+,l}_{\mathbf{P}\H}\\s_2 s_1\alpha<0}}\frac{1+q^{-1}\chi(a_{\alpha})}{1+q^{-1}\chi(a_{-\alpha})},
   \end{align*}
where recall that $\Phi^{+,l}_{\mathbf{P}\H}$ stands for the set of positive long roots. The identical expression (with $s_1 s_2$ in place of $s_2 s_1$ and viceversa) holds for $s_2s_1$. In general, since $s_1,s_2$ generate $W_\H$, applying recursively Propositions \ref{funct1} and \ref{funct2}, we get the following expression for any $\omega \in W_\H$. If we let $\iota : W_\H \to W_\H$ be the map which swaps $s_1$ and $s_2$, then we have
   \begin{align*} T^*_{\omega^{-1}}\Delta_{\tilde{S},^{\omega}{\xi}}(\mathbf{1}_{\rm Iw}) &=  (-1)^{\ell(\omega)} \prod_{\substack{\alpha\in\Phi_{\mathbf{P}\H}^+\\ \omega \alpha<0}}c_{\alpha}(\xi) \prod_{\substack{\alpha\in\Phi_{\mathbf{P}\H}^+\\ \iota(\omega) \alpha<0}}\chi(a_{-\alpha})\prod_{\substack{\alpha\in\Phi^{+,l}_{\mathbf{P}\H}\\ \iota (\omega)\alpha<0}}\frac{1+q^{-1}\chi(a_{\alpha})}{1+q^{-1}\chi(a_{-\alpha})}.
\end{align*}

\noindent Plugging this into \eqref{firstformulathm}, we get

\begin{align*}  Q^{-1} \delta^{1/2}_{B_\G}(g_{n}) \prod_{\substack{\alpha\in\Phi_{\mathbf{P}\H}^+}}c_{\alpha}(\xi)\sum_{\omega\in W_{\H}}(-1)^{\ell(\omega)}({}^{\omega}\xi^{-1}) (g_{n})\prod_{\substack{\alpha\in\Phi_{\mathbf{P}\H}^+\\ \iota(\omega) \alpha<0}}\chi(a_{-\alpha})\prod_{\substack{\alpha\in\Phi^{+,l}_{\mathbf{P}\H}\\ \iota (\omega)\alpha<0}}\frac{1+q^{-1}\chi(a_{\alpha})}{1+q^{-1}\chi(a_{-\alpha})}.
\end{align*}

\noindent Since $g_{n} = \left( \begin{smallmatrix}
    \varpi^n I & \\ & I\end{smallmatrix}\right)$, we can write  $({}^{\omega}\xi^{-1}) (g_{n}) = (-1)^n \cdot {} ^{\iota(\omega)}\chi(a_{n(2 \alpha_1 - \alpha_0)})$; hence, reordering the sum,

\begin{align*}   Q^{-1} \delta^{1/2}_{B_\G}(g_{n}) \prod_{\substack{\alpha\in\Phi_{\mathbf{P}\H}^+}}c_{\alpha}(\xi)\sum_{\omega\in W_{\H}}(-1)^{\ell(\omega) + n}({}^{\omega}\chi(a_{n(2 \alpha_1 - \alpha_0)}))\prod_{\substack{\alpha\in\Phi_{\mathbf{P}\H}^+\\ \omega \alpha<0}}\chi(a_{-\alpha})\prod_{\substack{\alpha\in\Phi^{+,l}_{\mathbf{P}\H}\\ \omega \alpha<0}}\frac{1+q^{-1}\chi(a_{\alpha})}{1+q^{-1}\chi(a_{-\alpha})}.
\end{align*}

\noindent Proceeding as in \cite{Sakellaridis}, using the expression (75) of \emph{loc.cit.} with $k_\alpha = \chi(a_{-\alpha})$, we rewrite 
\[\prod_{\substack{\alpha\in\Phi_{\mathbf{P}\H}^+\\ \omega \alpha<0}}\chi(a_{-\alpha}) = \chi(a_{-\rho + \omega^{-1} \rho}),\]
where $\rho$ is the half-sum of the positive roots. Similarly, letting $k_\alpha = 1+q^{-1}\chi(a_{-\alpha})$ in (75) of \emph{loc.cit.}, we have
\begin{equation}\label{symmetryExpressionCSproof}\prod_{\substack{\alpha\in\Phi^{+,l}_{\mathbf{P}\H}\\ \omega \alpha<0}}\frac{1+q^{-1}\chi(a_{\alpha})}{1+q^{-1}\chi(a_{-\alpha})} = \frac{\omega^{-1}\left( \prod_{\alpha \in \Phi_{\mathbf{P}\H}^{+,l}} (1+q^{-1}\chi(a_{-\alpha})) \right)}{\prod_{\alpha \in \Phi_{\mathbf{P}\H}^{+,l}}( 1+q^{-1}\chi(a_{-\alpha}))}. \end{equation}
Thus, the formula becomes
\begin{align*}  Q^{-1} \delta^{1/2}_{B_\G}(g_{n}) \prod_{\substack{\alpha\in\Phi_{\mathbf{P}\H}^+}}c_{\alpha}(\xi)\sum_{\omega\in W_{\H}}(-1)^{\ell(\omega) + n} \chi(a_{-\rho + \omega  (\rho + n(2 \alpha_1 - \alpha_0))}) \frac{\omega \left( \prod_{\alpha \in \Phi_{\mathbf{P}\H}^{+,l}} (1+q^{-1}\chi(a_{-\alpha})) \right)}{\prod_{\alpha \in \Phi_{\mathbf{P}\H}^{+,l}} (1+q^{-1}\chi(a_{-\alpha}))},
\end{align*}
where we have used that ${}^{\omega}\chi(a_{n(2 \alpha_1 - \alpha_0)}) = \chi(\omega^{-1} a_{n(2 \alpha_1 - \alpha_0)} \omega) = \chi(a_{ \omega^{-1} (n(2 \alpha_1 - \alpha_0))})$ and we reordered the sum by changing $\omega$ with its inverse. To further simplify the formula, we can express it all in terms of the Frobenius conjugacy class $g_{\chi}$ of $I_\H(\chi)$. Let $i:X^*(T_{\mathbf{PH}}(\C))\simeq X_*(T_{\Sp_4}(\C))$ be the isomorphism induced from the duality between $\mathbf{PH}(\C)$ and $\mathrm{Spin}_5(\C)$ and the exceptional isomorphism $\mathrm{Spin}_5(\C)\simeq \Sp_4(\C)$. This map identifies the long roots $\Phi^{+,l}_{\mathbf{P}\H}$ with the short roots $\Phi^{+,s}_{\Sp_4}$ (see \eqref{Identify_roots_and_coroots}).
Via the isomorphism $i$ and definition \eqref{Satake_parameters_GSp4} of $g_\chi$, we then obtain the equality 
$\chi(a_{\alpha}) = e^{i(\alpha)}(g_{\chi})$.
Rewriting the resulting formula in terms of the roots of $\Sp_4(\C)$, we obtain
\begin{align*}
 \frac{q^{-2n} \prod_{\substack{\alpha\in\Phi_{\mathbf{P}\H}^+}}c_{\alpha}(\xi)}{Q e^{\check{\rho}}(g_{\chi})\prod_{\alpha \in \Phi_{\mathbf{P}\H}^{+,l}} (1+q^{-1}\chi(a_{-\alpha}))} \mathcal{A}\left((-1)^n e^{\check{\rho}+n(\check{\alpha}_1 + \check{\alpha}_2)   }\prod_{\alpha \in \Phi^{+,s}_{\Sp_4}}(1 + q^{-1}e^{-\check{\alpha}})\right)(g_{\chi}),
 \end{align*}
where we used the definition of the alternator $\mathcal{A}(\cdot) = \sum_{\omega \in W_{{\rm Sp}_4}} (-1)^{\ell(\omega)} \omega(\cdot)$ and that $\delta^{1/2}_{B_\G}(g_{n}) = q^{-2n}$.
\end{proof}

\begin{remark}\label{remarkondet}

Notice that, using the relations \[\chi^2(a_{-(2\alpha_1-\alpha_0)}) = \xi(a_{\alpha_1 + \alpha_2 - \alpha_0}),\,\, \chi^2(a_{-(2\alpha_2-\alpha_0)}) = \xi(a_{\alpha_1 - \alpha_2}), \]  
we have  \[\prod_{\substack{\alpha\in\Phi^{+,s}_{\mathbf{P}\H}}}c_{\alpha}(\xi) = \prod_{\alpha \in \Phi_{\mathbf{P}\H}^{+,l}} \frac{1-q^{-2}\chi^2(a_{-\alpha}) }{1-\chi^2(a_{-\alpha})}, \]
therefore $\prod_{\alpha \in \Phi_{\mathbf{P}\H}^{+,l}} (1+q^{-1}\chi(a_{-\alpha}))$ divides the numerator of $\prod_{\substack{\alpha\in\Phi^{+}_{\mathbf{P}\H}}}c_{\alpha}(\xi)$.
\end{remark}

We conclude the section by re-normalizing the Shalika functional as presented in Theorem \ref{CasselmanShalikaformula}.

\begin{corollary}[Theorem \ref{CasselmanShalikaformula}]\label{Coronrightnormalizationinert}
   For a suitable normalization, the spherical Shalika functional satisfies
   \[\mathcal{S}(g_n) = \frac{q^{-2n}}{ 1 + q^{-1}} \mathcal{A}\left( (-1)^n e^{\check{\rho}+n(\check{\alpha}_1 + \check{\alpha}_2)   }\prod_{\alpha \in \Phi^{+,s}_{\Sp_4}}(1 + q^{-1}e^{-\check{\alpha}})\right)(g_{\chi}) (\mathcal{A}(e^{\check{\rho}}) (g_{\chi}))^{-1}.\] 
\end{corollary}

\begin{proof}

Recall that $I_\H(\chi)$ and $I_\G(\xi)$ are irreducible; this implies, by \cite[Proposition 3.5]{CasselmanSpherical} and Remark \ref{remarkondet}, that the following constant 
\[ C :=\frac{\prod_{\substack{\alpha\in\Phi_{\mathbf{P}\H}^+}}c_{\alpha}(\xi)}{\prod_{\alpha \in \Phi^{ +,l}_{\H}} (1-q^{-1}\chi(a_{-\alpha}))} \ne 0.\]
Moreover, as $\mathcal{A}(e^{\check{\rho}}) (g_{\chi}) = e^{\check{\rho}} \prod_{\substack{\alpha\in\Phi_{\mathbf{P}\H}^+}} (1 - e^{-\check{\alpha}})(g_\chi)$ by \cite[Lemma 24.3]{FultonHarris}, a computation as in Remark \ref{remarkondet} shows that $C':=\frac{C \cdot \mathcal{A}(e^{\check{\rho}}) (g_{\chi})}{e^{\check{\rho}}(g_{\chi})} \ne 0$.   We can therefore normalize the Shalika functional by multiplying our formula by $\frac{Q }{C'(1+q^{-1})}$.
\end{proof}

\section{An application to \texorpdfstring{$L$}{L}-functions on \texorpdfstring{$\mathrm{GSp}_4$}{GSp4}}\label{sec_application_to_Lvalues}

Let $\Pi$ be a globally generic cuspidal automorphic representation on $\mathbf{P}\G(\A_F)$, where $F$ is a number field and $\G$ is the group scheme over $F$ associated with a quadratic field extension $E/F$ as in \S \ref{ss:defgroups}. Consider
\begin{equation*}  J(\varphi,  s):=\int_{\H(F)Z_{\H}(\A_F)\setminus \H(\A_F)} E^*_{P_\H}(h,s) \varphi(h) d h,\end{equation*}
where $\varphi$ is a cusp form in $\Pi$ and $E^*_{P_\H}(h,s)$ is the normalized Siegel Eisenstein series for $\H$. Here the Eisenstein series is associated to a standard section $f_s = \otimes_v'f_{v,s}$ in the not normalized induction $ {\rm Ind}_{P_\H(\A_F)}^{\H(\A_F)}(\delta_{P_\H}^{1/3(s+1)})$ such that almost everywhere $f_{v,s} = \zeta_v(s+1)\zeta_v(2s) f_{v,s}^0$, with $f_{v,s}^0(k)=1$ for all $k \in \H(\mathcal{O}_{F_v})$.

In \cite{CauchiGuti}, we
showed that $J(\varphi, s)$ calculates the degree 5 standard {$L$-function of} $\mathrm{PGSp}_4$ twisted by the quadratic Hecke character $\chi_{E/F}$ associated to the field extension $E/F$, by realizing it as the residue of a two variable Rankin--Selberg integral which calculates a product of exterior square $L$-functions of $\Pi$. Since $J(\varphi, s)$ unfolds to the Shalika model of $\varphi$ (\textit{cf}. \cite[Proposition 4.6]{CauchiGuti}), Proposition \ref{Uniquenessinert} let us imply that $J(\varphi,s)$ is Eulerian. In what follows, we propose an alternative and more direct proof of \cite[Theorem 1.2]{CauchiGuti}, by explicitly calculating the local zeta integrals of $J(\varphi, s)$ at finite unramified places, using the Casselman--Shalika formulas of Theorems \ref{thm:sakellaridis} and \ref{CasselmanShalikaformula}. According to \cite[Theorem B]{morimoto}, $\Pi$ has a Shalika model if and only it is the theta lift of (a $\mathbf{PH}^+(\A_F)$-factor of) a globally generic cuspidal automorphic representation $\sigma$ of $\mathbf{PH}(\A_F)$.

Let $v$ be an unramified finite place for $\Pi$ and $\sigma$, $f_s$, and $E/F$ and let $\phi_0$ be the spherical vector in $\Pi_v$; then the local zeta integral of $J(\varphi, s)$ at $v$ is 
\begin{align}\label{localintegralfirstform}
      J_v( \phi_0, f_{v,s}, s) = \int_{\GL_2(F_v)N_\H(F_v)\setminus \H(F_v)} f_{v,s}(h) \mathcal{S}_{\Pi_v}(h) d h.
 \end{align}

Let ${\rm std}: \Sp_4(\C) \to \GL_5(\C)$ denote the irreducible representation of highest weight $\alpha_1+\alpha_2$. It can be realized as the composition of the projection of $\Sp_4(\C) \simeq {\rm Spin}_5(\C)$ to ${\rm SO}_5(\C)$ with the standard representation of the latter. Moreover, let  $\chi_{E_v/F_v}$ be the character on $F_v^\times$ which is trivial when $v$ splits in $E$, while, when $v$ is inert, is \[ \chi_{E_v/F_v}(x) = \begin{cases} -1 & \text{if } x \not \in N_{E_v/F_v}(E_v^\times)  \\
 1 & \text{otherwise.}\end{cases}
\]
Then, to the unramified representation $\sigma_v$ of $\mathbf{PH}(F_v)$,  we define the $L$-factor  
\[ L(s, \sigma_v, {\rm std} \otimes \chi_{E_v/F_v} ) : = \frac{1}{{\rm det}(1 - \chi_{E_v/F_v}(\varpi_v) {\rm std}(g_{\sigma_v})q_v^{-s} )},\]
where $g_{\sigma_v}$ is a representative of the Frobenius conjugacy class of $\sigma_v$ and $\varpi_v$ is a uniformizer of $F_v$ with $|\varpi_v|_v = q_v^{-1}$. 
In Theorems \ref{unrcompsplit} and \ref{unrcompinert} below, we calculate $J_v(\phi_0,f_{v,s},s)$, showing that it equals to  $L(s, \sigma_v, {\rm std} \otimes \chi_{E_v/F_v})$ when  $\sigma_v$ and $\Pi_v$ are related by the local theta correspondence.

\subsection{The unramified local zeta integrals }\label{ss:unramifiedintegral}
From now on to the rest of the manuscript, we drop the pedix ${\bullet}_v$ and let $F$ be a non-archimedean local field of characteristic zero and $E$ be either $F \times F$ or the unique unramified quadratic field extension of $F$. We suppose that $\Pi$ is an unramified representation of $\G(F)$ with trivial central character. Using the Iwasawa decomposition $\H(F) = P_\H(F) \H(\mathcal{O})$ and the $\H(\mathcal{O})$-invariance of $f_{s}$ and $\phi_0$, the local integral \eqref{localintegralfirstform} is \[ J(\phi_0, f_{s}, s) = \int_{\GL_2(F)N_\H(F)\setminus P_\H(F)} \delta_{P_\H}^{-1}(g) f_{s}(g) \mathcal{S}_\Pi(g) d g.\]
As $P_\H(F)=M_\H(F) N_\H(F)$ and the multiplier identifies $ \GL_2(F)\backslash M_\H(F) \simeq F^\times$, we have \begin{align*}
J(\phi_0, f_{s}, s) &= \int_{F^\times} \delta_{P_\H}^{-1}\left(\left(\begin{smallmatrix} I &  \\  & \mu I \end{smallmatrix} \right)\right) f_{s}\left(\left(\begin{smallmatrix} I &  \\  & \mu I \end{smallmatrix} \right)\right) \mathcal{S}_\Pi\left(\left(\begin{smallmatrix} I &  \\  & \mu I \end{smallmatrix} \right)\right) d^\times \mu  \\ 
&=f_{s}(1) \int_{F^\times} | \mu |^{2-s} \mathcal{S}_\Pi\left(\left(\begin{smallmatrix} I &  \\  & \mu I \end{smallmatrix} \right)\right) d^\times \mu.
\end{align*}

\begin{lemma}\label{stupidlemmaonmodel}
If $\mathcal{S}_\Pi\left(\left(\begin{smallmatrix} I &  \\  & \mu I \end{smallmatrix} \right)\right) \ne 0$, then  $| \mu | \geq 1$.
\end{lemma}

\begin{proof}
As $\phi_0$ is spherical, we have for $u \in N_\G(\mathcal{O}) \smallsetminus N_\H(\mathcal{O})$ \begin{align*}
\mathcal{S}_\Pi\left(\left(\begin{smallmatrix} I &  \\  & \mu I \end{smallmatrix} \right)\right) &= \Lambda_{\mathcal{S}}\left(\left(\left(\begin{smallmatrix} I &  \\  & \mu I \end{smallmatrix} \right) u \right)  \cdot \phi_0\right) \\ &= \Lambda_{\mathcal{S}}\left(\left(\left(\begin{smallmatrix} I &  \\  & \mu I \end{smallmatrix} \right) u \left(\begin{smallmatrix} I &  \\  & \mu^{-1} I \end{smallmatrix} \right) \left(\begin{smallmatrix} I &  \\  & \mu I \end{smallmatrix} \right)\right)  \cdot \phi_0\right) \\ &= \chi_{S}\left(\left(\begin{smallmatrix} I &  \\  & \mu I \end{smallmatrix} \right) u \left(\begin{smallmatrix} I &  \\  & \mu^{-1} I \end{smallmatrix} \right)\right) \mathcal{S}_\Pi\left(\left(\begin{smallmatrix} I &  \\  & \mu I \end{smallmatrix} \right)\right).
\end{align*}
Hence, since by hypothesis $\mathcal{S}_\Pi\left(\left(\begin{smallmatrix} I &  \\  & \mu I \end{smallmatrix} \right)\right) \ne 0$, then \[ \chi_{S}\left(\left(\begin{smallmatrix} I &  \\  & \mu I \end{smallmatrix} \right) u \left(\begin{smallmatrix} I &  \\  & \mu^{-1} I \end{smallmatrix} \right)\right) = 1.\]
This implies that $ | \mu | \geq 1$.
\end{proof}

By Lemma \ref{stupidlemmaonmodel}, we can further write the integral as 
\begin{align*}
  J(\phi_0, f_{s}, s) &= f_{s}(1) \sum_{n \geq 0} q^{n(2-s)} \mathcal{S}_\Pi\left(\left(\begin{smallmatrix} I &  \\  & \varpi^{-n} I \end{smallmatrix} \right)\right) \\
    &= f_{s}(1) \sum_{n \geq 0} q^{n(2-s)} \mathcal{S}_\Pi\left(\left(\begin{smallmatrix} \varpi^{n} I &  \\  &  I \end{smallmatrix} \right) \right), \end{align*}
where in the latter we have used that the central character of $\Pi$ is trivial.

\subsection{The case of $E=F \times F$}\label{SectionSplitZeta}

Recall that $\Pi$ can be identified with an irreducible unramified principal series  $I_{\GL_4}(\chi)$ of $\mathrm{PGL}_4(F) \simeq \mathbf{PG}(F)$. From \S \ref{subsec:onsplitsak}, $\Pi$ has a Shalika model if a representative of its Frobenius conjugacy class is in ${\rm Sp}_4(\C)$, implying that $\Pi$ is the functorial lift of an unramified representation $\sigma$ on $\mathbf{P}\H(F)$. We denote by $g_\chi$ a representative of the Frobenius conjugacy class of $\sigma$ (and $\Pi$). 

For any $n \geq 0$, we denote by $\rho_{n,n}$ the irreducible representation of highest weight $n(\alpha_1 + \alpha_2)$. To avoid any confusion, in the following computations we simply denote ${\rm std}: \Sp_4(\C) \to \GL_5(\C)$ by $\rho_{1,1}$. Following \cite[\S \;4]{Bumprankin}, we have that 
\[ L(s, \sigma, {\rm std}) = \sum_{k = 0}^{\infty}\mathrm{tr}\left(g_\chi |\mathrm{Sym}^k\rho_{1,1}\right)q^{-ks},\]
where $\mathrm{tr}\left(g_\chi |\mathrm{Sym}^k\rho_{1,1}\right)$ is the character associated to the representation $\mathrm{Sym}^k \rho_{1,1}$ of $\Sp_4(\C)$ evaluated at $g_\chi$. In order to explicit the formula, we have to decompose into irreducible factors the representation $\mathrm{Sym}^k \rho_{1,1}$ for every $k\geq 0$.

\begin{lemma}\label{lem:decofsymrepsI}
We have that
\[\mathrm{Sym}^{k}(\rho_{1,1}) = \bigoplus_{i = 0}^{\lfloor k/2 \rfloor} \rho_{k - 2i, k - 2i},\]
where $\rho_{k - 2i, k - 2i}$ is the irreducible representation of highest weight $(k-2i)(\alpha_1 + \alpha_2)$.
\end{lemma}

\begin{proof}
The symplectic form defining $\Sp_4$ defines a surjection 
\[{\rm Sym}^2 \rho_{1,1} \to \C,
\]
which induces a surjection
$\mathrm{Sym}^{k}(\rho_{1,1}) \to \mathrm{Sym}^{k-2}(\rho_{1,1})$ for all $k \geq 2$, with kernel the representation $\rho_{k,k}$ of highest weight $k \alpha_1 + k \alpha_2$ (\textit{cf}. \cite[Exercise 16.11]{FultonHarris}). Therefore 
\[\mathrm{Sym}^{k}(\rho_{1,1}) = \rho_{k,k} \oplus \mathrm{Sym}^{k-2}(\rho_{1,1}).\]
The result then follows from applying this formula recursively.
\end{proof}

\noindent Lemma \ref{lem:decofsymrepsI} and the Weyl's character formula (\textit{cf}. \cite[Theorem 24.2]{FultonHarris}) imply that
\begin{equation}\label{eq:standardLfunctionsplit}
   L(s, \sigma, {\rm std}) =  \sum_{k = 0}^{\infty} \sum_{i = 0}^{\lfloor k/2 \rfloor}\mathrm{tr}\left(g_\chi |\rho_{k-2i, k - 2i}\right)q^{-ks}= \sum_{k = 0}^{\infty}\sum_{i = 0}^{\lfloor   k/2 \rfloor} \frac{\mathcal{A}\left(e^{\check{\rho}+(k-2i,k-2i)}\right)}{\mathcal{A}\left(e^{\check{\rho}}\right)}(g_\chi)q^{-ks},
\end{equation} 
where we denoted $(k-2i,k-2i) = (k-2i)(\check{\alpha}_1 + \check{\alpha}_2)$. This is crucially used to prove the following.

\begin{theorem} \label{unrcompsplit}
We have \[J(\phi_0, f_{s}, s) = L(s, \sigma, {\rm std}). \]
\end{theorem}

\begin{proof}
Substituting the formula of Theorem \ref{thm:sakellaridis} in the integral, we have
\begin{align*}
   J(\phi_0, f_{s}, s) 
    &=  f_{s}(1)(1+ q^{-1})^{-1} \sum_{n \geq 0} q^{-ns}\mathcal{A}\Big( e^{\check{\rho}+n(\check{\alpha}_1+\check{\alpha}_2)} \cdot \prod_{\alpha \in \Phi_{\Sp_4}^{+, s }} (1 - q^{-1} e^{-\check{\alpha}}) \Big)(g_{\chi}) ({\mathcal{A}\left(e^{\check{\rho}}\right)}(g_{\chi}))^{-1}. \end{align*}
We now examine the alternator sum appearing above. First, recall that $\Phi_{\Sp_4}^{+, s }$ consists of the roots $\alpha_1 - \alpha_2$ and $\alpha_1 + \alpha_2$. Then, the term     $\mathcal{A}( e^{\check{\rho}+n(\check{\alpha}_1+\check{\alpha}_2)} \cdot \prod_{\alpha \in \Phi_{\Sp_4}^{+, s }} (1 - q^{-1} e^{-\check{\alpha}}) )$  equals
\begin{align}\label{firstmess}
\mathcal{A}(e^{\check{\rho}+ n\check{\alpha}_1+ n \check{\alpha}_2} - q^{-1}e^{\check{\rho}+ (n-1)\check{\alpha}_1+(n-1)\check{\alpha}_2} - q^{-1}e^{\check{\rho}+(n-1)\check{\alpha}_1+(n+1)\check{\alpha}_2} + q^{-2}e^{\check{\rho}+(n-2)\check{\alpha}_1+n\check{\alpha}_2} ).
\end{align}
Now, let $s_1$ be the Weyl element associated to $\alpha_1-\alpha_2$; it acts on the alternator as multiplication by $-1$, i.e. \begin{equation}\label{alternator}\mathcal{A}(e^{s_1 \cdot \check{\mu}}) = - \mathcal{A}(e^{\check{\mu}}).  \end{equation}
This let us rearrange \eqref{firstmess} as 
\begin{align}\label{secondmess}
(1+ q^{-1})\mathcal{A}(e^{\check{\rho}+n\check{\alpha}_1+n\check{\alpha}_2}- q^{-1}e^{\check{\rho}+(n-1)\check{\alpha}_1+(n-1)\check{\alpha}_2} ).
\end{align}
Substituting \eqref{secondmess} in the integral yields 
\[  J(\phi_0, f_{s}, s) =f_{s}(1)\sum_{n \geq 0} q^{-ns}\frac{\mathcal{A}(e^{\check{\rho}+n\check{\alpha}_1+n\check{\alpha}_2}- q^{-1}e^{\check{\rho}+(n-1)\check{\alpha}_1+(n-1)\check{\alpha}_2} )} {\mathcal{A}\left(e^{\check{\rho}}\right)}(g_{\chi}).\]

\noindent For every $n \geq -1$, denote the term $a_n : =\frac{\mathcal{A}(e^{\check{\rho}+n\check{\alpha}_1+n\check{\alpha}_2})} {\mathcal{A}\left(e^{\check{\rho}}\right)}(g_{\chi})$; notice that for $n=-1$ we have that $a_{-1}= \frac{\mathcal{A}(e^{\check{\rho}-(\check{\alpha}_1 + \check{\alpha}_2)})} {\mathcal{A}\left(e^{\check{\rho}}\right)}(g_{\chi})$. Let $s_2 $ be the Weyl element associated to the root $2\alpha_2$, then $s:=s_2s_1s_2$ acts on the alternator $\mathcal{A}(e^{\check{\rho}-(\check{\alpha}_1 + \check{\alpha}_2)})$ as in \eqref{alternator}. Since $\check{\rho}-(\check{\alpha}_1 + \check{\alpha}_2)$ is invariant under $s$ we conclude that $a_{-1} = 0$.

\noindent Recall that $f_{s}(1) = \zeta_F(s+1) \zeta_F(2s)$; then, using the Cauchy product formula, we have
\begin{align*}
 J(\phi_0, f_{s}, s) &= \zeta_F(s+1)\sum_{n\geq 0}q^{-2ns} \cdot \sum_{n \geq 0} q^{-ns}(a_n - q^{-1} a_{n-1})  \\
  &= \zeta_F(s+1)\sum_{n \geq 0} \Big(\sum_{k=0}^{\lfloor n/2 \rfloor} (a_{n - 2k} - q^{-1} a_{n -1 - 2k} )\Big) q^{-ns} \\
  &= \zeta_F(s+1)\sum_{n \geq 0} \Big(\sum_{k=0}^{\lfloor n/2 \rfloor} a_{n - 2k} - q^{-1}\sum_{k=0}^{\lfloor (n-1)/2 \rfloor} a_{n -1 - 2k} \Big) q^{-ns},
\end{align*}
where for the latter equality we have used that $a_{-1}=0$.
This expression is terribly close to the one of \eqref{eq:standardLfunctionsplit}. Indeed, a simple manipulation of the series above gives that
\begin{align*}
    J(\phi_0, f_{s}, s) &= \zeta_F(s+1)\left( \sum_{n \geq 0} q^{-ns}\sum_{k=0}^{\lfloor n/2 \rfloor} a_{n - 2k}  - q^{-1}\sum_{n \geq 0} q^{-ns }\sum_{k=0}^{\lfloor (n-1)/2 \rfloor} a_{n -1 - 2k} \right) \\ 
    &= \zeta_F(s+1)\left(\sum_{n \geq 0} q^{-ns}\sum_{k=0}^{\lfloor n/2 \rfloor} a_{n - 2k}  -q^{-1-s}\sum_{n \geq 0} q^{-(n-1)s}\sum_{k=0}^{\lfloor (n-1)/2 \rfloor} a_{n -1 - 2k}\right) \\
    &= \zeta_F(s+1) (1 - q^{-1-s}) \sum_{n \geq 0} q^{-ns}(\sum_{k=0}^{\lfloor n/2 \rfloor} a_{n - 2k}) \\
    &=  \sum_{n \geq  0} \sum_{k = 0}^{\lfloor  n/2 \rfloor} \frac{\mathcal{A}\left(e^{\check{\rho}+(n-2k,n-2k)}\right)}{\mathcal{A}\left(e^{\check{\rho}}\right)}(g_\chi)q^{-ns} \\ 
    &= L(s, \sigma, {\rm std}),
\end{align*}
where the last equality follows from \eqref{eq:standardLfunctionsplit}.
\end{proof}

\subsection{The case of 
the unramified quadratic extension $E/F$}\label{SectionInertZeta}
Let $\Pi$ be an irreducible generic unramified representation of $\mathbf{P}\G(F)$, which is the local theta lift of the $\psi$-generic component of an unramified principal series $\sigma = I_\H(\chi)$. We let $g_\chi$ denote a representative of the Frobenius conjugacy class of $\sigma$. Our Casselman--Shalika formula let us relate $J(\phi_0, f_s,s)$ to the Euler factor $L(s, \sigma, {\rm std} \otimes \chi_{E/F})$. 

\begin{lemma}\label{intermediatecalcinert}
We have \[L(s, \sigma, {\rm std} \otimes \chi_{E/F}) = \sum_{n = 0}^{\infty} (-1)^n \cdot \sum_{i = 0}^{\lfloor   n/2 \rfloor} \frac{\mathcal{A}\left(e^{\check{\rho}+(n-2i)(\check{\alpha}_1 + \check{\alpha}_2)}\right)}{\mathcal{A}\left(e^{\check{\rho}}\right)}(g_\chi)q^{-ns} . \]
\end{lemma}

\begin{proof}

The proof of this is similar to the one of \eqref{eq:standardLfunctionsplit}. In particular, it follows from the formula
\[L(s, \sigma, {\rm std} \otimes \chi_{E/F}) = \sum_{n = 0}^{\infty} (-1)^n \cdot \mathrm{tr}\left(g_\chi |\mathrm{Sym}^n\rho_{1,1}\right) q^{-ns}, \]
as $(\chi_{E/F}(\varpi))^n=(-1)^n$, and the decomposition of Lemma \ref{lem:decofsymrepsI}.
\end{proof}

\begin{theorem}\label{unrcompinert}
    We have \[J(\phi_0, f_{s}, s) = L(s, \sigma, {\rm std} \otimes \chi_{E/F}). \]
\end{theorem}

\begin{proof}
Using the formula for $\mathcal{S}_\Pi\left(\left(\begin{smallmatrix} \varpi^{n} I &  \\  &  I \end{smallmatrix} \right)\right)$ of Theorem \ref{CasselmanShalikaformula}, we get \begin{align*}
    J(\phi_0, f_{s}, s)
    &=  f_{s}(1)(1+ q^{-1})^{-1} \sum_{n \geq 0} q^{-ns}\mathcal{A}\Big( (-1)^n e^{\check{\rho}+n(\check{\alpha}_1+\check{\alpha}_2)} \cdot \prod_{\alpha \in \Phi_{\Sp_4}^{+, s }} (1 + q^{-1} e^{-\check{\alpha}}) \Big)(g_{\chi}) ({\mathcal{A}\left(e^{\check{\rho}}\right)}(g_{\chi}))^{-1}. \end{align*}
The term    $\mathcal{A}( (-1)^n e^{\check{\rho}+n(\check{\alpha}_1+\check{\alpha}_2)} \cdot \prod_{\alpha \in \Phi_{\Sp_4}^{+, s }} (1 + q^{-1} e^{-\check{\alpha}}) )$  equals
\begin{align*} 
(-1)^n\mathcal{A}( e^{\check{\rho}+ n\check{\alpha}_1+ n \check{\alpha}_2} + q^{-1}e^{\check{\rho}+ (n-1)\check{\alpha}_1+(n-1)\check{\alpha}_2} + q^{-1}e^{\check{\rho}+(n-1)\check{\alpha}_1+(n+1)\check{\alpha}_2} + q^{-2}e^{\check{\rho}+(n-2)\check{\alpha}_1+n\check{\alpha}_2} ).
\end{align*}
Proceeding as in Theorem \ref{unrcompsplit}, we write it as  \begin{align*} 
(1+ q^{-1}) (-1)^n \mathcal{A}(e^{\check{\rho}+n\check{\alpha}_1+n\check{\alpha}_2} + q^{-1}e^{\check{\rho}+(n-1)\check{\alpha}_1+(n-1)\check{\alpha}_2} ).
\end{align*}
For every $n \geq -1$, denote the term $a_n' : =\frac{\mathcal{A}((-1)^n e^{\check{\rho}+n\check{\alpha}_1+n\check{\alpha}_2})} {\mathcal{A}\left(e^{\check{\rho}}\right)}(g_{\chi})$; then the integral becomes

\begin{align*}
   J(\phi_0, f_{s}, s) &=  f_{s}(1)\sum_{n \geq 0} q^{-ns}(a_n' - q^{-1} a_{n-1}') . \end{align*}

Proceeding exactly as in Theorem \ref{unrcompsplit},  we have that 
\begin{align*}
  J(\phi_0, f_{s}, s)  &=  \sum_{n \geq 0} q^{-ns}\left(\sum_{k=0}^{\lfloor n/2 \rfloor} a_{n - 2k}' \right) \\
    &=  \sum_{n \geq  0} (-1)^n \cdot \sum_{k = 0}^{\lfloor  n/2 \rfloor} \frac{\mathcal{A}\left(e^{\check{\rho}+(n-2k)(\check{\alpha}_1+\check{\alpha}_2)} \right)}{\mathcal{A}\left(e^{\check{\rho}}\right)}(g_\chi)q^{-ns} \\ 
    &= L(s, \sigma, {\rm std} \otimes \chi_{E/F}),
\end{align*}
where the latter equality follows from Lemma \ref{intermediatecalcinert}.
\end{proof}

\bibliographystyle{acm}

\bibliography{ShalikaV3}

\end{document}